\documentclass[11pt]{article}
\usepackage{amsmath}
\usepackage{amsthm}
\usepackage[xetex]{graphicx}
\usepackage{epstopdf}
\usepackage{amssymb,enumerate}
\usepackage{graphicx}
\usepackage{latexsym}
\newtheorem{thm}{Theorem}[section]
\newtheorem{cor}[thm]{Corollary}
\newtheorem{lem}[thm]{Lemma}
\newtheorem{prop}[thm]{Proposition}

\newtheorem{defn}[thm]{Definition}

\newtheorem{rem}[thm]{Remark}

\makeatletter

\newcommand{\Rmnum}[1]
{expandafter\@slowromancap\romannumeral#1@}
\makeatother
\numberwithin{equation}{section}%¹«Ê½ÒÔ¡°½Ú¡±ÎªÅÅÐòµ¥Î»
\numberwithin{figure}{section}%²åÍ¼±êÌâÐòºÅÒÔ"½Ú"ÎªÅÅÐòµ¥Î»
\allowdisplaybreaks[4]%¶àÐÐ¹«Ê½»»Ò³£¬ÓÃ¡°\\*¡± ×÷Îª»»ÐÐÃüÁî£»¶Ô¶àÐÐ¹«Ê½»·¾³split ºÍ¿é»·¾³aligned,gathered,alignedat ²»Æð×÷ÓÃ
\begin{document}
	\title{Long-term behavior of
		nonlocal reaction-diffusion equation under small random perturbations}
	\author{Xiuling Gui,\,\,Jin Yang,\,\,Chunfeng Wang,\,\,Jing Hou,\,\,Ji Shu\thanks{Corresponding author. \textit{E-mail addresses}: 2189157394@qq.com(X.Gui), 2252372813@qq.com(J.Yang), 18781613970@qq.com(C.Wang), 2542866897@qq.com(J.Hou), shuji@sicnu.edu.cn(J.Shu).}\small\\
		\textit{School of Mathematical Sciences, Sichuan Normal University,Chengdu 610066, PR China}}
	\date{}
	%\dedicatory{}%
	%\commby{}%
	% ----------------------------------------------------------------
	
	\maketitle
	
	{\begin{center}
			\begin{minipage}{6.0in}
				
				{\textbf{Abstract:} In this paper, we investigate the nonlocal reaction-diffusion equation driven by stationary noise, which is a regular approximation to white noise and satisfies certain properties. We show the existence of random attractor for the equation. When stochastic nonlocal reaction-diffusion equation driven by additive and multiplicative noise, we demonstrate that the solution converges to the corresponding deterministic equation and establish the upper semicontinuity of the attractors as the perturbation parameter $\delta$ and $\epsilon$ both approaches zero. }\\
				
				%{\textbf{2000 Mathematics Subject Classification:}\,37L55,\,60H15}
			\end{minipage}
	\end{center}}
	\section{Introduction}
	\indent
	\par
	In this paper, we study the limiting behavior of random attractors of random dynamical systems as the perturbation parameters approach zero. In particular, let $\mathcal{O} \subset \mathbb{R}^N$ be a bounded open set of class $C^k$ with $k \geq 2$, and consider the following nonautonomous nonlocal reaction-diffusion equation defined on $\mathcal{O}$:
		\begin{equation}
			\begin{cases}
				\label{eq:1.1}
				\frac{\partial u_{\delta,\epsilon}}{\partial t}-a(l(u_{\delta,\epsilon}))\Delta u_{\delta,\epsilon}=f(u_{\delta,\epsilon})+h(t)+\epsilon g(t,u_{\delta,\epsilon})\zeta_{_\delta}(\theta_{_t}\omega), & \text{in } \mathcal{O} \times (\tau, \infty) \\
				u = 0, & \text{on } \partial\mathcal{O} \times (\tau, \infty) \\
				\ u_{\delta, \epsilon}(x,\tau)=u_{\delta,\epsilon}(\tau)=u_\tau, & \text{in } \mathcal{O}
			\end{cases}
		\end{equation}
		where $\tau\in R$, $\epsilon$ is a small positive parameter, $h\in L_{loc}^2(R;H^{-1}(\mathcal{O}))$, $\zeta_{\delta}(\theta_{t}\omega)$ is a stationary noise defined in a probability space for each $\delta>0$, $f,a$ are continuous function satisfying some conditions and $g$ is a continuous mapping  which satisfies some assumptions (see Section 2.2). 
	\par 
	For the deterministic case (i.e., $\epsilon = 0$) of equation (1.1):
	\begin{equation}
		\label{eq:1.2}
		\frac{\partial u}{\partial t}-a(l(u))\Delta u=f(u)+h(t),
	\end{equation}
	it has already been studied on its solutions and the existence of pullback $\mathcal{D}$-attractors in \cite{Caraballo THerrera-Cobos M Marín-Rubio P}. There have been numerous publications concerning the existence and upper semicontinuity of global attractors and random attractors (see, for example \cite{A.V. Babin and M.I.,A. Andami Ovono. Asymptotic behaviou,B. WangUpper semicontinuity of random attractors,G. R. Sell and Y. You,J.C.RobinsonS,J.C. Robinson,J.K. Hale,J.K. Hale X. Lin and G,J.K. Hale and G. Raugel Upper semicontinuity,J.BallGlobalattractorsfor,R.TemamI,T.CaraballoandJ.A.LOn the upper,Y. LvandW.WangD}) and the references therein. More works on random attractors can be found in \cite{F. Flandoli and B. Schmalfuss Random attractors,H. Crauel A. Debussche and F. Flandoli,H. Crauel and F. Flandoli,M. J. Garrido-Atienza and B. SchmalfussErgodicity of,P. W. Bates K. N. Lu and B. Wang,T. Caraballo M. J. Garrido-Atienza B. Schmalfuss and J. Valero,T. Caraballo M. J. Garrido-Atienza B. Schmalfuss and J. Valero Asymptotic,T. Caraballo M. J. Garrido-Atienza and T. Taniguchi,W.-J. Beyn B. Gess P. Lescot and M. R¨ockner} for the autonomous stochastic equations, and in \cite{B. Gess,B. GessRandom attractors,B. Wang Sufficient and,J. Q. Duan and B. Schmalfuss,T. Caraballo J. A. Langa V. S. Melnik and J. Valero} for the non-autonomous stochastic systems. Additionally, Kloeden and Stoiner have obtained some results on the relationship between the attractors for an autonomous ordinary differential equation and small non-autonomous perturbations in \cite{P.E.Kloeden and D.J. Stoiner}. In this paper, we will examine the limiting behavior of random attractors for the stochastic nonautonomous nonlocal reaction-diffusion problem (1.1) defined on $\mathcal{O}$ when $\delta\rightarrow0^+$ and $\epsilon\rightarrow0^+$. As far as we know, there are no results on the dynamical behavior of problem (1.1).\par 
	\par 
	We consider a class of stationary processes that can be regarded as regular approximations of the white noise (which is known as formal derivative of the Brownian Motion). To describe such noise, we introduce a probability space $ \left(\Omega, \mathcal{F}, \Bbb P\right)$, defined as the classical Wiener probability space on the Brownian motion W(t, $\omega$), where
	\begin{equation}
		\Omega=C_0(\Bbb R,\Bbb R ) :=\left\{\omega\in C(\Bbb R,\Bbb R) : \omega(0)=0\right\}
		\nonumber
	\end{equation}
	with the open compact topology, $\mathcal{F}$ is its Borel $\sigma $-algebra, and $ \Bbb P$ is the Wiener measure on $(\Omega, \mathcal{F}) $. The Brownian motion has the form $W(t, \omega) =\omega(t)$. In what follows, we will consider the Wiener shift $\left\lbrace \mathcal{\theta}_{t}\right\rbrace _{t\in\Bbb R} $ defined on the probability space $ \left(\Omega,\mathcal{F},\Bbb P\right) $ by
	
	\begin{equation}
			\label{eq:1.3}
		\theta_{t}\omega(\cdot)=\omega(t+\cdot)-\omega(t) \quad \forall\omega\in\Omega, t\in\Bbb R.
	\end{equation}
	As is well-known, $\Bbb P $ is an ergodic invariant measure for $ \left\lbrace \mathcal{\theta}_{t}\right\rbrace _{t\in\Bbb R} $, and $ \left(\Omega,\mathcal{F},\left\lbrace \mathcal{F}_{t}\right\rbrace _{t\ge0},\Bbb P\right) $ forms a metric dynamical system (see \cite{Arnold}). There exists a $ \theta_t$ -invariant subset of full measure $\Omega_1 $ (see, e.g., \cite{Duan}) such that
	\begin{equation}
			\label{eq:1.4}
		\lim_{t\rightarrow\pm\infty}\frac{\omega(t)}{t}=0,\quad \forall \omega\in\Omega_1.
	\end{equation}
	Let \par 
	\begin{equation}
			\label{eq:1.5}
		C_{\omega}=\sup_{s\in \Bbb Q}\frac{ |\omega(s)|}{ |s|+1},
	\end{equation}
	where $ \Bbb Q $ is the set of rational numbers. By the pathwise continuity of the Wiener process,
	we find that $ C_{\omega} : \Omega_1 \rightarrow \Bbb R^+ $ is a measurable function, and
	\begin{equation}
			\label{eq:1.6}
		|\omega(s)|\le C_{\omega}(|s|+1)
	\end{equation}
	for all $ s\in \Bbb R $. Recall that $ \theta_t\omega(s)=\omega(s+t)-\omega(t) $, it then follows that 
	\begin{equation}
			\label{eq:1.7}
		C_{\theta_t\omega}\leq2C_{\omega}(|t|+1).
	\end{equation}
	From now on, we consider the probability space $(\Omega_1, \mathcal{F}_1, \Bbb P)$, where $\mathcal{F}_1$ is the trace algebra of $\Omega_1$. For simplicity, this space is still denoted as $\left(\Omega, \mathcal{F}, \Bbb P\right)$.\par 
	To describe the noise, we make the following assumptions on the noise in this paper:
	\par
	\noindent
	{\textbf{Hypothesis 1.1. }} $\zeta_{\delta} :\Omega\rightarrow\Bbb R$ is a measurable mapping such that\par 
	(1) $ \zeta_{\delta}(\theta_{t}\omega) $ is a stationary process and is continuous in $t$;\par 
	(2) we have the estimate
	\begin{equation}
			\label{eq:1.8}
		|\zeta_{\delta}(\theta_{t}\omega)|\leq K_{\delta}C_{\omega}(|t|+1), 
	\end{equation}
	where $K_{\delta} $ is a positive function of $\frac{1}{\delta}$, and $\lim\limits_{\delta\rightarrow0^+}K_{\delta}=+\infty$;\par
	(3) for every $T>0$ and $\omega\in\Omega$, we have
	\begin{equation}
			\label{eq:1.9}
		\lim_{\delta\rightarrow0^+}\sup_{|t|\leq T}|\int_{0}^{t}\zeta_{\delta}(\theta_{r}w)dr-\omega(t)|=0.
	\end{equation}\\
	\noindent
	\textbf{Hypothesis 1.2.} There exists a positive constant $\tilde{\delta}$ such that
	\begin{equation}
			\label{eq:1.10}
		\lim_{t\rightarrow\pm\infty}\frac{1}{t}\left[\int_{0}^{t}\zeta_{\delta}(\theta_{s}w)ds-\omega(t)\right]=0\quad \ uniformly\ for\ \delta\in(0,\tilde{\delta}].
	\end{equation}
	\noindent
	{\textbf{Hypothesis 1.3. }} If we consider the random variable
	\begin{equation}
		x_{\delta}(\omega):=\int_{-\infty}^{0}e^r\zeta_{\delta}(\theta_{r}\omega)dr,\quad\forall\omega\in\Omega,
		\nonumber
	\end{equation}
	then 
	\begin{equation}
		\Omega\times\Bbb R\ni(\omega,t)\longmapsto x_{\delta}(\theta_{t}\omega)=\int_{-\infty}^{0}e^r\zeta_{\delta}(\theta_{r+t}\omega)dr
		\nonumber
	\end{equation}
	is a stationary solution of the linear random differential equation:
	\begin{equation}
		\dot{x}=-x+\zeta_{\delta}(\theta_{t}\omega),
		\nonumber
	\end{equation}
	and for each $\omega\in\Omega$, the following properties hold.\par 
	(1) $x_{\delta}(\theta_{t}\omega)$ is continuous in $t$.\par 
	(2) There exists a positive constant $\bar{\delta}$ such that
	\begin{equation}
			\label{eq:1.11}
		\lim_{t\rightarrow\pm\infty}\frac{|x_{\delta}(\theta_{t}\omega)|}{|t|}=0,\quad\lim_{t\rightarrow\pm\infty}\frac{1}{t}\int_{0}^{t}x_{\delta}(\theta_{_r}\omega)dr=0
	\end{equation}
	uniformly with respect to $\delta\in\left(0,\bar{\delta}\right ]$.\par 
	For each $T > 0$,
	\begin{equation}
			\label{eq:1.12}
		\lim_{\delta\rightarrow0^+}\sup_{|t|\leq T}|x_{\delta}(\theta_{t}\omega)-x_{0}(\theta_{t}\omega)|=0,
	\end{equation}
	where $x_{0}(\theta_{t}\omega)=-\int_{-\infty}^{0}e^r\theta_{t}\omega(r)dr$ and is a stationary solution of the stochastic differential
	equation
	\begin{equation}
		dx=-xdt+dW,
		\nonumber
	\end{equation}
	with the properties that $x_{0}(\theta_{_t}\omega)$ is continuous in $t$ and
	\begin{equation}
			\label{eq:1.15}
		\lim_{t\rightarrow\pm\infty}\frac{|x_{0}(\theta_{t}\omega)|}{|t|}=0,\quad\lim_{t\rightarrow\pm\infty}\frac{1}{t}\int_{0}^{t}x_{0}(\theta_{r}\omega)dr=0.
	\end{equation}\par 
	(for existence and properties of $x_{0}(\theta_{t}\omega)$ see \cite{Constantin}).
	\par 
	According to reference \cite{J. Zhao  Shen. J  K. Lu}, Hypotheses 1.1-1.2 are not artificial conditions. For any $\delta>0$, there are at least three candidates of noise that we can choose to meet these conditions, such as the stationary Ornstein-Uhlenbeck process (also called the colored noise)
	\begin{equation}
		\xi_{\delta}(\theta_{t}\omega)=-\int_{-\infty}^{0}\frac{1}{\delta^2}e^{\frac{s}{\delta}}\theta_{t}\omega(s)ds,\quad\forall\omega\in\Omega,
		\nonumber
	\end{equation}
	which is the solution to the Langevin equation
	\begin{equation}
		d\xi_{\delta}=-\frac{1}{\delta}\xi_{\delta}dt+\frac{1}{\delta}dW,\quad\xi_{\delta}(0)=\frac{1}{\delta}\int_{-\infty}^{0}e^{\frac{s}{\delta}}dW(s),
		\nonumber
	\end{equation}
	the derivative of the mollifier of the Brownian motion
	\begin{equation}
		\eta_{\delta}(\theta_{t}\omega)=-\frac{1}{\delta^2}\int_{0}^{\delta}\dot{\phi}(\frac{s}{\delta})\theta_{t}\omega(s)ds,
		\nonumber
	\end{equation}
	where $\phi$ is a nonnegative $C^{\infty}$-function with the properties
	\begin{equation}
		supp\ \phi(t)\subset[0,1],\quad\int_{0}^{1}\phi(t)dt=1
		\nonumber
	\end{equation}
	and the difference quotient of Brownian motion
	\begin{equation}
		\zeta_{_\delta}^*(\theta_{t}\omega)=\frac{1}{\delta}(\omega(t+\delta)-\omega(t)),\quad\forall\omega\in\Omega.
		\nonumber
	\end{equation}
	\par
	These three types of noise have been studied by many authors (see, e.g., \cite{Gu A.A,A. GuB. Wang Asymptotic behavior,Shen J LK} and references therein). These three types of noise have different forms, and there may be other types of noise that satisfy Hypotheses 1.1-1.3. However, in this paper, we will investigate the results derived from all noise types that satisfy Hypotheses 1.1-1.3. Shen and his collaborators studied the persistence of dynamical behavior for a non-autonomous coupled system under a small random perturbation driven by stationary multiplicative noise in \cite{J. Zhao J.Shen P}, and the persistence of $C^1$ inertial manifolds under small random perturbations in \cite{J. Zhao  Shen. J  K. Lu}. To the best of the author's knowledge, this paper is the first one dealing with the limiting behavior of random attractors of nonlocal reaction-diffusion equations driven by stationary noise as small random perturbations approach zero.
	\par 
	In this paper, we will consider a class of stationary processes that satisfy Hypotheses 1.1-1.3, rather than the specific three types of noise mentioned earlier. We consider the limiting behavior of solutions to equation (1.1) as $\delta\rightarrow0^+$ and $\epsilon\rightarrow0^+$. This is different from the corresponding random equation driven by white noise
	\begin{equation}
		\label{eq:1.14}\indent
		\frac{\partial u_{0,\epsilon}}{\partial t}-a(l(u_{0,\epsilon}))\Delta u_{0,\epsilon}=f(u_{0,\epsilon})+h(t)+\epsilon g(t,u_{0,\epsilon})\circ\frac{dW}{dt},
	\end{equation}
	the symbol $\circ$ indicates that the equation is understood in the sense of Stratonovich's integration. Our findings reveal the limiting cases of the solution and random attractor for equation \eqref{eq:1.1} as $\delta\rightarrow0^+ $ and $\epsilon\rightarrow 0^+$. Additionally, when the noise term is linear additive or multiplicative noise, we prove the convergence relationship between the solutions and the random attractor of equation \eqref{eq:1.1} and equation \eqref{eq:1.14} as $\delta\rightarrow0^+$. When equation \eqref{eq:1.1} is driven by linear additive noise (i.e., $g(t, u) = \phi$) or linear multiplicative noise (i.e., $g(t, u) = u$),  we show that the solutions of equation \eqref{eq:1.1} converge to the solutions of equation \eqref{eq:1.2} as $\delta\rightarrow0^+$ and $\epsilon\rightarrow0^+$ (see Theorem \ref{thm:4.3} and Theorem \ref{thm:5.3}. By using appropriate changes of variables given by Ornstein-Uhlenbeck processes, we prove problem \eqref{eq:1.1} generates a random dynamical system that possesses a random attractor, and the random attractor of equation \eqref{eq:1.1} approaches that of equation \eqref{eq:1.2} in terms of the Hausdorff semidistance when $\delta\rightarrow0^+$ and  $\epsilon\rightarrow0^+$ under additive or multiplicative noise (see Theorem \ref{thm:4.14} and Theorem \ref{thm:5.14}. In this paper, we establish the convergence relationship between the random attractors of equation \eqref{eq:1.1} and equation \eqref{eq:1.14} as the parameter $\delta$ approaches zero. For the convergence analysis of the attractors between equation \eqref{eq:1.14} and equation \eqref{eq:1.2} as $\epsilon\rightarrow0^+$, we refer to the proof in \cite{Caraballo THerrera-Cobos M Marín-Rubio P}.
	\par 
	The paper is organized as follows: In Section 2, we provide some basic settings and state the main results. In Section 3, we show the existence of solutions and random nonautonomous attractors for (1.1)\eqref{eq:1.1}. In Section 4, under the additive noise case (i.e., $g(t,u)=\phi$), we prove the convergence of the solutions of problem \eqref{eq:1.1}, and establish the upper semi-continuity of the random attractors of \eqref{eq:1.1} as $ \delta\rightarrow 0^+$ and $\epsilon\rightarrow0^+$. In Section 5, we discuss the same convergence results for the case of multiplicative noise.
	\par

	\section{Preliminaries}
	\indent
	\par 
	\textbf{2.1. Notation.} In this paper, for simplicity, we denote by $H=L^2(\mathcal{O})$, $V=H_{0}^{1}(\mathcal{O})$ and $ V^*=H^{-1}(\mathcal{O})$. Identifying $H$ with its dual, we have the usual chain of dense and compact embeddings $V\subset H\subset V^{\ast} $. We denote by $| \cdot |_{p}$ the norm in $L^{p}(\mathcal{O})$, $| \cdot |$, $\|\cdot\|$ and $\| \cdot \|_{*} $ the norms in $H$, $V$, and $ V^*$, by $(\cdot, \cdot)$ and $((\cdot, \cdot)) $ the scalar products in $H$ and $V$,
	respectively, and by $<\cdot, \cdot>$ the duality product between $V$ and $V^*$. Last, let $C_{c}^{\infty}(\mathcal{O})$ be the space of all functions of class $C^{\infty}$  with compact supports contained in $\mathcal{O}$ .
	\par 
	Denote by $ A = - \Delta $ with Dirichlet boundary condition in our problem, and let
	$D(A)$ be the domain of $A$. In this way, the linear operator $ A : D(A) := V \cap H^2(\mathcal{O})\subset V \rightarrow H$ is positive, self-adjoint with compact resolvent. We denote by $0<\lambda_1\leq\lambda_2 \leq 
	\cdot\cdot\cdot$ the eigenvalues of $A$, and by $e_1, e_2, \cdot\cdot\cdot$, a corresponding complete orthonormal
	system in $L^2(\mathcal{O})$ of eigenvectors of $A$. Recall that for every $v\in V$, the Poincar\'e inequality
	\begin{equation}
		\lambda_1(\mathcal{O})|v|^2\leq\|v\|^2
		\nonumber
	\end{equation}
	holds. In what follows, unless otherwise specified, we write $\lambda_1$ instead of $\lambda_1(\mathcal{O})$.\\
	
	\textbf{2.2. Some assumptions about the problem (1.1).} For nonlocal term $a$, function $a\in C(\Bbb R; \Bbb R^+)$, let $l\in\mathcal{L}(L^2(\mathcal O); \Bbb R)$, and there exist two positive constants $m$ and $\widetilde m$, such that
	\begin{equation}
		\label{eq:2.1}
		m\leq a(s)\leq\widetilde m\quad\forall s\in\Bbb R .
	\end{equation} 
	For nonlinear terms $f$ and $g$, let $f\in C(\Bbb R)$, and there exist positive constants $\alpha_1$, $\alpha_2$,
	$\eta$, $\kappa$, and $p>2$, such that
	\begin{equation}
		\label{eq:2.2}
		|f(s)-f(r)|\leq\eta|s-r|\quad\forall s, r\in\Bbb{R},
	\end{equation}
	\begin{equation}
		\label{eq:2.3}
		-\kappa-\alpha_1|s|^p \leq f(s)s\leq\kappa-\alpha_2|s|^p \quad \forall s \in \Bbb R .
	\end{equation} 
	From \eqref{eq:2.3}, we can deduce that there exists $\beta>0$, such that
	\begin{equation}
		\label{eq:2.4}
		| f(s)|\le\beta(|s|^{p-1}+1)\quad\forall s \in\Bbb R. 
	\end{equation}
	Moreover, letting $ g : \Bbb R \times \Bbb R \rightarrow \Bbb R $ is a continuous function such that for all $t, s\in\Bbb R$,
	\begin{equation}
		\label{eq:2.5}
		|g(t, s)| \leq d_1|s| 
		^{q-1} + \psi_1(t),
	\end{equation}
	\begin{equation}
		\label{eq:2.6}
		|\frac{\partial g}{\partial s}(t, s)| \leq d_2| s| 
		^{q - 2} + \psi_2(t),
	\end{equation}
	\par 
	where $2\leq q<p$, $d_1$ and $d_2$ are nonnegative constants, $\psi_1\in L_{loc}^{p_1}(\Bbb R; L^{p_1}(\mathcal O))$, and $ \psi_2\in L_{loc}^{\infty}(\Bbb R; L^{\infty}(\mathcal O))$( $p_1$ is the conjugated number with $p$ ).\\

	\textbf{2.3. Definitions of weak solutions.} We introduce the concepts of a solution of problem (1.1).\par
	\begin{defn}
		A weak solution to problem (1.1) is a mapping $u_{\delta,\epsilon}(\cdot, \tau, \omega, u_{\tau}):
		[\tau, T)\rightarrow H$ for all $T > \tau$ with $ u_{\delta,\epsilon}(\tau) = u_{\tau} $ , satisfying for any $\tau\in R$, $ \omega \in \Omega $,
		\begin{equation}
			u_{\delta,\epsilon}(\cdot , \tau , \omega , u_{\tau} )\in C(\tau,T;H)\cap L^2(\tau,T;V)\cap L^p(\tau,T;L^p(\mathcal{O})).
			\nonumber
		\end{equation}
		
		Moreover, for every $ t > \tau$ and $ v \in V + L^p(\mathcal{O})$,
		\begin{equation}
			\begin{aligned}
				(u_{\delta,\epsilon},v)=& (u_{\tau},v)+\int_{\tau}^{t}a(l(u_{\delta,\epsilon}))((u_{\delta,\epsilon},v))ds+\int_{0}^{t}(f(u_{\delta,\epsilon}),u_{\delta,\epsilon})ds\\& +\int_{\tau}^{t}<h,v>ds+\int_{\tau}^{t}(g(s,u_{\delta,\epsilon}(s)))\zeta_{\delta}(\theta_{s}\omega).
			\end{aligned}
			\nonumber
		\end{equation}
	\end{defn} 
	\par 
	\textbf{2.4. Cocycles for nonlocal PDEs.} In this section, we will introduce some basic concepts related to non-autonomou random dynamical systems (see e.g.  \cite{A. Carvalho}, \cite{B. Wang}, \cite{P. W.B}, \cite{P. Kloeden}). 
	\par 
	Let $\left(\Omega,\mathcal{F} ,\Bbb P\right),\left\lbrace{\theta _t}\right\rbrace_{t\in\Bbb R} ) $ be a metric dynamical space (see \cite{A.V. Babin and M.I.}). In what follows, we use $(X, d)$ to denote a complete separable metric space. If $A$ and $B$ are two nonempty subsets of $X$, then we use $dist_{X}(A, B) := \sup_{a\in A} \inf_{b\in B}d(a, b)$ to denote their Hausdorff semi-distance. For other concepts such as upper semi-continuity of random attractors with respect to parameters, we can refer to \cite{A. Carvalho,J. Xu and T. Caraballo,P. W.B,P. Kloeden} and the references therein.
	\par
	\begin{defn}
		Let $D : \Bbb R \times \Omega \rightarrow 2^X$ be a set-valued mapping with closed nonempty images. We say $D$ is measurable with respect to $ \mathcal{F}$ in $\Omega$ if the mapping $\omega \in \Omega\rightarrow d(x, D(\tau, \omega))$ is $(\mathcal{F}, \mathcal{B}(\Bbb R))$-measurable for every fixed $x\in X $ and $\tau\in R$.
	\end{defn}
	\par
	\begin{defn}
		Let $ \mathcal{D} $ be a collection of some families of nonempty subsets of X and $ B=\{B(\tau, \omega) : \tau\in R, \omega\in\Omega\}\in\mathcal{D}$. Then $B$ is called a $\mathcal{D}$-pullback absorbing set for $\Phi$ if for all $\tau\in R$, $\omega \in \Omega$, and for every $B \in \mathcal{D}$, there
		exists $T = T(B,\tau, \omega)> 0$ such that
		\begin{equation}
			\Phi (t, \tau- t, \theta_{-t}\omega , B(\tau- t, \theta_{-t}\omega ))\subset B(\tau, \omega ),\quad \forall t \geq T.
			\nonumber
		\end{equation}
	\end{defn} 
	\par 
	\begin{defn}
		Let $\mathcal{D}$ be a collection of some families of
		nonempty subsets of $X$. Then $\Phi$ is said to be $\mathcal{D}$ -pullback asymptotically compact in
		$X$ if for all $\tau\in R$ and $\omega \in \Omega$ , the sequence
		\begin{equation}
			\{ \Phi (t_n, \tau- t_n, \theta _{-t_n} \omega , x_n)\} 
			_{\infty}^{n=1}\ has\ a\ convergent\ subsequence\ in\ X,
			\nonumber
		\end{equation}
		as $t_n \rightarrow \infty$ and $x_n \in D(\tau- t_n, \theta_{-t_n} \omega )$ with $\{ D(\tau, \omega ) : \tau\in R,\omega \in \Omega \} \in \mathcal{D}$.
	\end{defn}
	\par 
	\begin{defn}
		Let  $\mathcal{D}$ be a collection of some families of
		nonempty subsets of X and $\mathcal{A} = \{ \mathcal{A}(\tau, \omega ) : \tau\in R, \omega \in \Omega \} \in \mathcal{D}$. Then $\mathcal{A}$ is called a $\mathcal{D}$ -pullback attractor for $\Phi$ if the following conditions (i)--(iii) are fulfilled:\par 
		(i) $\mathcal{A}$ is measurable in the sense of Definition 2.4.1, and $\mathcal{A} (\tau, \omega )$ is compact for all $\tau\in R$, $\omega \in \Omega$ .\par 
		(ii) $\mathcal{A}$ is invariant, that is, for every $\tau\in R$ and $\omega\in \Omega$, 
		\begin{equation}
			\Phi(t, \tau , \omega,\mathcal{A}(\tau, \omega))= A(\tau+t, \theta_t\omega),\quad\forall t \geq 0.
			\nonumber
		\end{equation}
		\par 
		(iii) $\mathcal{A}$ attracts every member of $\mathcal{D}$ , that is, for every $D = \{ D(\tau, \omega ) : \tau\in R,\omega \in 
		\Omega \} \in \mathcal{D}$ and for every $\tau\in R$, $ \omega \in \Omega $,
		\begin{equation}
			\lim_{t\rightarrow\infty}d(\Phi (t, \tau- t, \theta_{-t}\omega , D(\tau- t, \theta_{-t}\omega )),  \mathcal{A}  (\tau, \omega )) = 0.
			\nonumber
		\end{equation}
	\end{defn}
	We have provided all necessary definitions for random dynamical systems. Next, we will define a cocycle $\Phi : \Bbb R^+ \times \Bbb R \times \Omega \times H \rightarrow H $ for (\ref{eq:1.1}), such that
	for all $t \in \Bbb R^+, \tau\in R, \omega \in \Omega$, and $u_{\tau}\in H$,
	\begin{equation}
		\Phi (t, \tau, \omega , u_{\tau}) = u(t+\tau; \tau, \theta_{-\tau} \omega , u_{\tau} ).
		\nonumber
	\end{equation}
	where $u(\cdot ; \tau , \omega , u_\tau )$ denotes the solution to (1.2), which will be proved to exist in Section 3. Thus, $\Phi$ will be a continuous cocycle on $H$ over $(\Omega , \mathcal{F} , \Bbb P, \{ \theta_t\}_{t\in \Bbb R})$. Moreover, let $D = \{ D(\tau , \omega ) : \tau \in \Bbb R , \omega \in \Omega \}$ be a tempered family of bounded nonempty subsets of $H$, that is, for every $\gamma > 0, \tau \in \Bbb R$, and $\omega \in \Omega$,
	\begin{equation}
			\label{eq:2.7}
		\lim_{t\rightarrow-\infty}e^{\gamma t}|D(\tau+t,\theta_{t}\omega)|=0,
	\end{equation}
	where $| D| = sup_{u\in D}|u|$. Throughout this section, we will use $\mathcal{D}$ to denote the collection of all tempered families of bounded nonempty subsets of $H$, i.e.,
	\begin{equation}
		\mathcal{D}=\{D=\{D(\tau,\omega) : \tau\in R,\omega\in\Omega\} : D\quad satifies\ \eqref{eq:2.7} \}
	\end{equation}
	\par 
	\begin{rem}
		Since the cocycle generated by problem (1.1) depends on the parameter $\delta$ and $\epsilon$, we will use $\Phi_{\delta,\epsilon}$ instead of using the notation $\Phi$.
	\end{rem}
	
	\section{Attractors of nonlocal stochastic PDEs driven by stationary noise.}
	\indent
	\par 
	Our aim is to study the existence of the attractors for solution of problem \eqref{eq:1.1}. Recently, Xu and his collaborators studied the long time behavior of nonautonomous nonlocal partial differential equations driven by colored noise in \cite{J. Xu and T. Caraballo}. In this section, we will further analyze the problem \eqref{eq:1.1} driven by stationary noise. We first consider the existence and uniqueness of weak solution to problem \eqref{eq:1.1}.
	\indent
\begin{thm}\label{thm:3.1}
	 Assume that function $a$ is globally Lipschitz and satisfies \ref{eq:2.1},
	$ f \in C(\Bbb R ) $ fulfills \eqref{eq:2.2}--\eqref{eq:2.4}, $ h \in L_{loc}^2(\Bbb R^+; V^*) $ and $ l \in L^2(\mathcal{O})$. Additionally, function
	$g$ satisfies \eqref{eq:2.5}--\eqref{eq:2.6}, and exist $\epsilon_0\in(0,1)$ such that for all $\epsilon\in(0,\epsilon_0]$, $\delta\in(0, 1]$, $\omega\in\Omega$, and for each initial datum $ u_{0} \in H $. Then there exists a unique weak solution to problem (1.1) in the sense of Definition 2.3.1. Moreover, this solution
	behaves continuously in $H$ with respect to the initial values.
\end{thm}\par 
	\begin{proof}
		By applying the Galerkin method and energy estimations \cite[Chapter 3, Theorem
		3.3]{M. Herrera-Cobos}, we can prove problem (1.1) have a unique solution for every $ T >\tau$ and $ \omega \in \Omega $,
		\begin{equation}
			u_{\delta, \epsilon}(\cdot ,\tau, \omega , u_{\tau} )\in C(\tau,T;H)\cap L^2(\tau,T;V)\cap L^p(\tau,T;L^p(\mathcal{O})).
			\nonumber
		\end{equation}
	\end{proof}
	Next, we derive uniform estimations on the solution of problem (1.1) and then prove $\mathcal{D}$-pullback asymptotic compactness by using the idea introduced by Ball in \cite{J.BallGlobalattractorsfor}. To this end, we need the following assumptions:\par 
	(h1) Suppose that
	\begin{equation}
		\int_{-\infty}^{\tau}e^{m\lambda_1s}\|h(s)\|_{*}^2ds<\infty.
		\nonumber
	\end{equation}\\
	For the existence of tempered random attractors, we need the assumption below:\par 
	(h2) For every $ \gamma > 0$, it holds that
	\begin{equation}
		\lim_{t\rightarrow-\infty}e^{\gamma t}\int_{-\infty}^{0}e^{m\lambda_1s}\|h(s+t)\|_{*}^2ds=0.
		\nonumber
	\end{equation}
	It is worth stressing that (h1) and (h2) do not require that $h(t)$ is bounded in $ V^{\ast}$ as $ t \rightarrow \pm \infty$ .\\
	
	\begin{lem}\label{lem:3.2} Assume the conditions of Theorem 3.1 and (h1) hold. Letting $\epsilon_0\in(0,1)$ such that for all $\epsilon\in(0,\epsilon_0]$. Then, for every $\delta\in(0, 1]$, $\tau\in R$, $\omega\in\Omega$, and $ D = \{ D(\tau, \omega ): \tau\in R, \omega\in\Omega\}\in\mathcal{D}$, there exists $T = T(\tau,\omega,\delta,\epsilon, D) > 0$ such that for all $t\geq T$ and $\sigma\geq\tau-t$, the solution of problem (1.1) satisfies
	\begin{equation}
		\begin{aligned}
			&|u_{\delta,\epsilon}(\sigma;\tau-t,\theta_{-\tau}\omega,u_{\tau-t})|^2\leq e^{-m\lambda_1(\sigma-\tau)} \\
			&+\int_{-\infty}^{\sigma-\tau}e^{m\lambda_1(s-\sigma+t)}\left(\frac{2}{m}\|h(s+\tau)\|_{*}^2+\left(2\kappa+\epsilon c|\zeta_{\delta}(\theta_{s}\omega)|^{\frac{p}{p-q}}\right)|\mathcal{O}|+\epsilon c|\zeta_{\delta}(\theta_{s}\omega)|^{p_1}|\psi_1|_{L^{p_1}}^{p_1}\right)ds,
			\nonumber
		\end{aligned}
	\end{equation}
	\begin{equation}
		\begin{aligned}
			&\int_{\tau-t}^{\tau}e^{m\lambda_1(s-\tau)}\|u_{\delta,\epsilon}(s;\tau-t,\theta_{-\tau}\omega,u_{\tau-t})\|^2ds\\ &\leq\frac{2}{m}+\frac{2}{m}\int_{-\infty}^{0}e^{m\lambda_1s}\left(\frac{2}{m}\|h(s+\tau)\|_{*}^2+\left(2\kappa+\epsilon c|\zeta_{\delta}(\theta_{s}\omega)|^{\frac{p}{p-q}}\right)|\mathcal{O}|+\epsilon c|\zeta_{\delta}(\theta_{s}\omega)|^{p_1}|\psi_1|_{L^{p_1}}^{p_1}\right)ds,\\
			and\\
			&\int_{\tau-t}^{\tau}e^{m\lambda_1(s-\tau)}|u_{\delta,\epsilon}(s;\tau-t,\theta_{-\tau}\omega,u_{\tau-t})|_{p}^pds\\
			&\leq\frac{1}{\alpha_2}+\frac{1}{\alpha_2}\int_{-\infty}^{0}e^{m\lambda_1s}\left(\frac{2}{m}\|h(s+\tau)\|_{*}^2+\left(2\kappa+\epsilon c|\zeta_{\delta}(\theta_{s}\omega)|^{\frac{p}{p-q}}\right)|\mathcal{O}|+\epsilon c|\zeta_{\delta}(\theta_{s}\omega)|^{p_1}|\psi_1|_{L^{p_1}}^{p_1}\right)ds,
			\nonumber
		\end{aligned}
	\end{equation}
	where $ u_{\tau-t} :=u_{\delta,\epsilon}(\tau-t)\in D(\tau- t, \theta_{-t}\omega )$, and $c$ is a constant which depends on $ \alpha_2$, $p$, $p_1$, $q$, and $d_1$ but
	not on $ \delta $ and $\epsilon$.
\end{lem}
\par
	\begin{proof} 
		Multiplying by $ u_{\delta\,\epsilon}(\cdot)$ on both sides of \eqref{eq:1.1} in $H$, we derive
		\begin{equation}
			\label{eq:3.1}
			\frac{d}{dt}|u_{\delta,\epsilon}|^2+2a(l(u_{\delta,\epsilon}))\|u_{\delta,\epsilon}\|^2=2(f(u_{\delta,\epsilon}),u_{\delta,\epsilon})+2\langle h(t),u_{\delta,\epsilon}\rangle+2\epsilon\zeta_{\delta}(\theta_{t}w)(g(t,u_{\delta,\epsilon}),u_{\delta,\epsilon}).
		\end{equation}
		By \eqref{eq:2.3} and the Young inequality, we have
		\begin{equation}
			\label{eq:3.2}
			2(f(u_{\delta,\epsilon}),u_{\delta,\epsilon})\leq2\int_{\mathcal{O}}(\kappa-2\alpha_2|u_{\delta,\epsilon}|^p)dx\leq2\kappa|\mathcal{O}|-2\alpha_2|u_{\delta,\epsilon}|^p,
		\end{equation}
		\begin{equation}
			\label{eq:3.3}
			2\langle h(t),u_{\delta,\epsilon}\rangle\leq\frac{2}{m}\|h(t)\|_{*}^2+\frac{m}{2}\|u_{\delta,\epsilon}\|^2.
		\end{equation}
		By \eqref{eq:2.5}-\eqref{eq:2.6} and the Young inequality that
		\begin{equation}
				\label{eq:3.4}
			|2\epsilon\zeta_{\delta}(\theta_{t}w)(g(t,u_{\delta,\epsilon}),u_{\delta,\epsilon})|\leq \epsilon c|\mathcal{O}||\zeta_{\delta}(\theta_{t}w)|^{\frac{p}{p-q}}+\epsilon c|\zeta_{\delta}(\theta_{t}w)|^{p_1}|\psi_1|_{p_1}^{p_1}+\alpha_2|u_{\delta,\epsilon}|_{p}^p,
		\end{equation}
		where $c$ is a constant depending on $ \alpha_2, p,  p_1, r_1 $ and $ d_1 $.
		\par 
		From \eqref{eq:3.2}-\eqref{eq:3.4}, \eqref{eq:2.1} and Poincar\'e inequalities, we get
		\begin{equation}
				\label{eq:3.5}
			\begin{aligned}
				\frac{d}{dt}|u_{\delta,\epsilon}|^2&+m\lambda_1|u_{\delta,\epsilon}|^2+\frac{m}{2}\|u_{\delta,\epsilon}\|^2+\alpha_2|u_{\delta,\epsilon}|_{p}^p\\
				&\leq\frac{2}{m}\|h(t)\|_{*}^2+\left(2\kappa+\epsilon c|\zeta_{\delta}(\theta_{t}w)|^{\frac{p}{p-q}}\right)|\mathcal{O}|+\epsilon c|\zeta_{\delta}(\theta_{t}w)|^{p_1}|\psi_1|_{p_1}^{p_1}.
			\end{aligned}
		\end{equation}\\
		
		\noindent
		By direct calculations involving $ u_{\delta,\epsilon}(\sigma;\tau-t,\theta_{-(\tau-t)}\omega,u_{\tau-t}) $ and replacing $ \omega $ by $ \theta_{-t}\omega $, we derive 
		\begin{equation}
				\label{eq:3.6}
			\begin{aligned}
				|u_{\delta,\epsilon}&(\sigma;\tau-t,\theta_{-\tau}\omega,u_{\tau-t})|^2+\frac{m}{2}\int_{\tau-t}^{\sigma}e^{m\lambda_1 (s-\sigma)}\|u_{\delta,\epsilon}(s;\tau-t,\theta_{-\tau}\omega,u_{\tau-t})\|^2ds\\
				&+\alpha_2\int_{\tau-t}^{\sigma}e^{m\lambda_1(s-\sigma)}|u_{\delta,\epsilon}(s;\tau-t,\theta_{-\tau}\omega,u_{\tau-t})|_{p}^pds\\
				&\leq e^{-m\lambda_1(\sigma-\tau+t)}|u_{\tau-t}|^2\\ 
				&+\int_{-t}^{\sigma-\tau}e^{m\lambda_1(s-\sigma+\tau)}\left(\frac{2}{m}\|h(s+\tau)\|_{*}^2+\left(2\kappa+\epsilon c|\zeta_{\delta}(\theta_{s+\tau}\omega)|^{\frac{p}{p-q}}\right)|\mathcal{O}|+\epsilon c|\zeta_{\delta}(\theta_{s+\tau}\omega)|^{p_1}|\psi_1|_{p_1}^{p_1}\right)ds.
			\end{aligned}
		\end{equation}
		It follows from (h1) that
		\begin{equation}
				\label{eq:3.7}
			\int_{-\infty}^{\sigma-\tau}e^{m\lambda_1(s-\sigma+\tau)}\left(\frac{2}{m}\|h(s+\tau)\|_{*}^2+\left(2\kappa+\epsilon c|\zeta_{\delta}(\theta_{s+\tau}\omega)|^{\frac{p}{p-q}}\right)|\mathcal{O}|\right)ds<\infty,
		\end{equation}
		and from Hypothesis 1.1 and $ \psi_1 \in L_{loc}^2(\Bbb R ; L^{p_1}(\mathcal O))$, we obtain
		\begin{equation}
				\label{eq:3.8}
			\int_{-\infty}^{\sigma-\tau}e^{m\lambda_1(s-\sigma+\tau)}|\zeta_{\delta}(\theta_{s+\tau}\omega)|^{p_1}|\psi_1|_{p_1}^{p_1}ds<\infty.
		\end{equation}
		Since $ u_{\tau-t}\in D(\tau-t,\theta_{-t}\omega)\in\mathcal{D}$, we can get
		\begin{equation}
				\label{eq:3.9}
			e^{-m\lambda_1t}|u_{\tau-t}|^2\leq e^{-m\lambda_1t}|D(\tau-t,\theta_{-t}\omega)|^2\rightarrow0,\quad as\ t\rightarrow\infty.
		\end{equation}
		By \eqref{eq:3.9},  there exists $ T = T(\tau, \omega , D) > 0$, such that for all $ t \ge T$, we have,
		\begin{equation}
				\label{eq:3.10}
			e^{-m\lambda_1(\sigma-\tau+t)}|u_{\tau-t}|^2\leq1.
		\end{equation}
		Finally, by \eqref{eq:3.7}-\eqref{eq:3.10}, we completes the proof.
	\end{proof}
	\par 
	Next, we present the existence of $\mathcal{D}$-pullback absorbing set for the continuous cocycle $\Phi_{\delta,\epsilon}$ in $H$.\\
	
	\par 
		\begin{cor}\label{cor:3.3} Assume the conditions of Theorem 3.1 and (h2) hold. Then the continuous cocycle $\Phi_{\delta,\epsilon}$ associated with problem \eqref{eq:1.1} possesses a closed measurable $\mathcal{D}$-pullback absorbing set $K_{\delta,\epsilon}= \{K_{\delta,\epsilon}(\tau, \omega): \tau\in R, \omega\in\Omega\} \in\mathcal{D}$ in $H$. Namely, letting $\epsilon_0\in (0,1]$, for any
	given $\delta\in (0, 1]$ and $\epsilon\in(0,\epsilon_0]$, every $\tau\in R$, $\omega\in\Omega$, we denote
	\begin{equation}
		K_{\delta,\epsilon}(\tau,\omega)=\{u_{\delta, \epsilon}\in H: |u_{\delta, \epsilon}|^2\leq R_{\delta,\epsilon}(0,\omega)\},
		\nonumber
	\end{equation}
	where
	\begin{equation}
		R_{\delta,\epsilon}(\tau,\omega)=1+\int_{-\infty}^{0}e^{m\lambda_1s}\left(\frac{2}{m}\|h(s+\tau)\|_{*}^2+\left(2\kappa+\epsilon c|\zeta_{\delta}(\theta_{s+\tau}\omega)|^{\frac{p}{p-q}}\right)|\mathcal{O}|+\epsilon c|\zeta_{\delta}(\theta_{s+\tau}\omega)|^{p_1}|\psi_1|_{L^{p_1}}^{p_1}\right)ds.
		\nonumber
	\end{equation}
\end{cor}
	\begin{proof}
		For every $\tau\in R$, $ \omega\in\Omega$ and $D\in\mathcal{D}$, it follows from Lemma
		3.2 that there exists $ T = T(\tau, \omega , D) > 0 $ , such that for all $ t \geq T $,
		\begin{equation}
			\label{eq:3.11}
			\Phi_{\delta,\epsilon}(t, \tau- t, \theta_{-t}\omega , D(\tau- t, \theta_{-t}\omega )) = u_{\delta,\epsilon}(\tau;\tau-t, \theta_{-\tau}\omega, D(\tau-t, \theta_{-t}\omega))\subset K_{\delta,\epsilon}(\tau, \omega ).
		\end{equation}
		Next, to finish this proof, we need show $K_{\delta,\epsilon}$ belongs to $\mathcal{D}$. Letting $\gamma$ be an arbitrary positive number, for every $\tau\in R$, $\omega\in\Omega$, we have,
		\begin{equation}
			\begin{aligned}
				\label{eq:3.12}
				&\lim_{t\rightarrow-\infty}e^{\gamma t}|K_{\delta,\epsilon}(\tau+t,\theta_{t}\omega)|=\lim_{t\rightarrow-\infty}e^{\gamma t}R_{\delta,\epsilon}(t,\theta_{t}\omega)\\
				&=\lim_{t\rightarrow-\infty}e^{\gamma t}\left(1+\int_{-\infty}^{0}e^{m\lambda_1s}
				\left(\frac{2}{m}\|h(s+\tau+t)\|_{*}^2+\left(2\kappa+\epsilon c|\zeta_{\delta}(\theta_{s+\tau+t})|^{\frac{p}{p-q}}\right)|\mathcal{O}|\right)ds\right)\\
				&+\lim_{t\rightarrow-\infty}e^{\gamma t}\epsilon c\left(\int_{-\infty}^{0}e^{m\lambda_1s}|\zeta_{\delta}(\theta_{s+\tau+t}w)|^{p_1}|\psi_1|_{L^{p_1}}^{p_1}ds\right),
			\end{aligned}
		\end{equation}
		since (h2), for any $ \gamma >0$, we get
		\begin{equation}
			\label{eq:3.13}
			\lim_{t\rightarrow-\infty}e^{\gamma t}|K_{\delta,\epsilon}(\tau+t,\theta_{t}\omega)|=0,
		\end{equation}
		Along with \eqref{eq:3.11} and \eqref{eq:3.13}, we complete the proof.
	\end{proof}
	\par 
	Next, we discuss the asymptotic compactness of the solutions to problem \eqref{eq:1.1}. Namely, the sequence of solutions to problem \eqref{eq:1.1} is compact in $H$.\\
	
	\begin{lem}\label{lem:3.4} Under assumptions of Lemma \ref{lem:3.2}, the continuous cocycle $\Phi_{\delta,\epsilon}$ associated with problem \eqref{eq:1.1} is $\mathcal{D}$-pullback asymptotic compactness in $H$. That is, for every $ \tau\in R$, $\omega\in\Omega$, $D = \{D(\tau, \omega) : \tau\in R, \omega\in\Omega\}\in\mathcal{D}$, as $t_n\rightarrow\infty$, the initial data $u_{\tau,n} :=u_{\delta,\epsilon,n}(\tau)\in D(\tau- t_n, \theta{-t_n}\omega)$, and the sequence $\{\Phi_{\delta,\epsilon}(t_n, \tau-t_n, \theta_{-t_n}\omega, u_{\tau, n})=u_{\delta,\epsilon}(\tau; \tau-t_n, \theta_{-\tau}\omega, u_{\tau,n})\}$ (solutions to problem \eqref{eq:1.1}) has a convergence subsequence in $H$.
\end{lem}\par
	\begin{proof} 
		Letting $\{u_{\tau,n}\}_{n=1}^{\infty}$ be a sequence in $D(\tau, \omega)$, by Lemma \ref{lem:3.2} that there exists $T: =T(\tau, \omega, D)>0$, such that for all $t_n>T$, we find that
		\begin{equation}
			\label{eq:3.14}
			\{u_{\delta,\epsilon}(\cdot; \tau-t_n, \theta_{-\tau}\omega, u_{\tau,n})\}\ is\ bounded\ in\ L^{\infty}(\tau- T,\tau; H)\cap L^2(\tau-T, \tau ; V )\cap L^p(\tau- T,\tau;L^p(\mathcal{O} )).
		\end{equation}
		By \eqref{eq:2.1}, \eqref{eq:2.2}, \eqref{eq:2.5} and \eqref{eq:2.6}, we get from \eqref{eq:3.14} that 
		\begin{equation}
				\label{eq:3.15}
			\{ f(u_{\delta,\epsilon}(\cdot;\tau -t_n, \theta_{-\tau}\omega, u_{\tau ,n}))\}\ is\ bounded\ in\ L^q(\tau- T,\tau ;L^q(\mathcal{O})),
		\end{equation}
		\begin{equation}
				\label{eq:3.16}
			\{ g(\cdot , u_{\delta,\epsilon}(\cdot; \tau-t_n, \theta_{-\tau}\omega, u_{\tau ,n}))\}\ is\ bounded\ in\ L^{p_1}(\tau- T,\tau; L^{p_1}(\mathcal{O})).
		\end{equation}
		\begin{equation}
				\label{eq:3.17}
			\{a(l(u_{\delta,\epsilon}(\cdot;\tau-t_n, \theta_{-\tau}\omega, u_{\tau,n}))\Delta u_{\delta,\epsilon}(\cdot; \tau-t_n, \theta_{-\tau}\omega, u_{\tau,n}) \}\ is\ bounded\ in\ L^2(\tau- T, \tau ; V^{*}).
		\end{equation}
		Combing with \eqref{eq:3.14}-\eqref{eq:3.17}, we have
		\begin{equation}
				\label{eq:3.18}
			\left\{\frac{d}{dt}u_{\delta,\epsilon}(\cdot; \tau-t_n, \theta_{-\tau}\omega, u_{\tau,n})\right\}\in L^2(\tau- T, \tau ; V^{*})+L^q(\tau- T, \tau ;L^q(\mathcal{O} ))+L^{p_1}(\tau-T,\tau; L^{p_1}(\mathcal{O})).
		\end{equation}
		Since the embedding $ V \hookrightarrow H $ is compact, by \eqref{eq:3.14}-\eqref{eq:3.18} and the Aubin--Lions compactness lemma, we can deduce that there exists $ u_{\delta,\epsilon}\in L^2(\tau- T,\tau ; H)$  such that, up to a subsequence,
		\begin{equation}
				\label{eq:3.19}
			u_{\delta,\epsilon}(\cdot; \tau-t_n, \theta_{-\tau}\omega, u_{\tau,n})\rightarrow u_{\delta,\epsilon}\quad strongly\quad in\quad  L^2(\tau-T,\tau;H).
		\end{equation}
		By choosing a further subsequence ( we still denoted the same), by (3.19), for almost all $s\in(0, T)$, we have
		\begin{equation}
				\label{eq:3.20}
			u_{\delta,\epsilon}(\tau- s; \tau-t_n, \theta_{-\tau} \omega , u_{\tau,n}) \rightarrow u_{\delta,\epsilon}(\tau- s)\quad  strongly\quad in\ H .
		\end{equation}
		Since $ 0 < s < T $, there exists a constant $ 0 < T' < T$, such that for $s\in(\tau- T, \tau- T')$, the convergence \eqref{eq:3.19} is true, the solution with initial data in $H$, by \eqref{eq:3.20}, we obtain that
		\begin{equation}
			\begin{aligned}
				u_{\delta,\epsilon}(\tau; \tau- t_n, \theta_{-\tau} \omega , u_{\tau,n}) = &u_{\delta,\epsilon}(\tau; \tau- s, \theta_{-\tau} \omega , u_{\delta,\epsilon}(\tau- s; \tau- t_n, \theta _{-\tau} \omega , u_{\tau,n}))\\
				&\rightarrow u_{\delta,\epsilon}(\tau, \tau- s, \theta_{-\tau} \omega , u_{\delta,\epsilon}(\tau- s)),
				\nonumber
		\end{aligned}
		\end{equation}
		The proof is finished.
	\end{proof}\par 
	According to Lemma \ref{lem:3.4}, we can deduce that the continuous cocycle $ \Phi_{\delta, \epsilon}$ associated with problem \eqref{eq:1.1} is $\mathcal{D}$ -pullback asymptotic compactness in $H$.\\
	
	\begin{thm}\label{thm:3.5} Assume function $a$ is globally Lipschitz and satisfies \eqref{eq:2.1}, $ f\in C(\Bbb R ) $ fulfills \eqref{eq:2.2}-\eqref{eq:2.4}, $ h \in L_{loc}^2(\Bbb R^+;V^*)$ satisfies (h1)--(h2), and $ l \in L^2(\mathcal{O})$. In addition, function $g$ satisfies \eqref{eq:2.5}-\eqref{eq:2.6}. Then, letting $\epsilon_0\in (0,1]$, for any given $\delta \in (0, 1]$ and $\epsilon\in (0,\epsilon_0]$, the continuous cocycle $ \Phi_{\delta,\epsilon}$ associated to problem \eqref{eq:1.1} has a unique $ \mathcal{D}$-pullback attractor $\mathcal{A}_{\delta,\epsilon} = \{\mathcal{A}_{\delta,\epsilon}(\tau, \omega) : \tau\in R,\omega \in \Omega\}\in \mathcal{D}$ in $H$.
\end{thm}
	\begin{proof}
		The results follows from definition of weak solution in Section 2. By Corollary \ref{cor:3.3}, we given that $\Phi_{\delta,\epsilon}$ possesses a closed and measurable $\mathcal{D}$-pullback absorbing set $K_{\delta,\epsilon}(\tau,\omega)$ within $\mathcal{D}$, and by Lemma 3.4, $\Phi_{\delta,\epsilon}$ is $\mathcal{D}$-pullback asymptotically compactness in $H$. It follows that the existence and uniqueness of the $\mathcal{D}$-pullback attractor $\mathcal{A}_{\delta,\epsilon}(\Omega)$ for $\Phi_{\delta,\epsilon}$ can be deduced, for more details, see \cite[Proposition 2.10]{B. Wang}.
	\end{proof}
	\section{Convergence of random attractors for stochastic nonlocal PDEs with additive noise}
	\indent 
	\par 
	In the next two sections, we will consider two particular cases of problem (1.1), when the stochastic forcing term $g(t, u)$ in problem \eqref{eq:1.1} is linear (such as $g(t, u) = u$, multiplicative noise) or independent on u (such as $g(t, u) = \phi$ , additive noise). We can prove the existence of random attractors to problem \eqref{eq:1.1} via performing a conjugation which transforms the stochastic equation into a random one. Therefore, a reasonable question is, if we study the long time behavior of problem \eqref{eq:1.1} with additive noise or multiplicative noise, what is the relationship between problem \eqref{eq:1.1} and problem \eqref{eq:1.14} with additive or multiplicative noise when the perturbation parameter $\delta$ and $\epsilon$ goes to zero? We will answer this question in the next two sections.
	\par 
	To simplify the presentation, in the following sections we assume $h(t) = 0$, which means we will study the dynamics of the stochastic autonomous PDEs. Actually, the ideas to work on the stochastic non-autonomous PDEs are the same (see Section 3 ). In this section, we study the case that $g(t,u)$ in \eqref{eq:1.1} is a constant $\phi\in V\cap H^2(\mathcal{O})$, i.e. the case of additive noise. 
	\par 
	\textbf{4.1. Convergence of solutions}. As a bridge, we consider the convergence of solutions of stochastic differential equation
	\begin{equation}
		\label{eq:4.1}\indent
		\frac{\partial u_{0,\epsilon}}{\partial t}-a(l(u_{0,\epsilon}))\Delta u_{0,\epsilon}=f(u_{0,\epsilon})+\epsilon\phi \frac{dW}{dt}\quad u_{0,\epsilon}(\tau)=u_0\in H,
	\end{equation}\\
	and differential equation
	\begin{equation}
		\label{eq:4.2}\indent
		\frac{\partial u}{\partial t}-a(l(u))\Delta u=f(u),\quad u(\tau)=u_0\in H,
	\end{equation}
	as $\epsilon\rightarrow0^+$. We also study the convergence of solutions of the random differential equation
	\begin{equation}
		\label{eq:4.3}
		\frac{\partial u_{\delta,\epsilon}}{\partial t}-a(l(u_{\delta,\epsilon}))\Delta u_{\delta,\epsilon}=f(u_{\delta,\epsilon})+\epsilon\phi \zeta_{\delta}(\theta_{t}w),\quad u_{\delta,\epsilon}(\tau)=u_0\in H
	\end{equation}
	as $ \delta\rightarrow0^+ $ and $\epsilon\rightarrow0^+ $, where function $ a $ and $ f $ satisfty conditions \eqref{eq:2.2}-\eqref{eq:2.4} with $ p=2 $ and $\beta=C_{f}$, respectively.\par 
	For any $T>\tau$ with $T>0$, $\tau\in R$, $\omega\in\Omega$, we can show that the solution of equation \eqref{eq:4.1} is uniformly convergent to the solution of equation \eqref{eq:4.2} on $ [\tau, T ] $ as $ \epsilon\rightarrow0^+ $, and using this convergence, the solution of equation \eqref{eq:4.2} is uniformly convergent to the solution of equation \eqref{eq:4.2} on $ [\tau, T ] $ as $ \delta\rightarrow0^+ $ and $\epsilon\rightarrow0^+ $.
	\par 
	Let $u_{0,\epsilon}(t,\omega,u_0)$, $u(t,u_0)$ and $u_{\delta,\epsilon}(t,\omega,u_0)$ be solutions of equations \eqref{eq:4.1}, \eqref{eq:4.2} and \eqref{eq:4.3}, respectively. Based on the previous assumptions and \eqref{eq:1.10}, one can get their existence for all $t\ge\tau$.
	\par 
	Let $v_{0,\epsilon}=u_{0,\epsilon}-\epsilon\phi W_t$. From \eqref{eq:4.1} we get that $ v_{0,\epsilon}(t,\omega, v_{0,\epsilon}(\tau))$ satisfies equation
	\begin{equation}
		\begin{aligned}
			\label{eq:4.4}
			\frac{\partial{v_{0,\epsilon}}}{\partial{t}}=&a(l(v_{0,\epsilon})+\epsilon W_tl(\phi))\Delta v_{0,\epsilon}+f(v_{0,\epsilon}+\epsilon\phi W_t)+\epsilon\phi W_t\\
			&+a(l(v_{0,\epsilon})+\epsilon W_tl(\phi))\epsilon W_t\Delta\phi,
		\end{aligned}
	\end{equation}
	and
	\begin{equation}
		\label{eq:4.5}
		v_{0,\epsilon}(\tau)=u_0-\epsilon\phi W_\tau\in H.
	\end{equation}
	\par 
	Set $v_{\delta,\epsilon}=u_{\delta,\epsilon}-\epsilon\phi\int_{0}^{t}\zeta_{\delta}(\theta_{s}\omega)ds$. From \eqref{eq:4.3} we find that $ v_{\delta,\epsilon}(t,\omega, v_{\delta,\epsilon}(\tau))$ satisfies equation
	\begin{equation}
		\begin{aligned}
			\label{eq:4.6}
			\frac{\partial{v_{\delta,\epsilon}}}{\partial{t}}=&a(l(v_{\delta,\epsilon})+\epsilon \int_{0}^{t}\zeta_{\delta}(\theta_{s}\omega)dsl(\phi))\Delta v_{\delta,\epsilon}+f(v_{\delta,\epsilon}+\epsilon\phi\int_{0}^{t}\zeta_{\delta}(\theta_{s}\omega)ds)+\epsilon\phi\int_{0}^{t}\zeta_{\delta}(\theta_{s}\omega)ds\\
			&+a(l(v_{\delta,\epsilon})+\epsilon \int_{0}^{t}\zeta_{\delta}(\theta_{s}\omega)dsl(\phi))\epsilon\int_{0}^{t}\zeta_{\delta}(\theta_{s}\omega)ds\Delta\phi
		\end{aligned}
	\end{equation}
	and
	\begin{equation}
		\label{eq:4.7}
		v_{\delta,\epsilon}(\tau)=u_0-\epsilon\phi\int_{0}^{\tau}\zeta_{\delta}(\theta_{s}\omega)ds\in H.
	\end{equation}
	According to Hypothesis 1.1, \eqref{eq:4.5} and \eqref{eq:4.7}, for $T>\tau$ with $T>0$, $\omega\in\Omega$, we have
	\begin{equation}
		\begin{aligned}
			\sup_{\delta\rightarrow0^+}|v_{\delta,\epsilon}(\tau)-v_{0,\epsilon}(\tau)|=0.
			\nonumber
		\end{aligned}
	\end{equation}
	\par
	Before proving the convergence relationship of the solutions to equations \eqref{eq:4.1} and \eqref{eq:4.2}, we first give the following estimates.\\
	
	\par 
	\begin{lem}\label{lem:4.1} Let $v_{0,\epsilon}(t,\omega,v_{0,\epsilon}(\tau))$ and $u(t,u_0)$ be solutions of equation \eqref{eq:4.4} and \eqref{eq:4.2}, respectively. For each $\tau\in R$, $\omega\in\Omega$, $u_0\in H$ and $v_{0,\epsilon}(\tau)\in H$, there exists a positive constant $\alpha =\alpha(T,u_0,v_{0,\epsilon}(\tau))>0$ such that
	\begin{equation}
		\label{eq:4.8}
		\sup_{\tau\leq t\leq T}\|v_{0,\epsilon}(t,\omega,v_{0,\epsilon}(\tau))\|^2\leq\alpha,\quad\forall\epsilon>0,
	\end{equation}
	\begin{equation}
		\label{eq:4.9}
		\sup_{\tau\leq t\leq T}\|u(t,u_0)\|^2\leq\alpha.
	\end{equation}
\end{lem}
	\begin{proof}
		By \eqref{eq:2.1}, \eqref{eq:2.4} with $\beta=C_f$ and $p=2$, \eqref{eq:4.4}, the Young and poincar\'e inequalities, we have
		\begin{equation}
			\begin{aligned}
				\label{eq:4.10}
				\|v_{0,\epsilon}(t)\|^2\leq& e^{(m\lambda_1-4C_f)\tau}(\|v_{0,\epsilon}(\tau)\|^2+\int_{\tau}^{t}(2C_f|\mathcal{O}|+\lambda_1C_f|v_{0,\epsilon}(t)|^2\\
				&+\epsilon^2(C_f\lambda_1+\lambda_1C_f^{-1})|W_t|^2|\phi|^2+\frac{2\tilde{m}^2\epsilon^2}{m}|\Delta\phi|^2)e^{(m\lambda_1-4C_f)s}ds).
			\end{aligned}
		\end{equation}
		Thus \eqref{eq:4.8} holds. The proof of \eqref{eq:4.9} is similar and here for brevity we omit it. Then the proof is complete.
	\end{proof}
	Next, we give the approximation between $ u_{0, \epsilon}(t, \omega, u_0) $ and $u(t, u_0)$.\\
	
	\begin{lem}\label{lem:4.2} Assuming Hypothesis 1.1 and Lemma \ref{lem:4.1} hold with
	\begin{equation}
		m>\alpha L_a|l|\lambda_1^{-1}+2L_a|l|+2\eta\lambda_1^{-1},
		\nonumber
	\end{equation}
	where $a(\cdot)$ is supposed to be globally Lipschitz, and the Lipschitz constant is denoted by $L_a$.Then, for each $\omega\in\Omega$ and $u_0\in H $, we have
	\begin{equation}
		\lim_{\epsilon\rightarrow0^+}\sup_{t\in[\tau,T]}|u_{0,\epsilon}(t,\omega,u_0)-u(t,u_0)|^2=0,
	\end{equation}
	where $u_{0,\epsilon}(t,\omega, u_0)$ is the solution of equation \ref{eq:4.1}, and $u(t,u_0)$ is the solution of equation \ref{eq:4.2}.
\end{lem}
	\begin{proof}
		Let $v_{0, \epsilon}(t, \omega, v_{0,\epsilon}(\tau))=v_{0, \epsilon}(t)$ is the solution of equation \ref{eq:4.4}. Since $\omega$ is continuous in $t$, there exists a constant $C_r=C_r(\omega, T)>0$, we have 
		\begin{equation}
			|\omega(t)|\leq C_r,\quad\forall t\in[\tau,T].
			\nonumber
		\end{equation}
		By \eqref{eq:4.4} and \eqref{eq:4.2}, for $t\in[\tau,T]$, $T>\tau$ with $T>0$, $\tau\in R$, we have
		\begin{equation}
			\begin{aligned}
				\label{eq:4.12}
				\frac{d}{dt}|v_{0,\epsilon}(t)-u(t)|^2+2m\|v_{0,\epsilon}(t)-u(t)\|^2=&2\langle a(l(v_{0,\epsilon})+\epsilon W_tl(\phi))\Delta u-a(l(u(t)))\Delta u(t),v_{0, \epsilon}(t)-u(t)\rangle\\
				&+2\left(f(v_{0,\epsilon}+\epsilon\phi W_t)-f(u(t)),v_{0,\epsilon}(t)-u(t)\right)\\
				&+2(\epsilon\phi W_t,v_{0, \epsilon}(t)-u(t))\\
				&+2\langle a(l(v_{0,\epsilon})+\epsilon W_tl(\phi))\epsilon W_t\Delta\phi,v_{0, \epsilon}(t)-u(t)\rangle\\
				&=I_1'+I_2'+I_3'+I_4'.
			\end{aligned}
		\end{equation}
		Since $a$ is globally Lipschitz, denote this Lipschitz constant by $L_a$. By the Young inequality and Poincar\'e inequality, we have 
		\begin{equation}
			\begin{aligned}
				\label{eq:4.13}
				I_1'&\leq(2L_a|l|\lambda_1^{-1}\|u\|^2+2L_a|l|)\|v_{0,\epsilon}(t)-u(t)\|^2+2\epsilon\L_a|l|C_r|v_{0,\epsilon}(t)-u(t)|^2\\
				&+2\epsilon(\tilde{m}C_r|\Delta\phi|^2+L_a|l|C_r\|\phi\|^2\|u\|^2),
			\end{aligned}
		\end{equation}
		and 
		\begin{equation}
			\begin{aligned}
				\label{eq:4.14}
				I_2'&=2\left(f(v_{0,\epsilon}+\epsilon\phi W_t)-f(u(t)),v_{0,\epsilon}(t)-u(t)\right)\\
				&\leq4\eta\lambda_1^{-1}\|v_{0,\epsilon}(t)-u(t)\|^2+2\eta\epsilon C_r(\|\phi\|^2+|v_{0,\epsilon}(t)-u(t)|^2).
			\end{aligned}
		\end{equation}
		For $I'_3$, by the Young inequality, we have,
		\begin{equation}
			\begin{aligned}
				\label{eq:4.15}
				I_3'&=2(\epsilon\phi W_t,v_{0, \epsilon}(t)-u(t))\leq2\epsilon C_r(\|\phi\|^2+|v_{0,\epsilon}(t)-u(t)|^2)
			\end{aligned}
		\end{equation}
		For $I_4'$, by \eqref{eq:2.1} and the Young inequality, we have
		\begin{equation}
			\label{eq:4.16}
			I_4'\leq2\tilde{m}\epsilon C_r(|\Delta\phi|^2+|v_{0,\epsilon}(t)-u(t)|^2).
		\end{equation}
		Combining with $I_1'-I_4'$ and Lemma \ref{lem:4.1}, we get
		\begin{equation}
			\begin{aligned}
				\frac{d}{dt}|v_{0,\epsilon}(t)-u(t)|^2&\leq(-2m+2\alpha L_a|l|\lambda_1^{-1}+2L_a|l|+4\eta\lambda_1^{-1})\|v_{0,\epsilon}(t)-u(t)\|^2\\
				&+2\epsilon\gamma(t)|v_{0,\epsilon}(t)-u(t)|^2+2\epsilon\beta(t),
				\nonumber
			\end{aligned}
		\end{equation}
		with
		\begin{equation}
			\begin{aligned}
				\gamma(t)=&\eta C_r+L_a|l|C_r+\tilde{m}C_r+C_r,
				\nonumber
			\end{aligned}
		\end{equation}
		\begin{equation}
			\begin{aligned}
				\beta(t)=&\alpha L_aC_r|l|\|\phi\|^2+\eta C_r\|\phi\|^2+C_r\|\phi\|^2+\tilde{m}C_r|\Delta\phi|^2.
				\nonumber
			\end{aligned}
		\end{equation}
		\noindent
		Thanks to assumption $m>\alpha L_a|l|\lambda_1^{-1}+2L_a|l|+2\eta\lambda_1^{-1}$, we can obtain 
		\begin{equation}
			\begin{aligned}
				\label{eq:4.17}
				\frac{d}{dt}|v_{0,\epsilon}(t)-u(t)|^2\leq\epsilon\gamma(t)|v_{0,\epsilon}(t)-u(t)|^2+\epsilon\beta(t),
			\end{aligned}
		\end{equation}
		By applying Gronwall's inequality, we have
		\begin{equation}
			\label{eq:4.18}
			|v_{0,\epsilon}(t)-u(t)|^2\leq\epsilon\int_{0}^{t}\beta(s)e^{\epsilon\gamma(s)(t-s)}ds+|v_{0,\epsilon}(\tau)-u(\tau)|^2\rightarrow0,\quad as\ \epsilon\rightarrow0^+.
		\end{equation}
		Thus, we can obtain that
		\begin{equation}
			\label{eq:4.19}
			\lim_{\epsilon\rightarrow0^+}|v_{0,\epsilon}(t,\omega,u_0)-u(t,u_0)|^2=0.
		\end{equation}
		Finally, we observe that
		\begin{equation}
			\label{eq:4.20}
			|u_{0,\epsilon}(t,\omega, u_0)-u(t,u_0)|^2\leq2|v_{0,\epsilon}(t,\omega,u_0)-u(t,u_0)|^2+2\epsilon C_r\|\phi\|^2.
		\end{equation}
		By using \eqref{eq:4.17} and \eqref{eq:4.18}, we complete the proof.
	\end{proof}
	\par 
	The following theorem shows the approximation of $u_{\delta,\epsilon}(t,\omega,u_0)$ and $u(t,u_0)$. \\
	\par 
	\begin{thm}\label{thm:4.3}
Assuming that Hypothesis 1.1, Lemma \ref{lem:4.1} and Lemma \ref{lem:4.2} hold. For each $ \omega\in\Omega$ and $u_0\in H$,
	\begin{equation}
		\lim\limits_{\substack{\delta\rightarrow0^+\\\epsilon\rightarrow0^+}}\sup_{t\in[\tau,T]}|u_{\delta,\epsilon}(t,\omega,u_0)-u(t,u_0)|^2=0,
		\nonumber
	\end{equation}
	where $u_{\delta,\epsilon}(t,\omega, u_0)$ is the solution of equation \eqref{eq:4.3}, and $u(t,u_0)$ is the solution of equation \eqref{eq:4.2}.
		\end{thm}
	\par 
	\begin{proof}
		Let $v_{\delta, \epsilon}(t, \omega, v_{\delta,\epsilon}(\tau))=v_{\delta, \epsilon}(t)$ and $v_{0, \epsilon}(t, \omega, v_{0,\epsilon}(\tau))=v_{0, \epsilon}(t)$ are the solutions to \eqref{eq:4.6} and \eqref{eq:4.4}, respectively. From \eqref{eq:4.6} and \eqref{eq:4.4}, we have
		\begin{equation}
			\begin{aligned}
				\label{eq:4.21}
				&			\frac{d}{dt}|v_{\delta, \epsilon}(t)-v_{0, \epsilon}(t)|^2+2m\|v_{\delta, \epsilon}(t)-v_{0, \epsilon}(t)\|^2
				\\
				&\leq2\langle a(l(v_{\delta,\epsilon}(t)+\epsilon\int_{0}^{t}\zeta_{\delta}(\theta_{s}\omega)dsl(\phi))\Delta v_{0, \epsilon}(t)-a(l(v_{0, \epsilon}(t)+\epsilon W_rl(\phi))\Delta v_{0, \epsilon}(t),v_{\delta,\epsilon}(t)-v_{0,\epsilon}(t)\rangle\\
				&+2(f(v_{\delta, \epsilon}(t)+\epsilon\phi\int_{0}^{t}\zeta_{\delta}(\theta_{s}\omega)ds)-f(v_{0, \epsilon}(t)+\epsilon\phi W_t),v_{\delta,\epsilon}(t)-v_{0,\epsilon}(t))\\
				&+2(\epsilon\phi\int_{0}^{t}\zeta_{\delta}(\theta_t\omega)-\epsilon\phi W_t, v_{\delta,\epsilon}(t)-v_{0,\epsilon}(t))\\
				&+2\langle a(l(v_{\delta,\epsilon}(t)+\epsilon\int_{0}^{t}\zeta_{\delta}(\theta_{s}\omega)dsl(\phi))\epsilon\int_{0}^{t}\zeta_{\delta}(\theta_{s}\omega)ds\Delta\phi-a(l(v_{0, \epsilon}(t)+\epsilon W_t)l(\phi))\epsilon W_t\Delta\phi, v_{\delta,\epsilon}(t)-v_{0,\epsilon}(t)\rangle\\
				&=I_1+I_2+I_3+I_4.
			\end{aligned}
		\end{equation}
		Next, we estimate each term on the right side of the inequality. By (2.2), (1.11) and Young inequality, for any $T>\tau$ with $T>0$, $\tau\in R$, $\omega\in\Omega$ and $u_0\in H$, we have that
		\begin{equation}
			\begin{aligned}
				\label{eq:4.22}
				I_1\leq&\left(2L_a|l|\lambda_1^{-1}\|v_{0,\epsilon}(t)\|^2+2L_a|l|+2\epsilon L_a|l|\sup_{\tau\leq t\leq T}|\int_{0}^{t}\zeta_{\delta}(\theta_{s}\omega)ds-W_t|\right)\|v_{\delta,\epsilon}(t)-v_{0,\epsilon}(t)\|^2\\
				&+2\epsilon L_a|l|\sup_{\tau\leq t\leq T}|\int_{0}^{t}\zeta_{\delta}(\theta_{s}\omega)ds-W_t|\|\phi\|^2\|v_{0,\epsilon}(t)\|^2.
			\end{aligned}
		\end{equation}
		For $I_2$ and $I_3$, by \eqref{eq:1.11}, \eqref{eq:2.2}, the Young and poincar\'e inequalities, we have
		\begin{equation}
				\label{eq:4.23}
			I_2\leq4\eta\lambda_1^{-1}\|v_{\delta, \epsilon}(t)-v_{0, \epsilon}(t)\|^2+2\epsilon\eta\sup_{t\in[\tau,T]}|\int_{0}^{t}\zeta_{\delta}(\theta_{s}\omega)ds-W_t|\left(\|\phi\|^2+|v_{\delta, \epsilon}(t)-v_{0, \epsilon}(t)|^2\right),
		\end{equation}
		and 
		\begin{equation}
				\label{eq:4.24}
			I_3\leq2\epsilon\sup_{t\in[\tau,T]}|\int_{0}^{t}\zeta_{\delta}(\theta_t\omega)-W_t|(\|\phi\|^2+|v_{\delta,\epsilon}(t)-v_{0,\epsilon}(t)|^2).
		\end{equation}
		Next, we estimate $I_4$. Since the function $a$ satisfies the global Lipschitz condition, $l\in\mathcal{L}(L^(\mathcal{O}); R)$, by \eqref{eq:1.11} and the Young inequality, we can deduce that
		\begin{equation}
			\begin{aligned}
					\label{eq:4.25}
				&I_4\leq2\tilde{m}\epsilon\sup_{\tau\leq t\leq T}|\int_{0}^{t}\zeta_{\delta}(\theta_{s}\omega)ds-W_t|(\Delta\phi|^2+ |v_{\delta,\epsilon}(t)-v_{0,\epsilon}(t)|^2)\\
				&+2\epsilon L_a|l|C_r(|v_{\delta, \epsilon}(t)-v_{0, \epsilon}(t)|^2|\Delta\phi|^2+|v_{\delta, \epsilon}(t)-v_{0, \epsilon}(t)|^2)\\
				&+2\epsilon^2 L_a|l|C_r\sup_{t\in[\tau,T]}|\int_{0}^{t}\zeta_{\delta}(\theta_{s}\omega)ds-W_t|(\|\phi\|^2+|\Delta\phi|^2|v_{\delta,\epsilon}(t)-v_{0,\epsilon}(t)|^2).
			\end{aligned}
		\end{equation}
		Based on the analysis of $I_1-I_4$, we have
		\begin{equation}
			\begin{aligned}
					\label{eq:4.26}
				\frac{d}{dt}|v_{\delta, \epsilon}(t)-v_{0, \epsilon}(t)|^2&\leq(-2m+2L_a|l|\|v_{0,\epsilon}(t)\|^2\lambda_1^{-1}+2L_a|l|+4\eta\lambda_1^{-1})\|v_{\delta, \epsilon}(t)-v_{0, \epsilon}(t)\|^2\\
				&+2L_a|l|\epsilon\sup_{t\in[\tau,T]}|\int_{0}^{t}\zeta_{\delta}(\theta_{s}\omega)ds-W_t|\|v_{\delta, \epsilon}(t)-v_{0, \epsilon}(t)\|^2\\
				&+2\epsilon z(t)|v_{\delta, \epsilon}(t)-v_{0, \epsilon}(t)|^2+2\epsilon k(t),
			\end{aligned}
		\end{equation}
		where
		\begin{equation}
			\begin{aligned}
				z(t)=&\sup_{\tau\leq t\leq T}|\int_{0}^{t}\zeta_{\delta}(\theta_{s}\omega)ds-W_t|(\eta+1+\tilde{m}+\epsilon C_rL_a|l||\Delta\phi|^2)+L_a|l|C_r|\Delta\phi|^2+C_rL_a|l|,
				\nonumber
			\end{aligned}
		\end{equation}
		\begin{equation}
			\begin{aligned}
					\label{eq:4.27}
				k(t)=&\sup_{\tau\leq t\leq T}|\int_{0}^{t}\zeta_{\delta}(\theta_{s}\omega)ds-W_t|(L_a|l|\|\phi\|^2\|v_{0,\epsilon}(t)\|^2+\eta\|\phi\|^2+\|\phi\|^2+\tilde{m}|\Delta\phi|^2+\epsilon L_a|l|C_r\|\phi\|^2).
			\end{aligned}
		\end{equation}
		For the first and second terms on the right side of the inequality, we can use the fast that $\sup_{\tau\leq t\leq T}|\int_{0}^{t}\zeta_{\delta}(\theta_{s}\omega)ds-W_t|<1$ for sufficiently small $\epsilon$, $\delta>0$. By Lemma \ref{lem:4.1} and $m>\alpha L_a|l|\lambda_1^{-1}+2L_a|l|+2\eta\lambda_1^{-1}$, for all $t\in[\tau, T ]$ we get that
		\begin{equation}
			\begin{aligned}
					\label{eq:4.28}
				\frac{d}{dt}|v_{\delta, \epsilon}(t)-v_{0, \epsilon}(t)|^2&\leq\epsilon z(t)|v_{\delta, \epsilon}(t)-v_{0, \epsilon}(t)|^2+\epsilon\sup_{\tau\leq t\leq T}|\int_{0}^{t}\zeta_{\delta}(\theta_{s}\omega)ds-W_t|k(t).
			\end{aligned}
		\end{equation}
		Then, by Gronwall's inequality, we have
		\begin{equation}
			\begin{aligned}
					\label{eq:4.29}
				&|v_{\delta, \epsilon}(t)-v_{0, \epsilon}(t)|^2\leq\epsilon\sup_{\tau\leq t\leq T}|\int_{0}^{t}\zeta_{\delta}(\theta_{s}\omega)ds-W_t|\int_{\tau}^{t}e^{\epsilon z(s)(t-s)}k(s)ds+|v_{\delta, \epsilon}(\tau)-v_{0, \epsilon}(\tau)|^2,
			\end{aligned}
		\end{equation}
		and $|v_{\delta, \epsilon}(\tau)-v_{0, \epsilon}(\tau)|^2\rightarrow0$ as $\delta\rightarrow0^+$, $\epsilon\rightarrow0^+$, for all $t\in [\tau, T]$. Then,
		along with \eqref{eq:1.17} implies that
		\begin{equation}
				\label{eq:4.30}
			\lim\limits_{\substack{\delta\rightarrow0^+\\\epsilon\rightarrow0^+}}|v_{\delta,\epsilon}(t,\omega,v_{\delta,\epsilon}(\tau))-v_{0, \epsilon}(t,\omega,v_{0,\epsilon}(\tau))|^2=0.
		\end{equation}
		Since that
		\begin{equation}
			\begin{aligned}
					\label{eq:4.31}
				|u_{\delta,\epsilon}(t,\omega,u_0)-u_{0, \epsilon}(t,\omega,u_0)|^2&\leq2|v_{\delta,\epsilon}(t,\omega,v_{\delta,\epsilon}(\tau))-v_{0, \epsilon}(t,\omega,v_{0,\epsilon}(\tau))|^2\\
				&+2\|\phi\|^2\times\sup_{t\in[\tau,T]}|\int_{0}^{t}\zeta_{\delta}(\theta_{s}\omega)ds-W_t|^2,
			\end{aligned}
		\end{equation}
		we observe that
		\begin{equation}
				\label{eq:4.32}
			|u_{\delta,\epsilon}(t,\omega,u_0)-u(t,u_0)|^2\leq2|u_{\delta,\epsilon}(t,\omega,u_0)-u_{0,\epsilon}(t,\omega,u_0)|^2+2|u_{0,\epsilon}(t,\omega,u_0)-u(t,u_0)|^2.
		\end{equation}
		By using Lemma \ref{lem:4.2} and \eqref{eq:4.30}, we can get that 
		\begin{equation}
				\label{eq:4.33}
			\lim\limits_{\substack{\delta\rightarrow0^+\\\epsilon\rightarrow0^+}}\sup_{t\in[\tau,T]}|u_{\delta,\epsilon}(t,\omega,u_0)-u(t,u_0)|^2=0.
		\end{equation}
	\end{proof}
	\begin{rem}\label{rem:4.1} By \eqref{eq:1.11}, \eqref{eq:4.29}, \eqref{eq:4.31}, we clearly that for any $T>\tau$ with $T>0$, $\tau\in R$, $\epsilon>0$, $\omega\in\Omega$ and $u_0\in H$, the solution of equation \eqref{eq:4.3} are uniformly convergent to the solution of equation \eqref{eq:4.1} on $[\tau,T]$ as $\delta\rightarrow0^+$, i.e.,
	\begin{equation}
		\lim\limits_{\substack{\delta\rightarrow0^+}}|u_{\delta,\epsilon}(t,\omega,u_0)-u_{0, \epsilon}(t,\omega,u_0)|^2=0.
		\nonumber
	\end{equation}
	\end{rem}
	\par 
	\textbf{4.2. Random conjugate equation}. Next, we define a random variable
	\begin{equation}
		x_{0,\epsilon}^*(\theta_{t}\omega) :=\epsilon\int_{-\infty}^{t}e^{-\eta(t-s)}\phi dW_s,
		\nonumber
	\end{equation}
	is a stationary solution of the linear stochastic differential equation:
	\begin{equation}
		\label{eq:4.34}
		dx_{0,\epsilon}+\eta x_{0,\epsilon}dt=\epsilon\phi dW.
	\end{equation}
	On account of the change og variable $p_{0,\epsilon}(t)=u_{0,\epsilon}(t)-x_{0,\epsilon}^*(\theta_{t}\omega)$, equation \eqref{eq:4.1} can be written as
	\begin{equation}
		\begin{aligned}
			\label{eq:4.35}
			\frac{\partial{p_{0,\epsilon}}}{\partial{t}}=&a(l(v_{0,\epsilon}+x_{0,\epsilon}^*(\theta_{t}\omega))\Delta (p_{0,\epsilon}+x_{0,\epsilon}^*(\theta_{t}\omega))+f(p_{0,\epsilon}+x_{0,\epsilon}^*(\theta_{t}\omega))+x_{0,\epsilon}^*(\theta_{t}\omega),\quad p_{0,\epsilon}(\tau)=u_0-x_{0,\epsilon}^*(\theta_{\tau}\omega).
		\end{aligned}
	\end{equation}
	Similarly, we define a random variable
	\begin{equation}
		\label{eq:4.36}
		x_{\delta,\epsilon}^*(\theta_{t}\omega) := \epsilon\int_{-\infty}^{t}e^{-\eta(t-s)}\phi\zeta_{\delta}(\theta_s\omega)ds,
	\end{equation}
	is a stationary solution of the random differential equation:
	\begin{equation}
		\frac{dx_{\delta,\epsilon}}{dt}+\eta x_{\delta,\epsilon}=\epsilon\phi\zeta_{\delta}(\theta_{t}\omega),
		\nonumber
	\end{equation}
	and for each $\omega\in\Omega$, the following properties hold.\\
	There exists a positive constant $\bar{\delta}$ such that
	\begin{equation}
		\label{eq:4.37}
		\lim_{t\rightarrow\pm\infty}\frac{|x^*_{\delta,\epsilon}(\theta_{t}\omega)|}{|t|}=0,\quad\lim_{t\rightarrow\pm\infty}\frac{1}{t}\int_{0}^{t}x^*_{\delta,\epsilon}(\theta_r\omega)dr=0.
	\end{equation}
	For arbitrarily $\epsilon>0$ and $\delta\ge0$, we give the random transformation
	\begin{equation}
		\label{eq:4.38}
		p_{\delta,\epsilon}(t)=u_{\delta,\epsilon}(t)-x_{\delta,\epsilon}^*(\theta_{t}\omega),
	\end{equation}
	taking into \eqref{eq:4.3}, we have
	\begin{equation}
		\begin{aligned}
			\label{eq:4.39}
			\frac{\partial{p_{\delta,\epsilon}}}{\partial{t}}=&a(l(p_{\delta,\epsilon}+x_{\delta,\epsilon}^*(\theta_{t}\omega))\Delta (p_{\delta,\epsilon}+x_{\delta,\epsilon}^*(\theta_{t}\omega))+f(p_{\delta,\epsilon}+x_{\delta,\epsilon}^*(\theta_{t}\omega))\\
			&+x_{\delta,\epsilon}^*(\theta_{t}\omega),\quad p_{\delta,\epsilon}(\tau)=u_0-x_{\delta,\epsilon}^*(\theta_{\tau}\omega).
		\end{aligned}
	\end{equation}
	By the same way in \cite[Theorem 7]{J. Xu}, we are can prove that equation \eqref{eq:4.3} with initial value $p_{\delta,\epsilon}(\tau)\in H$ and Dirichlet boundary condition possesses a unique weak solution, for every $T>\tau$ with $T>0$.
	\begin{equation}
		p_{\delta, \epsilon}(\cdot; \tau, \omega, p_{\delta,\epsilon}(\tau))\in C(\tau, T; H) \cap L^2(\tau, T; V).
		\nonumber
	\end{equation}
	Furthermore, this solution is continuous in $H$.
	\begin{equation}
		p_{\delta, \epsilon}(\cdot; \tau, \omega, p_{\delta,\epsilon}(\tau))\in C(\tau, T; V) \cap L^2(\tau, T; V\cap H^2(\mathcal{O})).
		\nonumber
	\end{equation}
	For any $p_{\delta,\epsilon}(\tau)\in H$, $\tau\in R$, $\omega\in\Omega$, $\epsilon>0$, $\delta\ge0$, we define the mapping $\Xi_{\delta,\epsilon}:\Bbb R^+\times\Omega\times H\rightarrow H$ such that
	\begin{equation}
		\label{eq:4.40}
		\Xi_{\delta,\epsilon}=p_{\delta, \epsilon}(t; \tau, \omega, p_{\delta,\epsilon}(\tau)).
	\end{equation}
	Clearly, there is a mapping $\Psi_{\delta,\epsilon}:\Bbb R^+\times\Omega\times H\rightarrow H$ satisfying
	\begin{equation}
		\begin{aligned}
			\label{eq:4.41}
			\Psi_{\delta,\epsilon}&=u_{\delta,\epsilon}(t;0,\omega,u_{\delta,\epsilon}(\tau))\\
			&=p_{\delta, \epsilon}(t;0, \omega,u_0-x_{\delta,\epsilon}^*(\omega))+x_{\delta,\epsilon}^*(\theta_t\omega)\quad\forall u_0\in H,\quad\forall\omega\in\Omega.
		\end{aligned}
	\end{equation}
	\par 
	\textbf{4.3. Convergence of attractors.} In what follows, we prove the convergence of the random attractors of equation \eqref{eq:4.3} to those of equation \eqref{eq:4.2} as $\delta\rightarrow0^+$ and $\epsilon\rightarrow0^+$. To that end, we first show existence of random attractors of equation \eqref{eq:4.1} and equation \eqref{eq:4.3},  and show these random attractors of equation \eqref{eq:4.3} converge to the ones of equation \eqref{eq:4.2} as $\delta\rightarrow0^+$ and $\epsilon\rightarrow0^+$.\par
	Before proving the convergence of the attractors, it is necessary to study the approximation of stationary noises.
	\par 
	\begin{lem}\label{lem:4.5} Under the further assuming that Hypothesis 1.2 hold, for almost all $\omega\in\Omega$, $t\leq0$,
	\begin{equation}
		\lim\limits_{\substack{\delta\rightarrow0^+\\\epsilon\rightarrow0^+}}|x_{\delta,\epsilon}^*(\theta_{t}\omega)-x_{0,\epsilon}^*(\theta_{t}\omega)|=0,
		\nonumber
	\end{equation}
	and
	\begin{equation}
		\lim_{\epsilon\rightarrow0^+}|x_{0,\epsilon}^*(\theta_{t}\omega)|=0.
		\nonumber
	\end{equation}
		\end{lem}
	\begin{proof}
		By Hypothesis 1.2, for any $\varepsilon_0>0$, there exists a $T_1 = T_1(\omega)>0$ such that for $t\leq-T_1$ and $\delta\in(0, \tilde{\delta}_1]$, we get that
		\begin{equation}
			\label{eq:4.42}
			\left|\frac{1}{t}\left[\int_{0}^{t}\zeta_{\delta}(\theta_{s}\omega)ds-\omega(t)\right]\right|<\varepsilon_0.
		\end{equation}
		Thus there exists a $\delta_0\in(0, \tilde{\delta}_1)$ such that for any $t\leq0$, we have
		\begin{equation}
			\left|\int_{0}^{t}\zeta_{\delta}(\theta_{s}\omega)ds-\omega(t)\right|<\varepsilon_0.
			\nonumber
		\end{equation}
		For all $\delta\in(0, 1]$, using integration by part, we obtain 
		\begin{equation}
			\begin{aligned}
				x_{\delta,\epsilon}^*(\theta_{t}\omega)-x_{0,\epsilon}^*(\theta_{t}\omega)&=\epsilon\int_{-\infty}^{t}e^{-\eta(t-s)}\phi d\left[\int_{0}^{s}\zeta_{\delta}(\theta_{r}\omega)dr-\omega(s)\right]\\
				&=\epsilon\phi\left[\int_{0}^{t}\zeta_{\delta}(\theta_{r}\omega)dr-\omega(t)\right]+\epsilon\int_{-\infty}^{t}e^{-\eta(t-s)}\eta\phi \left[\int_{0}^{s}\zeta_{\delta}(\theta_{r}\omega)dr-\omega(s)\right]ds.
				\nonumber
			\end{aligned}
		\end{equation}
		By \eqref{eq:4.42}, for $t\in[-T_1,0]$, we have
		\begin{equation}
			\begin{aligned}
				\label{eq:4.43}
				&\left|\int_{-\infty}^{t}e^{-\eta(t-s)}\eta\phi \left[\int_{0}^{s}\zeta_{\delta}(\theta_{r}\omega)dr-\omega(s)\right]ds\right|\\
				&\leq\varepsilon_0\eta\|\phi\|\int_{-\infty}^{-T_2}e^{-\eta(t-s)}(-s)ds+\varepsilon_0\sup_{s\in[-T_2,0]}\left|\int_{0}^{s}\zeta_{\delta}(\theta_{r}\omega)dr-\omega(s)\right|\int_{-T_2}^{t}e^{-\eta(t-s)}\eta\|\phi\|ds\\
				&\leq\varepsilon_0\|\phi\|\left[\eta^{-1}\Gamma(2)+|t|\times\Gamma(1)\right]+\|\phi\|\sup_{s\in[-T_2,0]}\left|\int_{0}^{s}\zeta_{\delta}(\theta_{r}\omega)dr-\omega(s)\right|\times\Gamma(1).
			\end{aligned}
		\end{equation}
		For $t<-T_1$, by \eqref{eq:4.42} we can obtain that
		\begin{equation}
			\begin{aligned}
				\label{eq:4.44}
				&\left|\int_{-\infty}^{t}e^{-\eta(t-s)}\eta\phi \left[\int_{0}^{s}\zeta_{\delta}(\theta_{r}\omega)dr-\omega(s)\right]ds\right|\leq\varepsilon_0\|\phi\|\left[\eta^{-1}\Gamma(2)+|t|\times\Gamma(1)\right].
			\end{aligned}
		\end{equation}
		Then we obtain
		\begin{equation}
			\begin{aligned}
				\label{eq:4.45}
				|x_{\delta,\epsilon}^*(\theta_{t}\omega)-x_{0,\epsilon}^*(\theta_{t}\omega)|&\leq\epsilon\|\phi\|\left|\int_{0}^{t}\zeta_{\delta}(\theta_{r}\omega)dr-\omega(t)\right|+2\epsilon\varepsilon_0\|\phi\|\left[\eta^{-1}\Gamma(2)+|t|\times\Gamma(1)\right]\\
				&+\epsilon\|\phi\|\times\sup_{r\in[-T_2,0]}\left|\int_{0}^{t}\zeta_{\delta}(\theta_{r}\omega)ds-\omega(t)\right|\times\Gamma(1).
			\end{aligned}
		\end{equation}
		Taking first $\varepsilon_0\rightarrow0^+$, then letting $\delta\rightarrow0^+$ and $\epsilon\rightarrow0^+$, we can find the first result is true. 
		By integration by parts and \eqref{eq:1.6}, we obtain that
		\begin{equation}
			\begin{aligned}
				\label{eq:4.46}
				|x_{0,\epsilon}^*(\theta_{t}\omega)|&=\epsilon\left|\phi W_t-\int_{-\infty}^{t}e^{-\eta(t-\tau)}\eta\phi W_{\tau}d\tau\right|\\
				&\leq\epsilon\|\phi\| C_{\omega}\left[|t|+1+\Gamma(1)+\eta^{-1}\Gamma(2)+|t|\times\Gamma(1)\right]\rightarrow 0,\quad as\ \epsilon\rightarrow0^+.
			\end{aligned}
		\end{equation}
		The proof is completed.
	\end{proof}

	\par 
	Next, we show the existence of $\mathcal{D}_F$-pullback absorbing stes of equation \eqref{eq:4.4} and equation \eqref{eq:4.3} in $H$.
	\par 
	\begin{thm}\label{thm:4.5} Suppose that $a$ is globally Lipschitz and fulfills \eqref{eq:2.1}, $f\in C(\Bbb R)$ satisfies \eqref{eq:2.2} and \eqref{eq:2.4} with $p=2$ and $\beta=C_f$, $\phi\in V \cap H^2(\mathcal{O})$, and $l\in L^2(\mathcal{O})$. Also, let m$\lambda_1> 4C_f$. Then, there exists a $\epsilon_0\in(0,1)$ such that for all $\epsilon\in(0,\epsilon_0]$, (4.1) has a random $\mathcal{D}_F$-attractor $\mathcal{A}_{0,\epsilon}(\omega )$ (where $\mathcal{D}_F$ is the universe of fixed bounded sets) for the dynamical system $\Psi_{0,\epsilon}(t, \omega, u_0)$. In addition, the $\mathcal{D}_F$-pullback absorbing set $B_{0,\epsilon}=\{B_{0,\epsilon}(\omega):\omega\in\Omega \}\in\mathcal{D}$ in $H$ is given by
	\begin{equation}
		B_{0,\epsilon}(\omega)=\{u\in H :|u|^2\leq\lambda_1^{-1}R_{0,\epsilon}(\omega)\}
		\nonumber
	\end{equation}
	with
	\begin{equation}
		\begin{aligned}
			R_{0,\epsilon}(\omega)=&2|x_{0,\epsilon}^*(\omega)|^2+\frac{8C_f|\mathcal{O}|}{m(m\lambda_1-4C_f)}+\frac{4\lambda_1C_f^2|
				\mathcal{O}|}{(m\lambda_1-4C_f)^2}\\
			&+\frac{4+2\lambda_1C_fm+m\lambda_1-4C_f+2C_f|\mathcal{O}|}{m(m\lambda_1-4C_f)}\\
			&+(4m^{-1}+2\lambda_1C_f)\int_{-\infty}^{0}e^{(m\lambda_1-4C_f)s}\left(\frac{|x_{0,\epsilon}^*(\theta_{s}\omega)|^2}{\lambda_1C_f}+\frac{2C_f|x_{0,\epsilon}^*(\theta_{s}\omega)|^2}{\lambda_1}+\frac{2\tilde{m}^2}{m}\right)ds\\
			&+2\int_{-1}^{0}e^{(m\lambda_1-4C_f)s}\left(\lambda_1C_f|\mathcal{O}|+(C_f\lambda_1+\lambda_1C_f^{-1})|x_{0,\epsilon}^*(\theta_{s}\omega)|^2+\frac{\tilde{m}^2}{m}\right)ds.
			\nonumber
		\end{aligned}
	\end{equation}
		\end{thm}
	\begin{proof}
		The proof is similar to \cite[Theorem 9]{J. Xu} and we omit the details here.
	\end{proof}
	\par 
	\begin{thm}\label{thm:4.6} Assume the conditions in Theorem \ref{thm:4.5} hold. Then, there exists $\tilde{\delta} >0$ and $\epsilon_0\in(0,1]$ such that for all $0<\delta<\tilde{\delta}$, $\epsilon\in(0,\epsilon_0]$, \eqref{eq:4.3} has a random $\mathcal{D}_F$ -attractor $\mathcal{A}_{\delta,\epsilon}(\omega )$
	associated to the dynamical system $\Psi_{\delta,\epsilon}(t, \omega, u_0)$. In addition, the $\mathcal{D}_F$ -pullback absorbing set $B_{\delta,\epsilon}:=\{ B_{\delta,\epsilon}(\omega ): \omega\in\Omega\}\in\mathcal{D}$ in $H$ is given by
	\begin{equation}
		B_{\delta,\epsilon}(\omega)=\{u\in H :|u|^2\leq\lambda_1^{-1}R_{\delta,\epsilon}(\omega)\},
		\nonumber
	\end{equation}
	with
	\begin{equation}
		\begin{aligned}
			R_{\delta,\epsilon}(\omega)=&2|x_{\delta,\epsilon}^*(\omega)|^2+\frac{8C_f|\mathcal{O}|}{m(m\lambda_1-4C_f)}+\frac{4\lambda_1C_f^2|
				\mathcal{O}|}{(m\lambda_1-4C_f)^2}\\
			&+\frac{4+2\lambda_1C_fm+m\lambda_1-4C_f+2C_f|\mathcal{O}|}{m(m\lambda_1-4C_f)}\\
			&+(4m^{-1}+2\lambda_1C_f)\int_{-\infty}^{0}e^{(m\lambda_1-4C_f)s}\left(\frac{|x_{\delta,\epsilon}^*(\theta_{t}\omega)|^2}{\lambda_1C_f}+\frac{2C_f|x_{\delta,\epsilon}^*(\theta_t\omega)|^2}{\lambda_1}+\frac{2\tilde{m}^2}{m}\right)ds\\
			&+2\int_{-1}^{0}e^{(m\lambda_1-4C_f)s}\left(\lambda_1C_f|\mathcal{O}|+(C_f\lambda_1+\lambda_1C_f^{-1})|x_{\delta,\epsilon}^*(\theta_{s}\omega)|^2+\frac{\tilde{m}^2}{m}\right)ds.
			\nonumber
		\end{aligned}
	\end{equation}
	Moreover, for $\omega\in\Omega$,
	\begin{equation}
		\label{eq:4.47}
		\lim\limits_{\substack{\delta\rightarrow0^+}}R_{\delta,\epsilon}(\omega)=R_{0,\epsilon}(\omega),
	\end{equation}
	where $R_{0,\epsilon}(\omega)$ are given in Theorem \ref{thm:4.5}.
		\end{thm}
	\begin{proof} We prove the theorem in three steps.
		\par 
		\textbf{Step 1.} Multiplying \eqref{eq:4.39} by $p_{\delta, \epsilon}(t) := p_{\delta,\epsilon}(t; \tau, \omega, p_{\delta,\epsilon}(\tau))$ in $H$, we can get 
		\begin{equation}
			\begin{aligned}
				\label{eq:4.48}
				\frac{d}{dt}|p_{\delta,\epsilon}(t)|^2+2m\|p_{\delta,\epsilon}(t)\|^2&\leq2(f(p_{\delta,\epsilon}(t)+x_{\delta,\epsilon}^*(\theta_{t}\omega)),p_{\delta,\epsilon}(t))+2(x_{\delta,\epsilon}^*(\theta_{t}\omega),p_{\delta,\epsilon}(t))\\
				&+2\tilde{m}\|p_{\delta,\epsilon}(t)\|.
			\end{aligned}
		\end{equation}
		Next, we estimate each term on the right-hand side of the inequality. For the first term of the right-hand side of \eqref{eq:4.48}, by \eqref{eq:2.2}-\eqref{eq:2.4} and Young inequality, we get 
		\begin{equation}
			\begin{aligned}
				\label{eq:4.49}
				2(f(p_{\delta,\epsilon}(t)+x_{\delta,\epsilon}^*(\theta_{t}\omega)),p_{\delta,\epsilon}(t))\leq\frac{C_f}{\alpha_1}|\mathcal{O}|+\frac{C_f}{\alpha_1\lambda_1}|x_{\delta,\epsilon}^*(\theta_{t}\omega)|^2+2C_f(\alpha_1+1)|p_{\delta,\epsilon}(t)|^2.
			\end{aligned}
		\end{equation} 
		By Young inequality, we have
		\begin{equation}
			\begin{aligned}
				\label{eq:4.50}
				2(x_{\delta,\epsilon}^*(\theta_{t}\omega),p_{\delta,\epsilon}(t))\leq\frac{1}{\alpha_2\lambda_1}|x_{\delta,\epsilon}^*(\theta_t\omega)|^2+\alpha_2|p_{\delta,\epsilon}(t)|^2,
			\end{aligned}
		\end{equation}
		and 
		\begin{equation}
			\label{eq:4.51}
			2\tilde{m}\|p_{\delta,\epsilon}(t)\|\leq\frac{\tilde{m}^2}{\alpha_3}+\alpha_3\|p_{\delta,\epsilon}(t)\|^2.
		\end{equation}
		By \eqref{eq:4.49}-\eqref{eq:4.51}, and letting $\alpha_1=\frac{1}{2}$, $\alpha_2=C_f$, $\alpha_3=\frac{m}{2}$, we have
		\begin{equation}
			\begin{aligned}
				\label{eq:4.52}
				\frac{d}{dt}|p_{\delta,\epsilon}(t)|^2&\leq-(m\lambda_1-4C_f)|p_{\delta,\epsilon}(t)|^2+2C_f|\mathcal{O}|\\
				&+\left(\frac{|x_{\delta,\epsilon}^*(\theta_{t}\omega)|^2}{\lambda_1C_f}+\frac{2C_f|x_{\delta,\epsilon}^*(\theta_{t}\omega)|^2}{\lambda_1}+\frac{2\tilde{m}^2}{m}\right)-\frac{m}{2}\|p_{\delta,\epsilon}(t)\|^2.
			\end{aligned}
		\end{equation}
		By Gronwall's inequality in $[t_0, -1]$ with $t_0\leq-1$, for $\omega\in\Omega$, we get from \eqref{eq:4.52} that
		\begin{equation}
			\begin{aligned}
				|p_{\delta,\epsilon}(-1)|^2&\leq e^{(m\lambda_1-4C_f)(t_0+1)}|p_{\delta,\epsilon}(t_0)|^2\\
				&+\int_{t_0}^{-1}e^{(m\lambda_1-4C_f)(t+1)}\left(2C_f|\mathcal{O}|+\frac{|x_{\delta,\epsilon}^*(\theta_{t}\omega)|^2}{\lambda_1C_f}+\frac{2C_f|x_{\delta,\epsilon}^*(\theta_{t}\omega)|^2}{\lambda_1}+\frac{2\tilde{m}^2}{m}\right)dt.
				\nonumber		
			\end{aligned}
		\end{equation}
		Therefore, for a given $B_{\delta,\epsilon}(0, \rho_{\delta,\epsilon})\subset H$, there exists $T(\omega, \rho_{\delta,\epsilon})\leq-1$, such that for all $t_0\leq T(\omega,\rho_{\delta,\epsilon})$ and for all $u_0\in B_{\delta,\epsilon}(0,\rho_{\delta,\epsilon})$,
		\begin{equation}
			|p_{\delta,\epsilon}(-1; t_0, \omega, u_{\delta,\epsilon}(t_0)-x_{\delta,\epsilon}(\theta_{t_0}(\omega))|^2\leq r_{\delta,\epsilon}^2(\omega)
			\nonumber
		\end{equation}
		whith
		\begin{equation}
			r_{\delta,\epsilon}^2=1+\frac{2C_f|\mathcal{O}|}{m\lambda_1-4C_f}+\int_{-\infty}^{-1}e^{-(m\lambda_1-4C_f)(t+1)}\left(\frac{|x_{\delta,\epsilon}^*(\theta_{t}\omega)|^2}{\lambda_1C_f}+\frac{2C_f|x_{\delta,\epsilon}^*(\theta_{t}\omega)|^2}{\lambda_1}+\frac{2\tilde{m}^2}{m}\right)dt.
			\nonumber
		\end{equation}
		In fact, it is enough to choose $T(\omega, \rho_{\delta,\epsilon})$ such that, for any $t_0\leq T(\omega,\rho_{\delta,\epsilon})$, we have
		\begin{equation}
			\begin{aligned}
				e^{(m\lambda_1-4C_f)(t_0+1)}|p_{\delta,\epsilon}(t_0)|^2&=e^{(m\lambda_1-4C_f)(t_0+1)}|u_{\delta,\epsilon}(t_0)-x_{\delta,\epsilon}^*(\theta_{t_0}\omega)|^2\\
				&\leq 2e^{(m\lambda_1-4C_f)(t_0+1)}(|\rho_{\delta,\epsilon}^2|+|x_{\delta,\epsilon}^*(\theta_{t_0}\omega)|^2)\\
				&\leq1.
				\nonumber
			\end{aligned}
		\end{equation}
		\par 
		Next, we need to prove $p_{\delta,\epsilon}\in L^{\infty}([-1,t]; H)\bigcap L^2([-1,t]; V)$ with $t\in[-1,0]$ by energy estimations. From \eqref{eq:4.52}, for $t\in[-1, 0]$, we have
		\begin{equation}
			\begin{aligned}
				|p_{\delta,\epsilon}(t)|^2&\leq e^{-(m\lambda_1-4C_f)(t+1)}|p_{\delta,\epsilon}(-1)|^2+\frac{2C_f|\mathcal{O}|}{m\lambda_1-4C_f}\\
				&+\int_{-1}^{t}e^{-(m\lambda_1-4C_f)(s-t)}\left(\frac{|x_{\delta,\epsilon}^*(\theta_{s}\omega)|^2}{\lambda_1C_f}+\frac{2C_f|x_{\delta,\epsilon}^*(\theta_{s}\omega)|^2}{\lambda_1}+\frac{2\tilde{m}^2}{m}\right)ds.
				\nonumber		
			\end{aligned}
		\end{equation}
		Therefore,
		\begin{equation}
			\begin{aligned}
				\int_{-1}^{0}e^{(m\lambda_1-4C_f)s}\|p_{\delta,\epsilon}(s)\|^2ds&\leq \frac{2}{m}e^{-(m\lambda_1-4C_f)}|p_{\delta,\epsilon}(-1)|^2+\frac{4C_f|\mathcal{O}|}{m(m\lambda_1-4C_f)}\\
				&+\frac{2}{m}\int_{-1}^{0}e^{(m\lambda_1-4C_f)s}\left(\frac{|x_{\delta,\epsilon}^*(\theta_{s}\omega)|^2}{\lambda_1C_f}+\frac{2C_f|x_{\delta,\epsilon}^*(\theta_{s}\omega)|^2}{\lambda_1}+\frac{2\tilde{m}^2}{m}\right)ds.
				\nonumber	
			\end{aligned}
		\end{equation}
		Thus, for a given $B(0, \rho_{\delta,\epsilon})\subset H$ we conclude that there exists $T(\omega, \rho_{\delta,\epsilon})\leq-1$, such that for all $t_0\leq T(\omega, \rho_{\delta,\epsilon})$ and for all $u_0\in B_{\delta,\epsilon}(0, \rho_{\delta,\epsilon})$,
		\begin{equation}
			\begin{aligned}
				\label{eq:4.53}
				|p_{\delta,\epsilon}(t)|^2&\leq e^{-(m\lambda_1-4C_f)(t+1)}|r_{\delta,\epsilon}^2(\omega)|+\frac{2C_f|\mathcal{O}|}{m\lambda_1-4C_f}\\
				&+\int_{-1}^{t}e^{-(m\lambda_1-4C_f)(t-s)}\left(\frac{|x_{\delta,\epsilon}^*(\theta_{s}\omega)|^2}{\lambda_1C_f}+\frac{2C_f|x_{\delta,\epsilon}^*(\theta_{s}\omega)|^2}{\lambda_1}+\frac{2\tilde{m}^2}{m}\right)ds.
			\end{aligned}
		\end{equation}
		and
		\begin{equation}
			\begin{aligned}
				\int_{-1}^{0}e^{(m\lambda_1-4C_f)s}\|p_{\delta,\epsilon}(s)\|^2ds&\leq \frac{2}{m}e^{-(m\lambda_1-4C_f)}|r_{\delta,\epsilon}^2(\omega)|+\frac{4C_f|\mathcal{O}|}{m(m\lambda_1-4C_f)}\\
				&+\frac{2}{m}\int_{-1}^{0}e^{(m\lambda_1-4C_f)s}\left(\frac{|x_{\delta,\epsilon}^*(\theta_{s}\omega)|^2}{\lambda_1C_f}+\frac{2C_f|x_{\delta,\epsilon}^*(\theta_{s}\omega)|^2}{\lambda_1}+\frac{2\tilde{m}^2}{m}\right)ds.
				\nonumber	
			\end{aligned}
		\end{equation}
		\par 
		\textbf{Step 2.} We need to obtain a bounded absorbing set in $V$, we multiply equation \eqref{eq:4.39} by $-\Delta p_{\delta,\epsilon}(t)$, with the help of \eqref{eq:2.1}, \eqref{eq:2.4}, and the Poincar\'e and Young inequalities, we can get that
		\begin{equation}
			\begin{aligned}
				\label{eq:4.54}
				\frac{d}{dt}\|p_{\delta,\epsilon}(t)\|^2&\leq(-m\lambda_1+4C_f)\|p_{\delta,\epsilon}(t)\|^2+\lambda_1C_f|\mathcal{O}|+\lambda_1C_f|p_{\delta,\epsilon}(t)|^2\\
				&+\left(C_f\lambda_1+\frac{\lambda_1}{C_f}\right)|x_{\delta,\epsilon}^*(\theta_{t}\omega)|^2+\frac{\tilde{m}^2}{m}.
			\end{aligned}
		\end{equation}
		By integrating the above inequality between $s$ and 0, where $s\in[-1, 0]$, we obtain
		\begin{equation}
			\begin{aligned}
				\|p_{\delta,\epsilon}(0)\|^2&\leq e^{(m\lambda_1-4C_f)s}\|p_{\delta,\epsilon}(s)\|^2\\
				&+\int_{s}^{0}e^{(m\lambda_1-4C_f)t}\left(\lambda_1C_f|\mathcal{O}|+\lambda_1C_f|p_{\delta,\epsilon}(t)|^2+\left(C_f\lambda_1+\frac{\lambda_1}{C_f}\right)|x_{\delta,\epsilon}^*(\theta_{t}\omega)|^2+\frac{\tilde{m}^2}{m}\right)dt.
				\nonumber
			\end{aligned}
		\end{equation}
		By integrating again the above inequality in $[-1,0]$, combining with the above inequality, we have
		\begin{equation}
			\begin{aligned}
				\|p_{\delta,\epsilon}(0)\|^2&\leq\frac{2}{m}e^{-(m\lambda_1-4C_f)}r_{\delta,\epsilon}^2(\omega)+\frac{4C_f|\mathcal{O}|}{m(m\lambda_1-4C_f)}\\
				&+\frac{2}{m}\int_{-1}^{0}e^{(m\lambda_1-4C_f)s}\left(\frac{|x_{\delta,\epsilon}^*(\theta_{s}\omega)|^2}{\lambda_1C_f}+\frac{2C_f|x_{\delta,\epsilon}^*(\theta_{s}\omega)|^2}{\lambda_1}+\frac{2\tilde{m}^2}{m}\right)ds\\
				&+\int_{-1}^{0}e^{(m\lambda_1-4C_f)t}\left(\lambda_1C_f|\mathcal{O}|+\lambda_1C_f|p_{\delta,\epsilon}(t)|^2+\left(C_f\lambda_1+\frac{\lambda_1}{C_f}\right)|x_{\delta,\epsilon}^*(\theta_{t}\omega)|^2+\frac{\tilde{m}}{m}\right)dt.
				\nonumber
			\end{aligned}
		\end{equation}
		Therefore, there exists $\tilde{r}_{\delta,\epsilon}(\omega)$ and $T(\omega,\rho_{\delta,\epsilon})\leq-1$, for a given $\rho_{\delta,\epsilon}>0$, and all $t_0\leq T(\omega,\rho_{\delta,\epsilon})$ and $|u_{0}|\leq\rho_{\delta,\epsilon}$, we have
		\begin{equation}
			\|u_{\delta,\epsilon}(0;t_0,\omega,u_{0}\|^2=\|(p_{\delta,\epsilon}(0;t_0,\omega,u_{0})-x_{\delta,\epsilon}^*(\theta_{t_0}\omega))+x_{\delta,\epsilon}^*(\omega)\|^2\leq\tilde{r}_{\delta,\epsilon}^2(\omega),
			\nonumber
		\end{equation}
		where
		\begin{equation}
			\begin{aligned}
				\tilde{r}_{\delta,\epsilon}^2(\omega)&=2|x_{\delta,\epsilon}(\omega)|^2+\left(\frac{4}{m}+2\lambda_1C_f\right)r_{\delta,\epsilon}^2(\omega)+\frac{8C_f|\mathcal{O}|}{m(m\lambda_1-4C_f)}+\frac{4\lambda_1C_f^2|\mathcal{O}|}{(m\lambda_1-4C_f)^2}\\
				&+\left(\frac{4}{m}+2\lambda_1C_f\right)\int_{-\infty}^{0}e^{(m\lambda_1-4C_f)s}\left(\frac{|x_{\delta,\epsilon}^*(\theta_{s}\omega)|^2}{\lambda_1C_f}+\frac{2C_f|x_{\delta,\epsilon}^*(\theta_{s}\omega)|^2}{\lambda_1}+\frac{2\tilde{m}^2}{m}\right)ds\\
				&+2\int_{-1}^{0}e^{(m\lambda_1-4C_f)s}\left(\lambda_1C_f|\mathcal{O}|+\left(C_f\lambda_1+\frac{\lambda_1}{C_f}\right)|x_{\delta,\epsilon}^*(\theta_{s}\omega)|^2+\frac{\tilde{m}^2}{m}\right)ds.
				\nonumber
			\end{aligned}
		\end{equation}
		Thus, according to \cite[Theorem 1]{J. Xu}, there exists a unique random attractor
		$\mathcal{A}_{\delta,\epsilon}(\omega )$ for \eqref{eq:4.39} with respect to deterministic bounded sets.
		\par 
		\textbf{Step 3.} Now, we prove that \eqref{eq:4.47} holds. From Lemma \ref{lem:4.5} and \eqref{eq:4.37}, we have
		\begin{equation}
			\label{eq:4.55}
			\lim\limits_{\substack{\delta\rightarrow0^+}}x_{\delta,\epsilon}^*(\omega)=x_{0,\epsilon}^*(\omega),
		\end{equation}
		and there exists $r<0$, $\epsilon_0\in (0,1]$ and $\bar{\delta}>0$, such that for all $0<\delta<\bar{\delta}$, $\epsilon\in (0,\epsilon_0]$ we have
		\begin{equation}
			\label{eq:4.56}
			|x_{\delta,\epsilon}^*(\theta_{t}\omega)|\leq|t|\quad\forall t\leq r.
		\end{equation}	
		Note that
		\begin{equation}
			\begin{aligned}
				\int_{-\infty}^{0}&e^{(m\lambda_1-4C_f)t}\left(\frac{|x_{\delta,\epsilon}^*(\theta_{t}\omega)|^2}{\lambda_1C_f}+\frac{2C_f|x_{\delta,\epsilon}^*(\theta_{t}\omega)|^2}{\lambda_1}\right)dt\\
				&=\int_{-\infty}^{r}e^{(m\lambda_1-4C_f)t}\left(\frac{|x_{\delta,\epsilon}^*(\theta_{t}\omega)|^2}{\lambda_1C_f}+\frac{2C_f|x_{\delta,\epsilon}^*(\theta_{t}\omega)|^2}{\lambda_1}\right)dt\\
				&+\int_{0}^{r}e^{(m\lambda_1-4C_f)t}\left(\frac{|x_{\delta,\epsilon}^*(\theta_{t}\omega)|^2}{\lambda_1C_f}+\frac{2C_f|x_{\delta,\epsilon}^*(\theta_{t}\omega)|^2}{\lambda_1}\right)dt.
				\nonumber
			\end{aligned}
		\end{equation}
		Consequently, for all $0<\delta<\bar{\delta}$, from \eqref{eq:4.56} we get that
		\begin{equation}
			\begin{aligned}
				\int_{-\infty}^{r}&e^{(m\lambda_1-4C_f)t}\left(\frac{|x_{\delta,\epsilon}^*(\theta_{t}\omega)|^2}{\lambda_1C_f}+\frac{2C_f|x_{\delta,\epsilon}^*(\theta_{t}\omega)|^2}{\lambda_1}\right)dt\\
				&\leq\int_{-\infty}^{r}e^{(m\lambda_1-4C_f)t}\left(\frac{|t|^2}{\lambda_1C_f}+\frac{2C_f|t|^2}{\lambda_1}\right)dt<\infty.
				\nonumber
			\end{aligned}
		\end{equation}
		Combing with the above inequality, Lemma \ref{lem:4.5}, the continuity of $x_{\delta,\epsilon}^*(\theta_{t}\omega)$ and the Lebesgue dominated convergence theorem, we obtain that
		\begin{equation}
			\begin{aligned}
				\label{eq:4.57}
				\lim\limits_{\substack{\delta\rightarrow0^+}}&\int_{-\infty}^{r}e^{(m\lambda_1-4C_f)t}\left(\frac{|x_{\delta,\epsilon}^*(\theta_{t}\omega)|^2}{\lambda_1C_f}+\frac{2C_f|x_{\delta,\epsilon}^*(\theta_{t}\omega)|^2}{\lambda_1}\right)dt\\
				&=\int_{-\infty}^{r}e^{(m\lambda_1-4C_f)t}\left(\frac{|x_{0,\epsilon}^*(\theta_{t}\omega)|^2}{\lambda_1C_f}+\frac{2C_f|x_{0,\epsilon}^*(\theta_{t}\omega)|^2}{\lambda_1}\right)dt,
			\end{aligned}
		\end{equation}
		and
		\begin{equation}
			\begin{aligned}
				\label{eq:4.58}
				\lim\limits_{\substack{\delta\rightarrow0^+}}&\int_{0}^{r}e^{(m\lambda_1-4C_f)t}\left(\frac{|x_{\delta,\epsilon}^*(\theta_{t}\omega)|^2}{\lambda_1C_f}+\frac{2C_f|x_{\delta,\epsilon}^*(\theta_{t}\omega)|^2}{\lambda_1}\right)dt\\
				&=\int_{0}^{r}e^{(m\lambda_1-4C_f)t}\left(\frac{|x_{0,\epsilon}^*(\theta_{t}\omega)|^2}{\lambda_1C_f}+\frac{2C_f|x_{0,\epsilon}^*(\theta_{t}\omega)|^2}{\lambda_1}\right)dt.
			\end{aligned}
		\end{equation}
		By similar arguments to \eqref{eq:4.57} and \eqref{eq:4.58}, it is easy derive that
		\begin{equation}
			\begin{aligned}
				\label{eq:4.59}
				\lim\limits_{\substack{\delta\rightarrow0^+\\\epsilon\rightarrow0^+}}&\int_{-1}^{0}e^{(m\lambda_1-4C_f)t}\left(C_f\lambda_1+\frac{\lambda_1}{C_f}\right)|x_{\delta,\epsilon}^*(\theta_{t}\omega)|^2dt\\
				&=\int_{-1}^{0}e^{(m\lambda_1-4C_f)t}\left(C_f\lambda_1+\frac{\lambda_1}{C_f}\right)|x_{0,\epsilon}^*(\theta_{t}\omega)|^2dt.
			\end{aligned}
		\end{equation}
		Along with \eqref{eq:4.57}-\eqref{eq:4.59}, the proof is complete.
	\end{proof}
	\par 
	\begin{lem}\label{lem:4.9} Under assumptions of Theorem \ref{thm:4.5}, let $\{\delta_n\}_{n=1}^{\infty}$ be a sequence satisfying $\delta_n\rightarrow0^+$ as $n\rightarrow+\infty$. Let $u_{\delta_n,\epsilon}$ and $u_{0,\epsilon}$ be the solutions of \eqref{eq:4.3} and \eqref{eq:4.1} with initial values $u_{\delta_n,\epsilon}(\tau)$ and $u_{0,\epsilon}(\tau)$, respectively. If $u_{\delta_n,\epsilon}(\tau)\rightarrow u_{0,\epsilon}(\tau)$ strongly in $H$ as $n\rightarrow+\infty$, then for almost all $\omega\in\Omega$ and $t\geq\tau$, 
	\begin{equation}
		u_{\delta_n,\epsilon}(t; \tau, \omega, u_{\delta_n,\epsilon}(\tau))\rightarrow u(t;\tau, \omega, u_{0,\epsilon}(\tau))\quad strongly\ in\ H \ as\ n\rightarrow+\infty.
		\nonumber
	\end{equation}
	\end{lem}
	\begin{proof}
		The proof is similar to \cite[Lemma 4.4]{A. Gu} and we omit the details here.
	\end{proof}
	\begin{lem}\label{lem:4.10} Under assumptions of Theorem \ref{thm:4.5}, let $\{\delta_n\}_{n=1}^{\infty}$ be a sequence satisfying $\delta_n\rightarrow0^+$ as $n\rightarrow+\infty$. Let $p_{\delta_n,\epsilon}$ and $p_{0,\epsilon}$ be the solutions of \eqref{eq:4.39} and \eqref{eq:4.35} with initial values $p_{\delta_n,\epsilon}(\tau)$, and $p_{0,\epsilon}(\tau)$, respectively. If $p_{\delta_n,\epsilon}(\tau)\rightarrow p_{0,\epsilon}(\tau)$ weakly in $H$ as $n\rightarrow+\infty$, then for almost all $\omega\in\Omega$ and $t\geq\tau$, 
	\begin{equation}
		\label{eq:4.60}
		p_{\delta_n,\epsilon}(r; \tau, \omega, p_{\delta_n,\epsilon}(\tau))\rightarrow p(r;\tau, \omega, p_{0,\epsilon}(\tau))\quad weakly\ in\quad H\quad\forall r\ge\tau,
	\end{equation}
	and
	\begin{equation}
		\label{eq:4.61}
		p_{\delta_n,\epsilon}(\cdot; \tau, \omega, p_{\delta_n,\epsilon}(\tau))\rightarrow p(\cdot; \tau, \omega, p_{0,\epsilon}(\tau))\quad weakly\ in\quad L^2(\tau, \tau+T)\quad\forall T>0.
	\end{equation}
	\begin{proof}
		The results follow similarly to the proof of existence of solutions to equation \eqref{eq:4.39} \cite[Lemma 3.5]{A.G B}. We therefore omit the details.
	\end{proof}
	\end{lem}
	\begin{lem}\label{lem:4.11} Supposing the conditions of Theorem \ref{thm:4.5} hold, let $\omega\in\Omega$ be fixed. If $\delta_n\rightarrow0^+$ as $n\rightarrow+\infty$ and $u_{\delta_n,\epsilon}\in\mathcal{A}_{\delta_n,\epsilon}(\omega)$, then the sequence $\{ u_{\delta_n,\epsilon}\}_{n=1}^{\infty}$ has a convergent subsequence in $H$.
	\end{lem}
	\begin{proof}
		Thanks to $\delta_n\rightarrow0^+$ as $n\rightarrow+\infty$, by Theorem \ref{thm:4.6}, we can get that for almost all $\omega\in\Omega$, there exist $N=N(\omega)$, such that for all $n\ge N$
		\begin{equation}
			\label{eq:4.62}
			R_{\delta_n,\epsilon}(\omega)\leq2R_{0,\epsilon}(\omega).
		\end{equation}
		Since $u_n :=u_{\delta_n,\epsilon}(t; \tau, \omega, u_{\delta_n,\epsilon}(\tau))\in\mathcal{A}_{\delta_n,\epsilon}(\omega)$, and $\mathcal{A}_{\delta_n,\epsilon}\subset R_{\delta_n,\epsilon}(\omega)$, for all $n\ge N$, we obtain
		\begin{equation}
			\label{eq:4.63}
			|u_n|^2\leq2\lambda_1^{-1}R_{0,\epsilon}(\omega).
		\end{equation}
		Indeed, according to \eqref{eq:4.63} that $u_n$ is bounded in $H$, up to a subsequence, we have
		\begin{equation}
		\label{eq:4.64}
			u_n\rightarrow\tilde{u}\quad weakly\ in\ \quad H. 
		\end{equation}
		Next, we will prove that the weak convergence in \eqref{eq:4.64} is actually a strong one. \par 
		On the one hand, $u_n\in\mathcal{A}_{\delta_n,\epsilon}(\omega )$, since the invariance of $\mathcal{A}_{\delta_n,\epsilon}(\omega )$ for every $k\geq 1$, there exists $u_{n,k}(\omega) :=u_{0,\epsilon,n,k}(\omega)\in\mathcal{A}_{\delta_n,\epsilon}(\theta_{-k}\omega)$, we have
		\begin{equation}
			\label{eq:4.65}
			u_n=\Psi_{\delta_n, \epsilon}(k,\theta_{-k}\omega, u_{n,k})=u_{\delta_n,\epsilon}(0; -k, \omega, u_{n,k}).
		\end{equation}
		Thank to $u_{n,k}\in\mathcal{A}_{\delta_n,\epsilon}(\theta_{-k}\omega)$, and $\mathcal{A}_{\delta_n,\epsilon}(\theta_{-k}\omega)\subset B_{\delta_n,\epsilon}(\theta_{-k}\omega)$, according to \eqref{eq:4.62}, for each $k\ge1$ and $n\ge N :=N(\theta_{-k}\omega)$, we infer that
		\begin{equation}
			\label{eq:4.66}
			|u_{n,k}|^2\leq2\lambda_1^{-1}R_{0,\epsilon}(\theta_{-k}\omega).
		\end{equation}
		\par 
		On the other hand, by \eqref{eq:4.38}, we can get that
		\begin{equation}
			\label{eq:4.67}
			p_{\delta_n,\epsilon}(0; -k, \omega, p_{n,k})=u_{\delta_n,\epsilon}(0; -k, \omega, u_{n,k})-x_{\delta_n,\epsilon}^*(\omega),
		\end{equation}
		where $p_{n,k}=u_{n,k}-x_{\delta_n,\epsilon}(\theta_{-k}\omega)$. By \eqref{eq:4.65} and \eqref{eq:4.67}, it follows that
		\begin{equation}
			\label{eq:4.68}
			u_n=p_{\delta_n,\epsilon}(0; -k, \omega, p_{n.k})+x_{\delta_n,\epsilon}^*(\omega).
		\end{equation}
		Making use of \eqref{eq:4.66}, we have 
		\begin{equation}
			\label{eq:4.69}
			|p_{n,k}|^2\leq2|u_{n,k}|^2+2|x_{\delta_n,\epsilon}^*(\omega)|^2\leq4\lambda_1^{-1}R_{0,\epsilon}(\theta_{-k}\omega)+2|x_{\delta_n,\epsilon}^*(\omega)|^2.
		\end{equation}
		According to Lemma \ref{lem:4.5} and \eqref{eq:4.69}, there exists $N_1 :=N_1(\omega,k)$ such that for every $k\ge1$ and $n\ge N_1$, we have
		\begin{equation}
			\label{eq:4.70}
			|p_{n,k}|^2\leq4\lambda_1^{-1}R_{0,\epsilon}(\theta_{-k}\omega)+4(1+|x_{0,\epsilon}^*(\omega)|^2).
		\end{equation}
		It follows from Lemma \ref{lem:4.5}, \eqref{eq:4.66} and \eqref{eq:4.68}, as $n\rightarrow+\infty$, we have
		\begin{equation}
			\label{eq:4.71}
			p_{\delta_n,\epsilon}(0; -k, \omega, p_{n,k})\rightarrow\tilde{p}_{0,\epsilon}\quad weakly\ in\ \quad H \quad with\ \quad \tilde{p}_{0,\epsilon}=\tilde{u}_{0,\epsilon}-x_{0,\epsilon}^*(\omega).
		\end{equation}
		\par 
		Next, making use of energy estimations, we evaluate the limit of norm $|p_{\delta_n,\epsilon}(0; -k, \omega, p_{n,k})|$ for each $k$ as $n\rightarrow+\infty$. By (4.68) we can get that for each $k\geq1$, the sequence $\{ p_{n,k}\}_{n=1}^{\infty}$ is bounded in $H$, and by a diagonal process, we can derive a subsequence such that for each $k\geq 1$, there exists $\bar{p}_k\in H$, as $n\rightarrow +\infty$ such that
		\begin{equation}
			\label{eq:4.72}
			p_{n,k}\rightarrow\bar{p}_k\quad weakly\ in\ \quad H.
		\end{equation}
		By Lemma \ref{lem:4.10} and \eqref{eq:4.72}, as $n\rightarrow +\infty$, we have
		\begin{equation}
			p_{\delta_n,\epsilon}(0; -k, \omega, p_{n,k})\rightarrow p(0; -k, \omega, \bar{p}_k)\quad weakly\ in\ \quad H,
		\end{equation}
		and
		\begin{equation}
			\label{eq:4.74}
			p_{\delta_n,\epsilon}(\cdot; -k, \omega, p_{n,k})\rightarrow p(\cdot; -k, \omega, \bar{p}_k)\quad weakly\ in\ \quad L^2(\tau, \tau+T; V),
		\end{equation}
		Since the uniqueness of limit, by \eqref{eq:4.71} and \eqref{eq:4.74}, we can get that
		\begin{equation}
			\label{eq:4.75}
			p_{0,\epsilon}(0; -k, \omega, \bar{p}_k)=\tilde{p}_{0,\epsilon}.
		\end{equation}
		By energy equality and \eqref{eq:4.39}, we have
		\begin{equation}
			\begin{aligned}
				\label{eq:4.76}
				\frac{d}{dt}|p_{\delta_n,\epsilon}(t)|^2+2m\lambda_1|p_{\delta_n,\epsilon}(t)|^2+\Upsilon(p_{\delta_n,\epsilon}(t))=2(f(p_{\delta_n,\epsilon}+x_{\delta_n,\epsilon}^*(\theta_{t}\omega)),p_{\delta_n,\epsilon}(t))\\
				+2(x_{\delta_n,\epsilon}^*(\theta_{t}\omega),p_{\delta_n,\epsilon}(t))-((2a(l(p_{\delta_n,\epsilon}+x_{\delta_n,\epsilon}^*(\theta_{t}\omega)),p_{\delta_n,\epsilon}(t))),
			\end{aligned}
		\end{equation}
		where $\Upsilon(p_{\delta_n,\epsilon}(t)=2a(l(p_{\delta_n,\epsilon}+x_{\delta_n,\epsilon}^*(\theta_{t}\omega)))\|p_{\delta_n,\epsilon}(t)\|^2-m\lambda_1|p_{\delta_n,\epsilon}(t)|^2$, which is a functional in $V$.\\
		\par 
		Multiplying \eqref{eq:4.76} by $e^{m\lambda_1t}$ and integrating it from $- k$ to 0, we can obtain
		\begin{equation}
			\begin{aligned}
				|p_{\delta_n,\epsilon}(0; -k, \omega, p_{n,k})|^2&=e^{-m\lambda_1k}|p_{n,k}|^2-\int_{-k}^{0}e^{m\lambda_1t}\Upsilon( p_{\delta_n,\epsilon}(t; -k, \omega, p_{n,k}))dt\\
				&+2\int_{-k}^{0}e^{m\lambda_1t}(f(p_{\delta_n,\epsilon}(t; -k, \omega, p_{n,k})+\phi x_{\delta,\epsilon}^*(\theta_{t}\omega)),p_{\delta_n,\epsilon}(t; -k, \omega, p_{n,k}))dt\\
				&+2\int_{-k}^{0}e^{m\lambda_1t}(x_{\delta_n,\epsilon}^*(\theta_{t}\omega),p_{\delta_n,\epsilon}(t; -k, \omega, p_{n,k}))dt\\
				&-2\int_{-k}^{0}e^{m\lambda_1t}((2a(l(p_{\delta_n,\epsilon}+x_{\delta_n,\epsilon}^*(\theta_{t}\omega)),p_{\delta_n,\epsilon}(t; -k, \omega, p_{n,k})))dt.
				\nonumber
			\end{aligned}
		\end{equation}
		Similarly, by \eqref{eq:4.36}, \eqref{eq:4.71} and \eqref{eq:4.75}, we obtain
		\begin{equation}
			\begin{aligned}
				|\tilde{p}_{0,\epsilon}|^2:&=|\bar{p}_{0,\epsilon}(0,-k,\omega,\tilde{p}_k)|^2=e^{-m\lambda_1k}|\bar{p}_k|^2-\int_{-k}^{0}e^{m\lambda_1t}\Upsilon( p_{0,\epsilon}(t; -k, \omega,\tilde{p}_k))dt\\
				&+2\int_{-k}^{0}e^{m\lambda_1t}(f(p_{0,\epsilon}(t; -k, \omega, \tilde{p}_k)+ x_{0,\epsilon}^*(\theta_{t}\omega)),p_{0,\epsilon}(t; -k, \omega, \bar{p}_k))dt\\
				&+2\int_{-k}^{0}e^{m\lambda_1t}(x_{0,\epsilon}(\theta_{t}\omega),p_{0,\epsilon}(t; -k, \omega, \bar{p}_k))dt\\
				&-2\int_{-k}^{0}e^{m\lambda_1t}((2a(l(p_{0,\epsilon}(t;-k,\omega,\bar{p}_k)+ x_{0,\epsilon}^*(\theta_{t}\omega))),p_{0,\epsilon}(t;-k,\omega,\bar{p}_k)))dt.
			\end{aligned}
		\end{equation}
		It follows that
		\begin{equation}
			\begin{aligned}
				\label{eq:4.78}
				&\lim_{n\rightarrow\infty}\sup|p_{\delta_n,\epsilon}(0; -k, \omega, p_{n,k})|^2\\
				&\leq e^{-m\lambda_1k}((4\lambda_1^{-1}R_{0,\epsilon}(\theta_{-k}\omega))+4(1+|x_{0,\epsilon}^*(\omega)|^2))+|\tilde{p}|^2-e^{-m\lambda_1k}|\bar{p}_{k}|^2\\
				&\leq e^{-m\lambda_1k}((4\lambda_1^{-1}R_{0,\epsilon}(\theta_{-k}\omega))+4(1+|x_{0,\epsilon}^*(\omega)|^2))+|p(0; -k, \omega, \bar p_k)|^2.
			\end{aligned}
		\end{equation}
		According to \eqref{eq:4.75}, for $n\rightarrow+\infty$,
		\begin{equation}
			\label{eq:4.79}
			p_{0,\epsilon}(0; -k,\omega, \bar{p}_k)=\tilde{p}_{0,\epsilon}=u_{0,\epsilon}(0; -k, \omega, \bar{u}_k)-x_{0,\epsilon}^*(\omega) :=\tilde{u}_{0,\epsilon}-x_{0,\epsilon}^*(\omega).
		\end{equation}
		By \eqref{eq:4.68}, we have
		\begin{equation}
			\label{eq:4.80}
			p_{\delta_n,\epsilon}^*(0; -k, \omega, p_{n,k})=u_n-x_{\delta_n,\epsilon}^*(\omega).
		\end{equation}
		Combing with \eqref{eq:4.78}-\eqref{eq:4.80} that
		\begin{equation}
			\begin{aligned}
				\label{eq:4.81}
				\lim_{k\rightarrow\infty}\sup|u_n-x_{\delta_n,\epsilon}(\omega)|\leq e^{-m\lambda_1k}((4\lambda_1^{-1}R_{0,\epsilon}(\theta_{-k}\omega))+4(1+|x_{0,\epsilon}^*(\omega)|^2))+|\tilde{u}-x_{0,\epsilon}^*(\omega)|^2.
			\end{aligned}
		\end{equation}
		Since the $R_{0,\epsilon}$ and $x_{0,\epsilon}^*$ are tempered, we can get 
		\begin{equation}
			\lim_{n\rightarrow\infty}\sup e^{-m\lambda_1k}((4\lambda_1^{-1}R_{0,\epsilon}(\theta_{-k}\omega))+4(1+|x_{0,\epsilon}^*(\omega)|^2))=0
			\nonumber
		\end{equation}
		Letting $k\rightarrow+\infty$, we have
		\begin{equation}
			\label{eq；4.82}
			\lim_{n\rightarrow\infty}\sup|u_{\delta_n,\epsilon}-x_{\delta_n,\epsilon}^*(\omega)|\leq|\tilde{u}-x_{0,\epsilon}^*(\omega)|.
		\end{equation}
		It follows from Lemma \ref{lem:4.5}, \eqref{eq:4.64} and \eqref{eq:4.81}, we obtain
		\begin{equation}
			\label{eq:4.83}
			u_n\rightarrow\tilde{u}_{0,\epsilon}\quad strongly\ in\ \quad H,
		\end{equation}
		This completes the proof.
	\end{proof}
	Next, we will establish the upper semicontinuity of random attractor as $\delta\rightarrow0^+$ and $\epsilon\rightarrow0^+$.\\
	
	\textbf{ 4.4. Upper semi-continuity of random attractor.} In this section, we establish the upper semi-continuity of random attractor when small random perturbations $\delta$ and $\epsilon$ approach zero. Let $(H, \|\cdot\|_H)$ be a Banach space and $\Psi$ be an autonomous dynamical
	system defined on $H$. 
	
	Given a small positive parameter $\epsilon$, consider the following stochastically perturbed equation:
	\begin{equation}
		\label{eq:4.84}
			\frac{\partial u_{0,\epsilon}}{\partial t}-a(l(u_{0,\epsilon}))\Delta u_{0,\epsilon}=f(u_{0,\epsilon})+\epsilon\phi\frac{dW}{dt},
	\end{equation}
	with the initial condition:
	\begin{equation}
		\label{eq:4.85}
		u_{0,\epsilon}(\tau)=u_0\in H.
	\end{equation}
	We can associate a random dynamical system $\Psi_{0,\epsilon}$ with problem \eqref{eq:4.84}-\eqref{eq:4.85} via $u_{0,\epsilon}$ for each $\epsilon>0$,
	where $	\Psi_{0,\epsilon} :R^+\times\Omega\times H\rightarrow H$ is given by
	\begin{equation}
		\label{eq:4.86}
		\Psi_{0,\epsilon} (t,\omega,u_0)=u_{0,\epsilon}(t,\omega,u_0),\quad for\ every\ (t,\omega,u_0)\in R^+\times\Omega\times H.
	\end{equation}
	By Theorem \ref{thm:4.5}, we have that $\Psi_{0,\epsilon}$ has a unique $\mathcal{D}$-pullback random attractor $\mathcal{A}_{0,\epsilon}(\omega)$. When $\epsilon=0$, problem $\eqref{eq:4.84}-\eqref{eq:4.85}$ defines a continuous deterministic dynamical system $\Psi$ in $H$. In this case, the results of \cite{Caraballo THerrera-Cobos M Marín-Rubio P} imply that $\Psi$ has a unique global attractor $\mathcal{A}$ in $H$.
	\par 
	Given $\epsilon\in(1,0]$, according to Theorem \ref{thm:4.5} that  $\Psi_{0,\epsilon}$ is a random dynamical system such that for P-almost every $\omega\in\Omega$ and all $t\in R^+$
	\begin{equation}
		\Psi_{0,\epsilon}(t,\theta_{-t}\omega)x\rightarrow\Psi(t)x\quad as\ \epsilon\rightarrow0^+.
		\nonumber
	\end{equation}
	uniformly on bounded sets of $H$.
	\par 
	Then the relationships between $\mathcal{A}_{0,\epsilon}(\omega)$ and $\mathcal{A}$ are given by the following theorem.
	\begin{thm}\label{thm:4.12} Assume function $a$ is globally Lipschitz and satisfies \eqref{eq:2.1}, $f\in C(\Bbb R)$ satisfies \eqref{eq:2.2} and \eqref{eq:2.4} with $p=2$ and $\beta=C_f$, respectively, $\phi\in V\cap H^2(\mathcal{O})$, and $l\in L^2(\mathcal{O})$. Also, let $m\lambda_1>4C_f$. There exists $\epsilon_0\in (0,1]$ such that for all $\epsilon\in(0,\epsilon_0]$, then for almost all $\omega\in\Omega$,
	\begin{equation}
		\label{eq:4.87}
		\lim\limits_{\substack{\epsilon\rightarrow0^+}} dist_H(\mathcal{A}_{0,\epsilon}(\omega),\mathcal{A})=0,
	\end{equation}
	where
	\begin{equation}
		\lim\limits_{\substack{\epsilon\rightarrow0^+}} dist_H(\mathcal{A}_{0,\epsilon}(\omega),\mathcal{A})=\sup_{a\in\mathcal{A}_{0,\epsilon}(\omega)}\inf_{b\in\mathcal{A}}\|a-b\|_{H}.
		\nonumber
	\end{equation}
\end{thm}
	\begin{proof}
		The proof is similar to \cite[Theorem 2]{T.CaJ.A.} and we omit the detail here.
	\end{proof}
	\par 	
	\begin{thm}\label{thm:4.13} Assume function $a$ is globally Lipschitz and satisfies \eqref{eq:2.1}, $f\in C(\Bbb R)$ satisfies \eqref{eq:2.2} and \eqref{eq:2.4} with $p=2$ and $\beta=C_f$, respectively, $\phi\in V\cap H^2(\mathcal{O})$, and $l\in L^2(\mathcal{O})$. Also, let $m\lambda_1>4C_f$. Then, there exists $\tilde{\delta} >0$ and $\epsilon_0\in(0,1]$ such that for all $0<\delta<\tilde{\delta}$, $\epsilon\in(0,\epsilon_0]$, for almost all $\omega\in\Omega$,
	\begin{equation}
		\label{eq:4.88}
		\lim\limits_{\substack{\delta\rightarrow0^+}} dist_H(\mathcal{A}_{\delta,\epsilon}(\omega),\mathcal{A}_{0,\epsilon}(\omega))=0.
	\end{equation}
	where
	\begin{equation}
		\lim\limits_{\substack{\delta\rightarrow0^+}} dist_H(\mathcal{A}_{\delta,\epsilon}(\omega),\mathcal{A}_{0,\epsilon}(\omega))=\sup_{a\in\mathcal{A}_{\delta,\epsilon}(\omega)}\inf_{b\in\mathcal{A}_{0,\epsilon}(\omega)}\|a-b\|_{H}.
		\nonumber
	\end{equation}
		\end{thm}
	\begin{proof}
		For every fixed $\omega\in\Omega$, by Theorem \ref{thm:4.6} that for almost $\omega\in\Omega$, we have
		\begin{equation}
			\label{eq:4.89}
			\lim\limits_{\substack{\delta\rightarrow0^+}}|B_{\delta,\epsilon}(0,\omega)|=|B_{0,\epsilon}(0,\omega)|,
		\end{equation}
		along with Lemma \ref{lem:4.10} and Lemma \ref{lem:4.11}, by applying \cite[Theorem 3.1]{B.W E}, the proof is complete.
	\end{proof} 
	\par 
	Based on the above analysis, we can derive that the random attractor of the stochastic equation \eqref{eq:4.3} driven by additive noise and the global attractor of the corresponding deterministic equation \eqref{eq:4.2} possess the following convergence relationship as the perturbation parameters $\delta$ and $\epsilon$ both approach zero.
	\par 
	\begin{thm}\label{thm:4.14}
 Assume function $a$ is globally Lipschitz and satisfies \eqref{eq:2.1}, $f\in C(\Bbb R)$ satisfies \eqref{eq:2.2} and \eqref{eq:2.4} with $p=2$ and $\beta=C_f$, respectively, $\phi\in V\cap H^2(\mathcal{O})$. Also, let $m\lambda_1>4C_f$ and $l\in L^2(\mathcal{O})$. Then for almost all $\omega\in\Omega$, we have 
	\begin{equation}
		\lim\limits_{\substack{\delta\rightarrow0^+\\\epsilon\rightarrow0^+}} dist_H(\mathcal{A}_{\delta,\epsilon}(\omega),\mathcal{A})=0.
		\nonumber
	\end{equation}
	where
	\begin{equation}
		\lim\limits_{\substack{\delta\rightarrow0^+\\\epsilon\rightarrow0^+}} dist_H(\mathcal{A}_{\delta,\epsilon}(\omega),\mathcal{A})=\sup_{a\in\mathcal{A}_{\delta,\epsilon}(\omega)}\inf_{b\in\mathcal{A}}\|a-b\|_{H}.
		\nonumber
	\end{equation}
		\end{thm}
	\begin{proof}
		By Theorem \ref{thm:4.12} and Theorem \ref{thm:4.13}, there exists $\epsilon_0\in (0,1]$, $\tilde{\delta} >0$ such that for all $\epsilon\in(0,\epsilon_0]$, $0<\delta<\tilde{\delta}$, we have
		\begin{equation}
			\begin{aligned}
				dist_H(\mathcal{A}_{\delta,\epsilon}(\omega),\mathcal{A})&=\sup_{a\in\mathcal{A}_{\delta,\epsilon}(\omega)}\inf_{b\in\mathcal{A}}\|a-b\|_{H}\\
				&\leq\sup_{a\in\mathcal{A}_{\delta,\epsilon}(\omega)}\inf_{c\in\mathcal{A}_{0,\epsilon}(\omega)}\|a-c\|_{H}+\sup_{c\in\mathcal{A}_{0,\epsilon}(\omega)}\inf_{b\in\mathcal{A}}\|c-b\|_{H}.
				\nonumber
			\end{aligned}
		\end{equation}
		The proof is complete.
	\end{proof}
	\section{Convergence of random attractors for stochastic nonlocal PDEs with multiplicative noise}
	In this section, we study the case that $ g(t,u)$ in (1.1) is $u$ itself, i.e., the case of multiplicative noise.\\
	
	\textbf{5.1. Convergence of solutions}. In this section, we first study the convergence of solutions of the random differential equation
	\begin{equation}
		\label{eq:5.1}\indent
		\frac{\partial u_{0,\epsilon}}{\partial t}-a(l(u_{0,\epsilon}))\Delta u_{0,\epsilon}=f(u_{0,\epsilon})+\epsilon u_{0,\epsilon}\circ\frac{dW}{dt},\quad u_{0,\epsilon}(\tau)=u_0\in H.
	\end{equation}
	to the differential equation \eqref{eq:4.2} as $\epsilon\rightarrow0^+ $. Then we consider that the convergence of solutions of random differential equation
	\begin{equation}
		\label{eq:5.2}\indent
		\frac{\partial u_{\delta,\epsilon}}{\partial t}-a(l(u_{\delta,\epsilon}))\Delta u_{\delta,\epsilon}=f(u_{\delta,\epsilon})+\epsilon u_{\delta,\epsilon}\zeta_{\delta}(\theta_{t}\omega),\quad u_{\delta,\epsilon}(\tau)=u_0\in H,
	\end{equation}
	to equation \eqref{eq:4.2} as $\delta\rightarrow0^+$ and $\epsilon\rightarrow0^+$. 
	\par 
	We need to verify for any $T>\tau$ with $T>0$, $\tau\in R$, $\omega\in\Omega $, the solution of equation \eqref{eq:5.1} is uniformly convergent to the solution of equation \eqref{eq:4.2} on $t\in[\tau,T]$ as $ \epsilon\rightarrow0^+ $, and consequently the solution of equation \eqref{eq:5.2} is uniformly convergent to solution of equation \eqref{eq:4.2} as $\delta\rightarrow0^+$ and $\epsilon\rightarrow0^+$.
	\par 
	In what follows, we denote the same notations as in Section 4 and Section 5 without the confusion. We still let $u_{\delta,\epsilon}(t,\omega,u_0)$, $u(t,u_0)$, $u_{0,\epsilon}(t,\omega,u_0)$ be solutions of equations \eqref{eq:5.2}, \eqref{eq:4.2} and \eqref{eq:5.1}, respectively. 
	\par 
	Let $v_{0, \epsilon}=e^{-\epsilon W(s)}u_{0, \epsilon}$. From \eqref{eq:5.1} we find that $v_{0,\epsilon}(t, \omega, v_{0,\epsilon}(\tau))$ satisfies equation
	\begin{equation}
		\begin{aligned}
			\label{eq:5.3}
			\frac{\partial{v_{0,\epsilon}}}{\partial{t}}=&a(l(v_{0,\epsilon})e^{\epsilon W(t)})\Delta v_{0,\epsilon}+e^{-\epsilon W(t)}f(v_{0,\epsilon}e^{\epsilon W(t)})+\epsilon v_{0,\epsilon}W(t),
		\end{aligned}
	\end{equation}
			\begin{equation}
				\label{eq:5.4}
				v_{0,\epsilon}(\tau)=e^{-\epsilon W(\tau)}u_0\in H.
			\end{equation}
	\par 
	Set $v_{\delta, \epsilon}=e^{-\epsilon\int_{0}^{t}\zeta_{\delta}(\theta_{s}\omega)ds}u_{\delta, \epsilon}$. From \eqref{eq:5.2} we obtain that 
	$v_{\delta, \epsilon}(t, \omega, v_{\delta,\epsilon}(\tau))$ satisfies equation
	\begin{equation}
		\begin{aligned}
			\label{eq:5.5}
			\frac{\partial{v_{\delta,\epsilon}}}{\partial{t}}&=a(l(v_{\delta,\epsilon})e^{\epsilon\int_{0}^{t}\zeta_{\delta}(\theta_{s}\omega)ds})\Delta v_{\delta,\epsilon}+e^{-\epsilon\int_{0}^{t}\zeta_{\delta}(\theta_{s}\omega)ds}f(v_{\delta,\epsilon}e^{\epsilon\int_{0}^{t}\zeta_{\delta}(\theta_{s}\omega)ds})+\epsilon v_{\delta,\epsilon}{\int_{0}^{t}\zeta_{\delta}(\theta_{s}\omega)ds},
				\end{aligned}
		\end{equation}
			\begin{equation}
			v_{\delta,\epsilon}(\tau)=e^{-\epsilon\int_{0}^{\tau}\zeta_{\delta}(\theta_{s}\omega)ds}u_0\in H.
			\end{equation} 
	\par 
	According to Hypothesis 1.1, for $T>\tau$ with $T>0$, $\omega\in\Omega$, we have
	\begin{equation}
		\begin{aligned}
			\sup_{\delta\rightarrow0^+}|v_{\delta,\epsilon}(\tau)-v_{0,\epsilon}(\tau)|=0.
			\nonumber
		\end{aligned}
	\end{equation}
	\par 
	Before proving the convergence relationship of the solutions to equations \eqref{eq:5.2} and \eqref{eq:4.2}, we first give the following estimates.\\
	
	\par 
	\begin{lem}
		\label{lem:5.1}
For each $\omega\in\Omega$, $v_{\delta,\epsilon}(\tau)\in H$ and $v_{0,\epsilon}(\tau)\in H$, there exists a positive constant $\alpha =\alpha(T, v_{\delta,\epsilon}(\tau), v_{0,\epsilon}(\tau))$, $\tau\in R$ such that
	\begin{equation}
		\label{eq:5.7}
		\sup_{\tau\leq t\leq T}\|v_{0,\epsilon}(t,\omega,v_{0,\epsilon}(\tau))\|^2\leq \alpha,\quad for\ \forall\epsilon>0,
	\end{equation}
	\begin{equation}
		\label{eq:5.8}
		\sup_{\tau\leq t\leq T}\|v_{\delta,\epsilon}(t,\omega,v_{\delta,\epsilon}(\tau))\|^2\leq \alpha,\quad for\ \forall\epsilon>0, \delta>0,
	\end{equation}
		\end{lem}
	\begin{proof}
		By \eqref{eq:2.1}, \eqref{eq:5.3},\eqref{eq:1.11} and \eqref{eq:2.4} with $\beta=C_f$ and $p=2$, for all $\epsilon>0$, $\delta\ge0$ and $t>\tau$ we observe that
		\begin{equation}
			\|v_{0,\epsilon}(t,\omega,v_{0,\epsilon}(\tau)\|^2\leq e^{(m\lambda_1-\frac{2C_f^2}{m\lambda_1})\tau+2\epsilon\int_{\tau}^{t}W(s)ds}\|v_{0,\epsilon}(\tau)\|^2+\frac{2}{m}C_f^2|\mathcal{O}|\int_{\tau}^{t}e^{-2\epsilon W(s)+(m\lambda_1-\frac{2C_f^2}{m\lambda_1})s+\int_{s}^{t}W(r)dr}ds.
			\nonumber
		\end{equation}
		\par 
		Since $\omega$ is continuous on $[\tau,T]$, therefore \eqref{eq:5.7} holds. The proof of \eqref{eq:5.8} is similar and here for brevity we omit it. Then the proof is complete.
	\end{proof}
	\par 
	The following lemma shows that the approximation of $u_{0,\epsilon}(t,\omega,u_0)$ and $u(t,u_0)$.\\
	
	\par 
	\begin{lem}
		\label{lem:5.2}  Assume Hypotheses 1.1 hold with
	\begin{equation}
		m>(\alpha+2)L_a|l|\lambda_1^{-1}+L_a|l|(\alpha+2)+4\eta\lambda_1^{-1},
		\nonumber
	\end{equation}
	where $a(\cdot)$ is supposed to be globally Lipschitz, the Lipschitz constant is still denoted the same by $L_a$. Then, for each $\tau\in R$, we have $\omega\in\Omega$ and $u_0\in H $,
	\begin{equation}
		\label{eq:5.9}
		\lim_{\epsilon\rightarrow0^+} \sup_{t\in[\tau,T]}|u_{0,\epsilon}(t,\omega,u_0)-u(t,u_0)|^2=0,
	\end{equation}
	where $u_{0,\epsilon}(t,\omega, u_0)$ is the solution of equation \eqref{eq:5.1}, and $u(t,u_0)$ is the solution of equation \eqref{eq:4.2}.
		\end{lem}
	\par
	\begin{proof}
		Let $v_{0, \epsilon}(t, \omega, v_{0,\epsilon}(\tau))=v_{0, \epsilon}(t)$, $u(t, u_0)=u(t)$ respectively. By \eqref{eq:5.3} and \eqref{eq:4.2}, we can obtain that
		\begin{equation}
			\begin{aligned}
				\frac{d}{dt}|v_{0, \epsilon}(t)-u(t)|^2+2m\|v_{0,\epsilon}(t)-u(t)\|^2\leq&2\langle a(l(v_{0, \epsilon}))\Delta u(t)-a(l(u(t)))\Delta u(t), v_{0, \epsilon}(t)-u(t)\rangle\\
				&+2\left(e^{-\epsilon W(t)}f(v_{0, \epsilon}(t)e^{\epsilon W(t)}-f(u(t)), v_{0, \epsilon}(t)-u(t)\right)\\
				&+2\left(\epsilon v_{0, \epsilon}(t)W(t), v_{0, \epsilon}(t)-u(t)\right).
				\nonumber
			\end{aligned}
		\end{equation}
		Since $a$ is globally Lipschitz, denote this Lipschitz constant by $L_a$, and by \eqref{eq:2.1},\eqref{eq:2.2}, \eqref{eq:2.4} with $\beta=C_f$ and $p=2$, the Young inequality and Poincar\'e inequality, for any $t\ge\tau$ we have
		\begin{equation}
			\begin{aligned}
				\frac{d}{dt}|v_{0, \epsilon}(t)-u(t)|^2&\leq(-2m+2L_a|l|\lambda_1^{-1}e^{\epsilon W_t}\|u\|^2+2L_a|l||e^{-\epsilon W_t}-1|\|u\|^2+2L_a|l||e^{\epsilon W_t}|)\|v_{0, \epsilon}(t)-u(t)\|^2\\
				&+(2C_f|1-e^{\epsilon W_t}|+2C_f|e^{-\epsilon W_t}-1|+4\eta|e^{\epsilon W_t}|+2\eta|e^{\epsilon W_t}-1|)|v_{0, \epsilon}(t)-u(t)|^2\\
				&+2C_f|e^{-\epsilon W_t}-1|\lambda_1^{-1}\|v_{0,\epsilon}\|^2+2C_f|e^{-\epsilon W_t}-1|+4\eta\lambda_1^{-1}|e^{\epsilon W_t}-1|\|u(t)\|^2\\
				&+2\epsilon C_r\lambda_1^{-1}\|v_{0,\epsilon}(t)\|^2+2\epsilon C_f|v_{0,\epsilon}(t)-u(t)|^2.
				\nonumber
			\end{aligned}
		\end{equation}
		Note that 
		\begin{equation}
			\sup_{0\leq t\leq T}|1-e^{-\epsilon W_t}|\leq(e^{\epsilon|W_t|}-1)+(1-e^{-\epsilon|W_t|})\rightarrow0,
			\nonumber
		\end{equation}
		as $\epsilon\rightarrow0^+$. Then, for sufficiently small $\epsilon>0$, we can get $|1-e^{-\epsilon W_t}|\leq\frac{1}{2}$,
		\begin{equation}
			\begin{aligned}
				\label{eq:5.10}
				\frac{d}{dt}|v_{0, \epsilon}(t)-u(t)|^2\leq&\left(-2m+4L_a|l|\lambda_1^{-1}\|u\|^2+L_a|l|\|u\|^2+4L_a|l|+8\eta\lambda_1^{-1}\right)\|v_{0, \epsilon}(t)-u(t)\|^2\\
				&+(2C_f|1-e^{\epsilon W_t}|+2C_f|e^{-\epsilon W_t}-1|+2\eta|e^{\epsilon W_t}-1|+2\epsilon C_r)|v_{0, \epsilon}(t)-u(t)|^2\\
				&+2C_f|e^{-\epsilon W_t}-1|\lambda_1^{-1}\|v_{0,\epsilon}\|^2+2C_f|e^{-\epsilon W_t}-1|\\
				&+4\eta\lambda_1^{-1}|e^{\epsilon W_t}-1|\|u(t)\|^2+2\epsilon C_r\lambda_1^{-1}\|v_{0,\epsilon}(t)\|^2.
			\end{aligned}
		\end{equation}
		By Lemma \ref{lem:5.1} and $m>2\alpha L_a|l|\lambda_1^{-1}+L_a|l|(\alpha+2)+4\eta\lambda_1^{-1}$, we have
		\begin{equation}
			\begin{aligned}
				\label{eq:5.11}
				\frac{d}{dt}|v_{0, \epsilon}(t)-u(t)|^2&\leq m(t)|v_{0, \epsilon}(t)-u(t)|^2+n(t),
			\end{aligned}
		\end{equation}
		where
		\begin{equation}
			\begin{aligned}
				m(t)=2C_f|1-e^{\epsilon W_t}|+2C_f|e^{-\epsilon W_t}-1|+2\eta|e^{\epsilon W_t}-1|+2\epsilon C_r,
				\nonumber
			\end{aligned}
		\end{equation}
		\begin{equation}
			\begin{aligned}
				n(t)=&2C_f|e^{-\epsilon W_t}-1|\lambda_1^{-1}\|v_{0,\epsilon}\|^2+2C_f|e^{-\epsilon W_t}-1|\\
				&+4\eta\lambda_1^{-1}|e^{\epsilon W_t}-1|\|u(t)\|^2+2\epsilon C_r\lambda_1^{-1}\|v_{0,\epsilon}(t)\|^2.
				\nonumber
			\end{aligned}
		\end{equation}
		By Gronwall's inequality, we get that
		\begin{equation}
			\begin{aligned}
				|v_{0, \epsilon}(t)-u(t)|^2&\leq\int_{\tau}^{t}e^{m(s)(t-s)}n(s)ds+|v_{0,\epsilon}(\tau)-u_0|^2\rightarrow0,\quad as\ \epsilon\rightarrow0^+.
				\nonumber
			\end{aligned}
		\end{equation}
		Thus, we conclude that
		\begin{equation}
			\label{eq:5.12}
			\lim_{\epsilon\rightarrow0^+}\sup_{t\in[\tau,T]}|v_{0,\epsilon}(t,\omega,u_0)-u(t,u_0)|^2=0.
		\end{equation}
		By \eqref{eq:5.10}, we conclude that
		\begin{equation}
			\label{eq:5.13}
			|u_{0,\epsilon}(t,\omega,u_0)-u(t,u_0)|^2\leq 2e^{2\epsilon|W_t|}|v_{0,\epsilon}(t,\omega,u_0)-u(t,u_0)|^2+2|e^{\epsilon W_t}-1|^2|u(t)|^2.
		\end{equation}
		By  Lemma \ref{lem:5.1} and \eqref{eq:5.11}, the proof is complete. 	
	\end{proof}
	\par 
	Next, we show the approximation between $u_{\delta,\epsilon}(t,\omega,u_0)$ and $ u(t,u_0)$.\\
	
	\par 
	\begin{thm}
		\label{thm:5.3}
 Assuming that Hypothesis 1.1, Lemma \ref{lem:5.1} and Lemma \ref{lem:5.2} hold. For each $\tau\in R$, $ \omega\in\Omega$ and $u_0\in H$,
	\begin{equation}
		\label{lem:5.14}
		\lim\limits_{\substack{\delta\rightarrow0^+\\\epsilon\rightarrow0^+}} \sup_{t\in[\tau,T]}|u_{\delta,\epsilon}(t,\omega,u_0)-u(t,u_0)|^2=0,
	\end{equation}
	where $u_{\delta,\epsilon}(t,\omega, u_0)$ is the solution of equation \eqref{eq:5.2}, and $u(t,u_0)$ is the solution of equation \eqref{eq:4.2}.
		\end{thm}
	\begin{proof}
		Let $v_{\delta, \epsilon}(t, \omega, v_{\delta, \epsilon}(\tau))=v_{\delta, \epsilon}(t)$ and $v_{0,\epsilon}(t,\omega, v_{0, \epsilon}(\tau))=v_{0,\epsilon}(t)$ are the solutions of equation \eqref{eq:5.5} and \eqref{eq:5.3}, respectively. By \eqref{eq:5.5} and \eqref{eq:5.3}, for any $t\in[\tau, T]$, we can obtain
		\begin{equation}
			\begin{aligned}
				\label{eq：5.15}
				&\frac{d}{dt}|v_{\delta, \epsilon}(t)-v_{0,\epsilon}(t)|^2+2m\|v_{\delta,\epsilon}(t)-v_{0,\epsilon}(t)\|^2\\
				&\leq2\langle a(l(v_{\delta,\epsilon}(t))e^{\int_{0}^{t}\zeta_{\delta}(\theta_{s}\omega)ds})\Delta v_{0,\epsilon}(t)-a(l(v_{0,\epsilon}(t)e^{\epsilon W(t)})\Delta v_{0,\epsilon}(t),v_{\delta, \epsilon}(t)-v_{0,\epsilon}(t)\rangle\\
				&+2\left(e^{-\int_{0}^{t}\zeta_{\delta}(\theta_{s}\omega)ds}f(v_{\delta,\epsilon}(t)e^{\int_{0}^{t}\zeta_{\delta}(\theta_{s}\omega)ds})-e^{-\epsilon W(t)}f(v_{0,\epsilon}(r)e^{\epsilon W(t)}),v_{\delta, \epsilon}(t)-v_{0,\epsilon}(t)\right)\\
				&+2\left(\epsilon v_{\delta,\epsilon}(t){\int_{0}^{t}\zeta_{\delta}(\theta_{s}\omega)ds}-\epsilon v_{0,\epsilon}(t)W(t),v_{\delta, \epsilon}(t)-v_{0,\epsilon}(t)\right)\\
				&=I_1''+I_2''+I_3''.
			\end{aligned}
		\end{equation}
		According to \eqref{eq:1.11}, we have $e^{\epsilon\sup_{t\in[\tau,T]}|\int_{0}^{t}\zeta_{\delta}(\theta_{s}\omega)ds-W(t)|}<2$, thus for $\epsilon\in(0,1]$ and $\delta\rightarrow0^+$, we have 
		\begin{equation}
			\begin{aligned}
				\label{eq:5.16}
				\sup_{t\in[\tau,T]}|1-e^{\epsilon\left (\int_{0}^{t}\zeta_{\delta}(\theta_{s}\omega)ds-W(t)\right)}|&\leq(e^{\epsilon\sup_{t\in[\tau,T]}|\int_{0}^{t}\zeta_{\delta}(\theta_{s}\omega)ds-W(t)|}-1)\\
				&+(1-e^{-\epsilon\sup_{t\in[\tau,T]}|\int_{0}^{t}\zeta_{\delta}(\theta_{s}\omega)ds-W(t)|})\rightarrow0.
			\end{aligned}
		\end{equation}
	By \eqref{eq:2.1}, the Poincar\'e inequality and the Young inequality, for any $T>\tau$ with $T>0$, $\tau\in R$, $\omega\in\Omega$, we have
		\begin{equation}
			\begin{aligned}
				I_1''\leq&(2L_a|l||e^{\epsilon W_t}||e^{\epsilon\sup_{0\leq t\leq T}|\int_{0}^{t}\zeta_{\delta}(\theta_{s}\omega)ds-W(t)|}-1|\|v_{0,\epsilon}\|^2\lambda_1^{-1}\\
				&+2L_a|l||e^{\epsilon W_t}|\|v_{0,\epsilon}(t)\|^2+2L_a|l||e^{\epsilon W_t}|\lambda_1^{-1})\|v_{\delta, \epsilon}(r)-v_{0,\epsilon}(r)\|^2\\
				&+2L_a|l||e^{\epsilon W_t}||e^{\epsilon\sup_{0\leq t\leq T}|\int_{0}^{t}\zeta_{\delta}(\theta_{s}\omega)ds-W(t)|}-1|\|v_{\delta,\epsilon}(t)\|^2|v_{\delta,\epsilon}(t)-v_{0,\epsilon}(t)|^2\\
				&+2L_a|l||e^{\epsilon W_t}||e^{\epsilon\sup_{0\leq t\leq T}|\int_{0}^{t}\zeta_{\delta}(\theta_{s}\omega)ds-W(t)|}-1|(\|v_{0,\epsilon}(t)\|^2+\|v_{\delta,\epsilon}(t)\|^2).
				\nonumber
			\end{aligned}
		\end{equation}
		For $I_2''$, by \eqref{eq:2.2}, \eqref{eq:5.16}, \eqref{eq:2.4} with $\beta=C_f$ and $p=2$, we have
		\begin{equation}
			\begin{aligned}
				I_2''=&2(e^{-\epsilon\int_{0}^{t}\zeta_{\delta}(\theta_{s}\omega)ds}-e^{-\epsilon W(t)}f(v_{\delta,\epsilon}e^{\epsilon\int_{0}^{t}\zeta_{\delta}(\theta_{s}\omega)ds})\\
				&+e^{-\epsilon W(r)}\left[f(v_{\delta,\epsilon}e^{\epsilon\int_{0}^{t}\zeta_{\delta}(\theta_{s}\omega)ds})-f(v_{0,\epsilon}e^{\epsilon\int_{0}^{t}\zeta_{\delta}(\theta_{s}\omega)ds})\right]\\
				&+e^{-\epsilon W(t)}\left[f(v_{0,\epsilon}e^{\epsilon\int_{0}^{t}\zeta_{\delta}(\theta_{s}\omega)ds})-f(v_{0,\epsilon}e^{\epsilon W(t)})\right],v_{\delta,\epsilon}(t)-v_{0,\epsilon}(t))\\
				&\leq 2C_f\left|1-e^{\sup_{t\in[\tau,T]}\epsilon\left|\int_{0}^{t}\zeta_{\delta}(\theta_{s}\omega)ds-W(t)\right|}\right|\left(|v_{\delta,\epsilon}(t)-v_{0,\epsilon}(t)|^2+\lambda_1^{-1}\|v_{\delta,\epsilon}(t)\|^2\right)\\
				&+2C_f|e^{-\epsilon W_t}||e^{-\epsilon\sup_{t\in[\tau,T]}\left|\int_{0}^{t}\zeta_{\delta}(\theta_{s}\omega)ds-W(t)\right|}-1|(|\mathcal{O}|+|v_{\delta,\epsilon}(t)-v_{0,\epsilon}(t)|^2)\\
				&+4\eta\lambda_1^{-1} e^{\sup_{t\in[\tau,T]}\epsilon\left|\int_{0}^{t}\zeta_{\delta}(\theta_{s}\omega)ds-W(t)\right|}\|v_{\delta,\epsilon}(t)-v_{0,\epsilon}(t)\|^2\\
				&+2\eta|e^{\sup_{t\in[\tau,T]}\epsilon\left|\zeta_{\delta}(\theta_{s}\omega)ds-W(t)\right|}-1|(\lambda_1^{-1}\|v_{0,\epsilon}(t)\|^2+|v_{\delta,\epsilon}(t)-v_{0,\epsilon}(t)|^2),
				\nonumber
			\end{aligned}
		\end{equation}
		and
		\begin{equation}
			\begin{aligned}
				I_3''&\leq2\epsilon\sup_{t\in[\tau,T]}\left|\int_{0}^{t}\zeta_{\delta}(\theta_{s}\omega)ds-W(t)\right|\left(\lambda_1^{-1}\|v_{\delta,\epsilon}(t)\|^2+|v_{\delta,\epsilon}(t)-v_{0,\epsilon}(t)|^2\right)\\
				&+4\epsilon C_r|v_{\delta, \epsilon}(t)-v_{0,\epsilon}(t)|^2.
				\nonumber
			\end{aligned}
		\end{equation}
		Combing with $I_1''$--$I_3''$ and Lemma 5.1. For all $t\in(\tau, T]$ ,we choose sufficiently small $\delta, \epsilon>0$ such that $|e^{\epsilon W_t}|<1$ and $e^{\epsilon\sup_{t\in[\tau,T]}|\int_{0}^{t}\zeta_{\delta}(\theta_{s}\omega)ds-W(t)|}<2$. Then, we have
		\begin{equation}
			\begin{aligned}
				\frac{d}{dt}|v_{\delta, \epsilon}(t)-v_{0,\epsilon}(t)|^2&\leq(-2m+2L_a|l|\lambda_1^{-1}|e^{\epsilon W_t}|\|v_{0,\epsilon}\|^2+2L_a|l||e^{\epsilon W_t}|\|v_{0,\epsilon}(t)\|^2\\
				&+2L_a|l||e^{\epsilon W_t}|\lambda_1^{-1}+8\eta\lambda_1^{-1})\|v_{\delta, \epsilon}(r)-v_{0,\epsilon}(r)\|^2+y(t)|v_{\delta, \epsilon}(t)-v_{0,\epsilon}(t)|^2+x(t)\\
				&\leq(-2m+2\alpha L_a|l|\lambda_1^{-1}+2\alpha L_a|l|+2L_a|l|\lambda_1^{-1}+8\eta\lambda_1^{-1})\|v_{\delta, \epsilon}(r)-v_{0,\epsilon}(r)\|^2\\
				&+y(t)|v_{\delta, \epsilon}(t)-v_{0,\epsilon}(t)|^2+x(t),
			\end{aligned}
		\end{equation}
		with
		\begin{equation}
			\begin{aligned}
				y(t)=&2\alpha L_a|l||e^{\epsilon\sup_{0\leq t\leq T}|\int_{0}^{t}\zeta_{\delta}(\theta_{s}\omega)ds-W(t)|}-1|+2C_f\left|1-e^{\sup_{t\in[\tau,T]}\epsilon\left|\int_{0}^{t}\zeta_{\delta}(\theta_{s}\omega)ds-W(t)\right|}\right|\\
				&+2C_f|e^{-\epsilon W_t}||e^{-\epsilon\sup_{t\in[\tau,T]}\left|\int_{0}^{t}\zeta_{\delta}(\theta_{s}\omega)ds-W(t)\right|}-1|+2\eta|e^{\sup_{t\in[\tau,T]}\epsilon\left|\zeta_{\delta}(\theta_{s}\omega)ds-W(t)\right|}-1|\\
				&+2\epsilon\sup_{t\in[\tau,T]}\left|\int_{0}^{t}\zeta_{\delta}(\theta_{s}\omega)ds-W(t)\right|+4\epsilon C_r,
				\nonumber
			\end{aligned}
		\end{equation}
		\begin{equation}
			\begin{aligned}
				x(t)=&2\alpha L_a|l||e^{\epsilon\sup_{0\leq t\leq T}|\int_{0}^{t}\zeta_{\delta}(\theta_{s}\omega)ds-W(t)|}-1|(\|v_{0,\epsilon}(t)\|^2+\|v_{\delta,\epsilon}(t)\|^2)\\
				&+2C_f\left|1-e^{\sup_{t\in[\tau,T]}\epsilon\left|\int_{0}^{t}\zeta_{\delta}(\theta_{s}\omega)ds-W(t)\right|}\right|\lambda_1^{-1}\|v_{\delta,\epsilon}(t)\|^2\\
				&+2C_f|e^{-\epsilon W_t}||e^{-\epsilon\sup_{t\in[\tau,T]}\left|\int_{0}^{t}\zeta_{\delta}(\theta_{s}\omega)ds-W(t)\right|}-1||\mathcal{O}|\\
				&+2\eta\lambda_1^{-1}|e^{\sup_{t\in[\tau,T]}\epsilon\left|\zeta_{\delta}(\theta_{s}\omega)ds-W(t)\right|}-1|\|v_{0,\epsilon}(t)\|^2\\
				&+2\epsilon\sup_{t\in[\tau,T]}\left|\int_{0}^{t}\zeta_{\delta}(\theta_{s}\omega)ds-W(t)\right|\lambda_1^{-1}\|v_{\delta,\epsilon}(t)\|^2.
				\nonumber
			\end{aligned}
		\end{equation}
		By Lemma \ref{lem:5.1} and assumption $m>(\alpha+2)L_a|l|\lambda_1^{-1}+(\alpha+2)L_a|l|+4\eta\lambda_1^{-1}$, we obtain
		\begin{equation}
			\begin{aligned}
				\label{eq:5.18}
				\frac{d}{dt}|v_{\delta, \epsilon}(t)-v_{0,\epsilon}(t)|^2&\leq y(t)|v_{\delta, \epsilon}(t)-v_{0,\epsilon}(t)|^2+x(t).
			\end{aligned}
		\end{equation}
		By Gronwall's inequality, we obtain that
		\begin{equation}
			\begin{aligned}
				\label{eq:5.19}
				|v_{\delta, \epsilon}(t)-v_{0,\epsilon}(t)|^2\leq\int_{\tau}^{t}x(s)e^{-\int_{t}^{s}y(\tau)d\tau}ds+|v_{\delta,\epsilon}(\tau)-v_{0,\epsilon}(\tau)|^2.
			\end{aligned}
		\end{equation}
		By $|v_{\delta, \epsilon}(\tau)-v_{0,\epsilon}(\tau)|^2\rightarrow0$ as $\delta\rightarrow0^+$ and $\epsilon\rightarrow0^+$. Thus, we can obtain that
		\begin{equation}
			\label{eq:5.20}
			\lim\limits_{\substack{\delta\rightarrow0^+\\\epsilon\rightarrow0^+}}\sup_{t\in[\tau,T]}|v_{\delta,\epsilon}(t,\omega,v_{\delta,\epsilon}(\tau))-v_{0, \epsilon}(t,\omega,v_{0,\epsilon}(\tau))|^2=0.
		\end{equation}
		For any $t\in[\tau, T]$, $\tau\in R$, $\omega\in\Omega$ and $u_0\in H$, we can obtain that
		\begin{equation}
			\begin{aligned}
				\label{eq:5.21}
				|u_{\delta,\epsilon}(t,\omega,u_0)-u_{0, \epsilon}(t,\omega,u_0)|^2&\leq 2 e^{\epsilon W(t)}|v_{\delta,\epsilon}(t,\omega,u_0)-v_{0, \epsilon}(t,\omega,u_0)|^2\\
				&+2e^{\epsilon W(t)}|1-e^{\epsilon\left(\int_{0}^{r}\zeta_{\delta}(\theta_{s}\omega)ds-W(r)\right)}|^2\times|v_{\delta,\epsilon}(t)|^2.
			\end{aligned}
		\end{equation}
		Combing \eqref{eq:5.20}, \eqref{eq:5.21} and $u_0\in H$, we can get that
		\begin{equation}
			\label{eq:5.22}
			\lim\limits_{\substack{\delta\rightarrow0^+\\\epsilon\rightarrow0^+}} \sup_{t\in[\tau,T]}|u_{\delta,\epsilon}(t,\omega,u_0)-u_{0, \epsilon}(t,\omega,u_0)|^2=0.
		\end{equation}
		Finally, we find that
		\begin{equation}
			|u_{\delta,\epsilon}(t,\omega,u_0)-u(t,u_0)|^2\leq2|u_{\delta,\epsilon}(t,\omega,u_0)-u_{0, \epsilon}(t,\omega,u_0)|^2+2|u_{0,\epsilon}(t,\omega,u_0)-u(t,u_0)|^2.
			\nonumber
		\end{equation}
		Combining with Lemma \ref{lem:5.2} and \eqref{eq:5.22}, the proof is complete.
	\end{proof}
	\par 
	\begin{rem}
		\label{rem:5.4} By \eqref{eq:1.11}, \eqref{eq:5.16}, \eqref{eq:5.19}, \eqref{eq:5.21} and Lemma \ref{lem:5.2}, we clearly see that for any $T>\tau$ with $T > 0$, $\epsilon>0$, $\tau\in R$, $\omega\in\Omega$ and $u_0\in H$, the solutions of equation \eqref{eq:5.2} are uniformly convergent to those of equation \eqref{eq:5.1} on $[\tau, T ]$ as $\delta\rightarrow0^+$, i.e.,
	\begin{equation}
		\lim\limits_{\substack{\delta\rightarrow0^+}} \sup_{t\in[\tau,T]}|u_{\delta,\epsilon}(t,\omega,u_0)-u_{0, \epsilon}(t,\omega,u_0)|^2=0.
		\nonumber
	\end{equation}
		\end{rem}
	
	\textbf{5.2. Random conjugate equations}. To study the pathwise dynamics of problem \eqref{eq:5.1}, for each $\epsilon>0$ we need consider the random variable 
	\begin{equation}
		y_{0, \epsilon}(\omega) :=-\epsilon\int_{-\infty}^{0}e^{s}\omega(s)ds,\quad\forall\omega\in\Omega.
		\nonumber
	\end{equation}
	We can find that
	\begin{equation}
		\Omega\times\Bbb R\ni(\omega,t)\mapsto y_{0,\epsilon}(\theta_{t}\omega)=-\epsilon\int_{-\infty}^{0}e^{s}\theta_{t}\omega(s)ds=-\int_{-\infty}^{0}e^{s}\omega(s+t)ds+\epsilon\omega(t),
		\nonumber
	\end{equation}
	which is solution of the following stochastic differential equation:
	\begin{equation}
		dy=-ydt+\epsilon dW.
		\nonumber
	\end{equation}
	For each $\omega\in\Omega$, the following properties are satisfies
	\begin{equation}
		\label{eq:5.23}
		\lim_{t\rightarrow\pm\infty}\frac{|y_{0,\epsilon}(\theta_{t}\omega)}{|t|}=0,
	\end{equation}
	\begin{equation}
		\label{eq；5.24}
		\lim_{t\rightarrow\pm\infty}\frac{1}{t}\int_{0}^{t}y_{0,\epsilon}(\theta_{s}\omega)ds=0.
	\end{equation}
	For any $\epsilon>0$, let the change of variable $q(t) = e^{-y_{0,\epsilon}(\theta_{t}\omega)}u(t)$, by \eqref{eq:5.1} we have
	\begin{equation}
		\begin{aligned}
			\label{eq:5.25}
			\frac{d{q(t)}}{d{t}}=a(l(q)e^{y_{0,\epsilon}(\theta_{t}\omega)})\Delta q(t)+e^{-y_{0,\epsilon}(\theta_{t}\omega)}f(q(t)e^{y_{0,\epsilon}(\theta_{t}\omega)})+q(t)y_{0,\epsilon}(\theta_{t}\omega),
		\end{aligned}
	\end{equation}
	with initial condition $q'_0 :=u_0e^{-y_{0,\epsilon}(\theta_{\tau}\omega)}$.
	Analogously, we also consider the random variable
	\begin{equation}
		y_{\delta, \epsilon}(\omega) :=-\epsilon\int_{-\infty}^{0}e^{s}\zeta_{\delta}(\theta_{s}\omega)ds,\quad\forall\omega\in\Omega.
		\nonumber
	\end{equation}
	According to Hypothesis 1.3 we see that
	\begin{equation}
		\Omega\times\Bbb R\ni(\omega,t)\mapsto y_{\delta,\epsilon}(\theta_{t}\omega))=-\epsilon\int_{-\infty}^{0}e^{s}\zeta_{\delta}(\theta_{s+t}\omega)ds
		\nonumber
	\end{equation}
	is the solution of the random differential equation:
	\begin{equation}
		dy=-y+\epsilon\zeta_{\delta}(\theta_{t}\omega).
		\nonumber
	\end{equation}
	For each $\omega\in\Omega$, we have
	\begin{equation}
		\label{eq:5.26}
		\lim_{t\rightarrow\pm\infty}\frac{|y_{\delta,\epsilon}(\theta_{t}\omega)}{|t|}=0,
	\end{equation}
	\begin{equation}
		\label{eq:5.27}
		\lim_{t\rightarrow\pm\infty}\frac{1}{t}\int_{0}^{t}y_{\delta,\epsilon}(\theta_{s}\omega)ds=0.
	\end{equation}
	We define the random transformation
	\begin{equation}
		\label{eq:5.28}
		q_{\delta,\epsilon}=e^{-y_{\delta,\epsilon}(\theta_{t}\omega)}u_{\delta,\epsilon},\quad\forall\epsilon, \delta\ge0.
	\end{equation}
	It follows from \eqref{eq:5.2} and \eqref{eq:5.28} that
	\begin{equation}
		\begin{aligned}
			\label{eq:5.29}
			\frac{d{q_{\delta,\epsilon}}}{d{t}}=a(l(q_{\delta,\epsilon}(\theta_{t}\omega))e^{y_{\delta,\epsilon}(\theta_{t}\omega)})\Delta q_{\delta,\epsilon}+e^{-y_{\delta,\epsilon}(\theta_{t}\omega)}f(q_{\delta,\epsilon}e^{y_{\delta,\epsilon}(\theta_{t}\omega)})+q_{\delta,\epsilon}y_{\delta,\epsilon}(\theta_{t}\omega),
		\end{aligned}
	\end{equation}
	with initial condition $q_0 :=q_{\delta,\epsilon}(\tau)=u_{\delta,\epsilon}(\tau)e^{-y_{\delta,\epsilon}(\theta_{\tau}\omega)}$.\par 
	In what follows, we denote $q_{\delta,\epsilon}(\cdot; \tau, \omega, q_0)$ by the solution of \eqref{eq:5.29}. In similar way, we deduce that \eqref{eq:5.29} has weak solution in the sense of Definition 2.3.1, for every $T\ge\tau$, which belong to $L^2(\tau, T; V)\cap L^{\infty}(\tau, T;H)$. Meanwhile, due to the transformation of \eqref{eq:5.28}, for every $T\ge\tau$, there exists a unique weak solution $u_{\delta,\epsilon}(\cdot;\tau, \omega,u_0)\in L^2(\tau, T;V)\cap L^{\infty}(\tau, T; H)$, which is continuous in $H$ with respect to the initial condition.
	\par 
	Define a mapping $\Sigma_{\delta,\epsilon}:\Bbb R^+\times\Omega\times H \rightarrow H$, for every $t\in\Bbb R^+$, we have
	\begin{equation}
		\Sigma_{\delta, \epsilon}(t, \omega , q_0) = q_{\delta, \epsilon} (t; 0, \omega , q_0)\quad\forall q_0\in H, \forall\omega\in\Omega .
		\nonumber
	\end{equation}
	\par 
	Thanks to the conjugation, there is a mapping $\Phi_{\delta,\epsilon}: \Bbb R^+\times \Omega \times H \rightarrow H$, for all $t \in \Bbb R^+$, we have
	\begin{equation}
		\Phi_{\delta,\epsilon}(t, \epsilon, u_{0})=u_{\delta,\epsilon}(t; 0,\omega, u_{0}):=q_{\delta,\epsilon}(t; 0, \omega,e^{-x_{\delta,\epsilon}}q_0)e^{x_{\delta,\epsilon}}\quad \forall  u_{0}\in H, \forall \omega\in\Omega.
		\nonumber
	\end{equation}
	\par 
	\begin{prop}
		\label{prop:5.5}
 Suppose assumptions \eqref{eq:2.1}-\eqref{eq:2.6} are true with $p=2$ and $\beta=C_f$, respectively. Then, for almost all $\omega\in\Omega$, function $a(\omega,\cdot)=a(l((\cdot))e^{y_{\delta,\epsilon}(\theta_{t}\omega)})\in C(\Bbb R; \Bbb R^+)$ is globally Lipschitz and satisfies \eqref{eq:2.1}. Furthermore, there exists a constant
	$C_{F,\delta}$ depending on $\omega, \epsilon, C_f$, and $\eta$, such that
	\begin{equation}
		|F(\omega,s)|\leq C_{F, \delta}(1+|s|)\quad and\ \quad (F(\omega,s)-F(\omega,r))(s-r)\leq\eta|s-r|^2.\quad\forall s,r\in \Bbb R,
		\nonumber
	\end{equation}
	where $F(\omega,s)=e^{-y_{\delta,\epsilon}(\omega)}f(se^{y_{\delta,\epsilon}})+sy_{\delta,\epsilon}$.
		\end{prop}
		\par
	\textbf{5.3. Approximations of random attractors under the multiplicative
		noise}. In what follows, we prove that for any $T>\tau$ with $T>0$, $\tau\in R$, $\omega\in\Omega$, the random attractors of equation \eqref{eq:5.1} is uniformly convergent to random attractors of equation \eqref{eq:5.2} as $\delta\rightarrow0^+$, $\epsilon\rightarrow0^+$.
	\par 
	Before proving the convergence of  random attractors, it is necessary to study the approximation of stationary noises.
	\par 
	\begin{lem}
		\label{lem:5.6}
Assume Hypotheses 1.1-1.3 hold. For each $\omega\in\Omega$ and $T>0$,
	\begin{equation}
		\label{eq:5.30}
		\lim\limits_{\substack{\delta\rightarrow0^+\\\epsilon\rightarrow0^+}}\sup_{|t|\leq T}|y_{\delta,\epsilon}(\theta_{t}\omega)-y_{0,\epsilon}(\theta_{t}\omega)|=0,
	\end{equation}
	and
	\begin{equation}
		\label{eq:5.31}
		\lim_{\epsilon\rightarrow0^+}\sup_{|t|\leq T}|y_{0, \epsilon}(\theta_{t}\omega)|.
	\end{equation}
		\end{lem} 
	\begin{proof}
		It follows from Hypothesis 1.3 (2), we can get \eqref{eq:5.30} is true.
		By \eqref{eq:1.11}, we have
		\begin{equation}
			y_{0,\epsilon}=-\int_{-\infty}^{0}e^r\omega(r+t)dr+\epsilon\omega(t).
			\nonumber
		\end{equation}
		By \eqref{eq:1.6}, for all $|t|\leq T$ we can get that
		\begin{equation}
			|y_{0,\epsilon}(\theta_{t}\omega)|=\int_{-\infty}^{0}e^rC_{\omega}(|t|+1)dr+\epsilon|\omega(t)|\leq\epsilon C_{\omega}(T+2+\sup_{|t|\leq T}|\omega(t)|).
		\end{equation}
		Then it implies that \eqref{eq:5.31} holds.
	\end{proof}
	Next, we show the convergence of random attractors of \eqref{eq:5.29} and \eqref{eq:5.25}.\\
	
	\begin{thm}
		\label{thm:5.7}
  Assume that function $a\in C(\Bbb R; \Bbb R^+)$ fulfills
	\eqref{eq:2.1}, function f satisfies \eqref{eq:2.2} and \eqref{eq:2.4} with $p=2$ and $\beta=C_f$, respectively, and
	$l\in L^2(\mathcal{O})$. Also, there exsists $\epsilon_0\in(0,1]$ such that for all $\epsilon\in(0,\epsilon_0]$ and let $m\lambda_1>4C_f$, $\omega\in\Omega$. Then there exists a unique random attractor $\mathcal{A}_{0,\epsilon}(\omega)$ for the dynamical system $\Phi_{0,\epsilon}(t, \omega, u_0)$ associated to equation \eqref{eq:5.1}. Additionally, this $\mathcal{D}_F$ -pullback absorbing set $B_{0,\epsilon} := \{ B_{0,\epsilon}(\omega) : \omega\in\Omega\}$ in $H$ is given by
	\begin{equation}
		B_{0,\epsilon}(\omega)=\{u\in H :|u|^2\leq\lambda_1^{-1}R_{0,\epsilon}(\omega)\}
		\nonumber
	\end{equation}
	with
	\begin{equation}
		\begin{aligned}
			R_{0,\epsilon}(\omega)=&\frac{1}{m}e^{\int_{-1}^{0}2y_{0,\epsilon}(\theta_{s}\omega)ds+2y_{0,\epsilon}(\omega)}\\
			&\times\left(1+C_f|\mathcal{O}|\int_{-\infty}^{-1}e^{-2y_{0,\epsilon}(\theta_{s}\omega)+(m\lambda_1-3C_f)s+\int_{-1}^{0}2y_{0,\epsilon}(\theta_{\tau}\omega)d\tau}\right)\\
			&+\left(\frac{1}{m}C_f|\mathcal{O}|+\frac{2}{m}C_f^2|\mathcal{O}|\right)\int_{-1}^{0}e^{-2y_{0,\epsilon}(\theta_{s}\omega)+(m\lambda_1-3C_f)s+2y_{0,\epsilon}(\omega)+\int_{s}^{0}2y_{0,\epsilon}(\theta_{r}\omega)dr}ds.
			\nonumber
		\end{aligned}
	\end{equation}
		\end{thm}
		\par
	\begin{proof}
		The proof is similar to \cite[Theorem 5]{J. Xu} and we omit the details here.
	\end{proof}
	\begin{thm}
		\label{thm:5.8}
 Under the assumptions of Theorem \ref{thm:5.7}. Then, there exsists $\delta_0>0$ and $\epsilon_0\in(0,1]$ such that for all $0<\delta<\delta_0$, $\epsilon\in(0,\epsilon_0]$, \eqref{eq:5.2} generates a random dynamical system $\Phi_{\delta,\epsilon}(t, \omega, u_0)$, which possesses a unique random attractor $\mathcal{A}_{\delta,\epsilon}(\omega)$. Additionally, the $\mathcal{D}_F$ -pullback absorbing set $B_{\delta,\epsilon} := \{B_{\delta,\epsilon}(\omega) : \omega\in\Omega\}$ in $H$ is given by
	\begin{equation}
		B_{\delta,\epsilon}(\omega)=\{u\in H :|u|^2\leq\lambda_1^{-1}R_{\delta,\epsilon}(\omega)\}
		\nonumber
	\end{equation}
	with
	\begin{equation}
		\begin{aligned}
			R_{\delta,\epsilon}(\omega)=&\frac{1}{m}e^{\int_{-1}^{0}2y_{\delta,\epsilon}(\theta_{s}\omega)ds+2y_{\delta,\epsilon}(\omega)}\\
			&\times\left(1+C_f|\mathcal{O}|\int_{-\infty}^{-1}e^{-2y_{\delta,\epsilon}(\theta_{s}\omega)+(m\lambda_1-3C_f)s+\int_{-1}^{0}2y_{\delta,\epsilon}(\theta_{\tau}\omega)d\tau}\right)\\
			&+\left(\frac{1}{m}C_f|\mathcal{O}|+\frac{2}{m}C_f^2|\mathcal{O}|\right)\int_{-1}^{0}e^{-2y_{\delta,\epsilon}(\theta_{s}\omega)+(m\lambda_1-3C_f)s+2y_{\delta,\epsilon}(\omega)+\int_{s}^{0}2y_{\delta,\epsilon}(\theta_{r}\omega)dr}ds.
			\nonumber
		\end{aligned}
	\end{equation}
	In addition, for almost all $\omega\in\Omega$,
	\begin{equation}
		\label{eq:5.33}
		\lim\limits_{\substack{\delta\rightarrow0^+}}R_{\delta,\epsilon}(\omega)=R_{0,\epsilon}(\omega),
	\end{equation}
	where $R_{0,\epsilon}(\omega)$ in Theorems \ref{thm:5.7}.
		\end{thm}
	\begin{proof}
		The idea to prove the existence of random $\mathcal{D}_F$ -attractor to \eqref{eq:5.2} is the same as Theorem \ref{thm:4.6}. We prove the theorem in three steps. \\
		
		\par 
		\textbf{Step 1.} We need to derive the boundedness of $q_{\delta,\epsilon}(\cdot) : = q_{\delta,\epsilon}(\cdot; t_0, \omega, q_0)$ in $H$ for all $t\in[t_0, -1]$ with $t_0\leq-1$, where $q_0=u_{\delta,\epsilon}(\tau)e^{-y_{\delta,\epsilon}(\theta_{0}\omega)}$ and exist a deterministic bounded set $D$ such that $u_{\delta,\epsilon}(\tau)\in D$. Multiplying \eqref{eq:5.29} by $q_{\delta,\epsilon}$ in $H$, thanks to \eqref{eq:2.1} and \eqref{eq:2.4}, the Young inequality and Poincar\'e inequality, we can get 
		\begin{equation}
			\begin{aligned}
				\label{eq:5.34}
				\frac{d}{dt}|q_{\delta,\epsilon}(t)|^2\leq(-m\lambda_1-3C_f+2y_{\delta,\epsilon}(\theta_{t}\omega))|q_{\delta,\epsilon}(t)|^2+e^{-y_{\delta,\epsilon}(\theta_{t}\omega)}C_f|\mathcal{O}|-m\|q_{\delta,\epsilon}(t)\|^2.
			\end{aligned}
		\end{equation}
		Neglecting the last term of \eqref{eq:5.34} and applying Gronwall inequality in $[t_0, -1]$ with $t_0\leq-1$, we have
		\begin{equation}
			\begin{aligned}
				|q_{\delta,\epsilon}(-1)|^2&\leq e^{(m\lambda_1-3C_f)(t_0+1)+\int_{t_0}^{-1}2y_{\delta,\epsilon}(\theta_{s}\omega)ds}|q_{\delta,\epsilon}(t_0)|^2\\
				&+C_f|\mathcal{O}|\int_{t_0}^{-1}e^{(m\lambda_1-3C_f)(s+1)-2y_{\delta,\epsilon}(\theta_{s}\omega)+\int_{t_0}^{-1}2y_{\delta,\epsilon}(\theta_{\tau}\omega)d\tau}ds.
				\nonumber
			\end{aligned}
		\end{equation}
		Therefore, for a given deterministic bounded set $D\subset H$, there exist a constant
		$\rho_{\delta,\epsilon}>0$ and $T(\omega, \rho_{\delta,\epsilon})\leq- 1, \Bbb P -a.e.$, such that, for any $u_0\in D\subset B(0, \rho_{\delta,\epsilon})$, for all $t_0\leq T(\omega, \rho_{\delta,\epsilon})$, we have
		\begin{equation}
			|q_{\delta,\epsilon}(-1; t_0, \omega, e^{-y_{\delta,\epsilon}(\theta_{t_0}\omega)}u_0)|^2\leq\bar{r}_{\delta,\epsilon}^2(\omega)
			\nonumber
		\end{equation}
		with
		\begin{equation}
			\bar{r}_{\delta,\epsilon}^2(\omega)=e^{(m\lambda_1-3C_f)}\left(1+C_f|\mathcal{O}|\int_{t_0}^{-1}e^{(m\lambda_1-3C_f)(s+1)-2y_{\delta,\epsilon}(\theta_{s}\omega)+\int_{t_0}^{-1}2y_{\delta,\epsilon}(\theta_{\tau}\omega)d\tau}ds\right).
			\nonumber
		\end{equation}
		\par 
		For $t\in[-1,0]$, we show $q\in L^{\infty}(-1,t; H)\cap L^2(-1,t;V)$ by energy estimations. Applying Gronwall inequality in $[-1, t]$ with $t\ge-1$, we have
		\begin{equation}
			\begin{aligned}
				\label{eq:5.35}
				|q_{\delta,\epsilon}(t)|^2&\leq e^{-(m\lambda_1-3C_f)(t+1)+\int_{-1}^{t}2y_{\delta,\epsilon}(\theta_{s}\omega)ds}|q_{\delta,\epsilon}(-1)|^2\\
				&+C_f|\mathcal{O}|\int_{-1}^{t}e^{-2y_{\delta,\epsilon}(\theta_{s}\omega)+(3C_f-m\lambda_1)(t-s)+\int_{s}^{t}2y_{\delta,\epsilon}(\theta_{\tau}\omega)d\tau}ds\\
				&-m\int_{-1}^{t}e^{(3C_f-m\lambda_1)(t-s)+\int_{s}^{t}2y_{\delta,\epsilon}(\theta_{\tau}\omega)d\tau}\|q_{\delta,\epsilon}(s)\|^2ds.
			\end{aligned}
		\end{equation}
		Consequently, we conclude that for a given deterministic subset
		$D\subset B(0, \rho_{\delta,\epsilon})\subset H$, there exists $T(\omega, \rho_{\delta,\epsilon})\leq- 1, \Bbb P -a.e.$, such that for all $t_0\leq T(\omega, \rho_{\delta,\epsilon})$, for all $u_0\in D$, we have
		\begin{equation}
			\begin{aligned}
				|q_{\delta,\epsilon}(t)|^2&\leq e^{-(m\lambda_1-3C_f)(t+1)+\int_{-1}^{t}2y_{\delta,\epsilon}(\theta_{s}\omega)ds}\bar{r}_{\delta,\epsilon}(\omega)^2\\
				&+C_f|\mathcal{O}|\int_{-1}^{t}e^{-2y_{\delta,\epsilon}(\theta_{s}\omega)+(3C_f-m\lambda_1)(t-s)+\int_{s}^{t}2y_{\delta,\epsilon}(\theta_{\tau}\omega)d\tau}ds
				\nonumber
			\end{aligned}
		\end{equation}
		and
		\begin{equation}
			\begin{aligned}
				\label{eq:5.36}
				\int_{-1}^{0}e^{(m\lambda_1-3C_f)s+\int_{s}^{0}2y_{\delta,\epsilon}(\theta_{\tau}\omega)d\tau}\|q_{\delta,\epsilon}(s)\|^2ds&\leq\frac{1}{m}e^{-(m\lambda_1-3C_f)+\int_{-1}^{t}2y_{\delta,\epsilon}(\theta_{s}\omega)ds}\bar{r}_{\delta,\epsilon}(\omega)^2\\
				&+\frac{C_f|\mathcal{O}|}{m}\int_{-1}^{t}e^{-2y_{\delta,\epsilon}(\theta_{s}\omega)+(m\lambda_1-3C_f)s+\int_{s}^{0}2y_{\delta,\epsilon}(\theta_{\tau}\omega)d\tau}ds.
			\end{aligned}
		\end{equation}
		\par 
		\textbf{Step 2.} For all $t\in[-1, 0]$, we need to obtain a bounded absorbing set in $V$ and by the compact embedding $V\hookrightarrow H$ ensure the existence of a compact absorbing ball in $H$. Taking the inner product \eqref{eq:5.29} with $-\Delta q_{\delta,\epsilon}(t)$, making use of \eqref{eq:2.2}, the Young inequality and the Poincar\'e inequality, we have
		\begin{equation}
			\begin{aligned}
				\label{eq:5.37}
				\frac{d}{dt}\|q_{\delta,\epsilon}(t)\|^2&\leq-m|\Delta q_{\delta,\epsilon}(t)|^2+\frac{2}{m}C_f^2e^{-2y_{\delta,\epsilon}(\theta_{t}\omega)}+\frac{2C_f^2}{m}|q_{\delta,\epsilon}(t)|^2+2y_{\delta,\epsilon}(\theta_{t}\omega)\|q_{\delta,\epsilon}(t)\|^2\\
				&\leq\left(-m\lambda_1+\frac{2C_f^2}{m\lambda_1}+2y_{\delta,\epsilon}(\theta_{t}\omega)\right)\|q_{\delta,\epsilon}(t)\|^2+\frac{2}{m}C_f^2|\mathcal{O}|e^{-2y_{\delta,\epsilon}(\theta_{t}\omega)}.
			\end{aligned}
		\end{equation}
		Applying Gronwall inequality in $[s, 0]$ with $s\in[-1,0]$, we have
		\begin{equation}
			\begin{aligned}
				\|q_{\delta,\epsilon}(0)\|^2&\leq e^{(m\lambda_1-\frac{2C_f^2}{m\lambda_1})s+\int_{s}^{0}2y_{\delta,\epsilon}(\theta_{\tau}\omega)d\tau}\|q_{\delta,\epsilon}(s)\|^2\\
				&+\frac{2}{m}C_f^2|\mathcal{O}|\int_{s}^{0}e^{-2y_{\delta,\epsilon}(\theta_{\tau}\omega)+(m\lambda_1-\frac{2C_f^2}{m\lambda_1})\tau+\int_{\tau}^{0}2y_{\delta,\epsilon}(\theta_{t}\omega)dt}d\tau.
				\nonumber
			\end{aligned}
		\end{equation}
		Integrating the above inequality again in $[-1, 0]$, we obtain
		\begin{equation}
			\begin{aligned}
				\label{eq:5.38}
				\|q_{\delta,\epsilon}(0)\|^2&\leq \int_{-1}^{0}e^{(m\lambda_1-\frac{2C_f^2}{m\lambda_1})s+\int_{s}^{0}2y_{\delta,\epsilon}(\theta_{\tau}\omega)d\tau}\|q_{\delta,\epsilon}(s)\|^2\\
				&+\frac{2}{m}C_f^2|\mathcal{O}|\int_{-1}^{0}e^{-2y_{\delta,\epsilon}(\theta_{s}\omega)+(m\lambda_1-\frac{2C_f^2}{m\lambda_1})s+\int_{s}^{0}2y_{\delta,\epsilon}(\theta_{r}\omega)dr}ds.
			\end{aligned}
		\end{equation}
		Since assumption $4C_f<m\lambda_1$, it is easy to verify that $m\lambda_1-3C_f<m\lambda_1-\frac{2C_f^2}{m\lambda_1}$. Then by (5.34), we can get
		\begin{equation}
			\begin{aligned}
				\|q_{\delta,\epsilon}(0)\|^2&\leq \frac{1}{m}e^{-(m\lambda_1-3C_f)+\int_{-1}^{0}2y_{\delta,\epsilon}(\theta_{s}\omega)ds}\bar{r}_{\delta,\epsilon}^2(\omega)\\
				&+\left(\frac{1}{m}C_f+\frac{2}{m}C_f^2|\mathcal{O}|\right)\int_{-1}^{0}e^{-2y_{\delta,\epsilon}(\theta_{s}\omega)+(m\lambda_1-3C_f)s+\int_{s}^{0}2y_{\delta,\epsilon}(\theta_{r}\omega)dr}ds,
				\nonumber
			\end{aligned}
		\end{equation}
		and
		\begin{equation}
			\begin{aligned}
				\|u_{\delta,\epsilon}(0)\|^2&=\|q_{\delta,\epsilon}(0)e^{y_{\delta,\epsilon}(\omega)}\|^2\\
				&\leq\frac{1}{m}e^{-(m\lambda_1-3C_f)+2y_{\delta,\epsilon}(\omega)+\int_{-1}^{0}2y_{\delta,\epsilon}(\theta_{s}\omega)ds}\bar{r}_{\delta,\epsilon}^2(\omega)\\
				&+\left(\frac{1}{m}C_f+\frac{2}{m}C_f^2|\mathcal{O}|\right)\int_{-1}^{0}e^{-2y_{\delta,\epsilon}(\theta_{s}\omega)+2y_{\delta,\epsilon}(\omega)+(m\lambda_1-3C_f)s+\int_{s}^{0}2y_{\delta,\epsilon}(\theta_{r}\omega)dr}ds.
				\nonumber
			\end{aligned}
		\end{equation}
		\par 
		Therefore, there exists $\tilde{r}_{\delta,\epsilon}(\omega)$ such that for a given $\rho_{\delta,\epsilon}>0$, exists $\tilde{T}(\omega, \rho_{\delta,\epsilon})\leq-1$ satisfying, for all $t_0\leq\tilde{T}(\omega, \rho_{\delta,\epsilon})$ and $u_{\delta,\epsilon}(0)\in H$ with $|u_0|\leq \rho_{\delta,\epsilon}$,
		\begin{equation}
			\|u_{\delta,\epsilon}(0; t_0, \omega, u_0)\|^2\leq\tilde{r'}_{\delta,\epsilon}^2(\omega),
			\nonumber
		\end{equation}
		where
		\begin{equation}
			\begin{aligned}
				\tilde{r'}_{\delta,\epsilon}(\omega)^2&=\frac{1}{m}\int_{-1}^{0}e^{\int_{-1}^{0}2y_{\delta,\epsilon}(\theta_{s}\omega)ds+2y_{\delta,\epsilon}(\omega)}\\
				&\times\left(1+C_f|\mathcal{O}|\int_{-\infty}^{-1}e^{-2y_{\delta,\epsilon}(\theta_{s}\omega)+(m\lambda_1-3C_f)s+\int_{s}^{-1}2y_{\delta,\epsilon}(\theta_{\tau}\omega)d\tau}ds\right)\\
				&+\left(\frac{1}{m}C_f|\mathcal{O}|+\frac{2}{m}C_f^2|\mathcal{O}|\right)\int_{-1}^{0}e^{-2y_{\delta,\epsilon}(\theta_{s}\omega)+(m\lambda_1-3C_f)s+2y_{\delta,\epsilon}(\omega)+\int_{s}^{0}2y_{\delta,\epsilon}(\theta_{r}\omega)dr}ds.
				\nonumber
			\end{aligned}
		\end{equation}
		It follows from that (5.24), for a given $\epsilon_0'=\frac{m\lambda_1-3C_f}{8}$, there exists $T_1(\epsilon_0',\omega)<0$, for all $t\leq T_1$, we have
		\begin{equation}
			\label{eq:5.39}
			|y_{\delta,\epsilon}(\theta_{t}\omega)|\leq-\frac{m\lambda_1-3C_f}{8}t.
		\end{equation}
		From \eqref{eq:5.26}, for any $\epsilon>0$, there exists $T_2(\epsilon,\omega)<0$, for all $t\leq T_2$,
		\begin{equation}
			\label{eq:5.40}
			\left|\int_{0}^{t}y_{\delta,\epsilon}(\theta_{\tau}\omega)d\tau\right|\leq-\frac{m\lambda_1-3C_f}{8}t.
		\end{equation}
		Let $T_0=min\{T_1, T_2\}$, we have
		\begin{equation}
			\begin{aligned}
				\int_{-\infty}^{-1}&e^{-2y_{\delta,\epsilon}(\theta_{s}\omega)+(m\lambda_1-3C_f)s+\int_{s}^{-1}2y_{\delta,\epsilon}(\theta_{\tau}\omega)d\tau}ds\\
				&=\int_{-\infty}^{T_0}e^{-2y_{\delta,\epsilon}(\theta_{s}\omega)+(m\lambda_1-3C_f)s+\int_{s}^{-1}2y_{\delta,\epsilon}(\theta_{\tau}\omega)d\tau}ds\\
				&+\int_{T_0}^{-1}e^{-2y_{\delta,\epsilon}(\theta_{s}\omega)+(m\lambda_1-3C_f)s+\int_{s}^{-1}2y_{\delta,\epsilon}(\theta_{\tau}\omega)d\tau}ds\\
				&=Q_1+Q_2.
				\nonumber
			\end{aligned}
		\end{equation}
		Since $y_{\delta,\epsilon}(\omega)$ is continuous that $Q_2$ is bounded. Next, we show that $Q_1$ is bounded, according to \eqref{eq:5.39} and \eqref{eq:5.40} we have
		\begin{equation}
			\begin{aligned}
				\label{eq:5.41}
				Q_1&=\int_{-\infty}^{T_0}e^{-2y_{\delta,\epsilon}(\theta_{s}\omega)+(m\lambda_1-3C_f)s+\int_{s}^{-1}2y_{\delta,\epsilon}(\theta_{\tau}\omega)d\tau}ds\\
				&\leq\int_{-\infty}^{T_0}e^{2|y_{\delta,\epsilon}(\theta_{s}\omega)|+(m\lambda_1-3C_f)s+|\int_{s}^{-1}2y_{\delta,\epsilon}(\theta_{\tau}\omega)d\tau}|ds\\
				&\leq\int_{-\infty}^{T_0}e^{(m\lambda_1-3C_f)(s+\frac{1}{4})}ds<\infty.
			\end{aligned}
		\end{equation}
		Thus, it follows from [33, Theorem 2] that there exists a unique random attractor
		$\mathcal{A}_{\delta,\epsilon}(\omega)$ to equation \eqref{eq:5.2}.
		\par 
		\textbf{Step 3.} We show \eqref{eq:5.33} holds. Since the properties of $y_{\delta,\epsilon}(\theta_{t}\omega)$ (cf. (\eqref{eq:5.26}-\eqref{eq:5.27} and Lemma \ref{lem:5.6}, the same idea as in Theorem \ref{thm:4.6} to proof this result, so we omit the details.
	\end{proof}
	\begin{lem}
		\label{lem:5.9}
Assume the conditions of Theorem \ref{thm:5.7} hold. Let $\{\delta_n\}_{n=1}^{\infty}$ be a sequence so that $\delta_n\rightarrow0^+$ as $n \rightarrow+\infty$, and exists $\epsilon_0\in (0,1]$ such that for all $\epsilon\in (0,\epsilon_0]$. Let $q_{\delta_n,\epsilon}$ and $q_{0,\epsilon}$ be the solutions of equations \eqref{eq:5.2}
	and \eqref{eq:5.1} with initial data $p_{\delta,\epsilon}(\tau)$ and $q_{0,\epsilon}(\tau)$, respectively. If $q_{\delta_n,\epsilon}(\tau)\rightarrow q_{0,\epsilon}(\tau)$ weakly in $H$ as $n\rightarrow+\infty$, then for almost all $\omega\in\Omega$,
	\begin{equation}
		\label{eq:5.42}
		q_{\delta_n,\epsilon}(r; \tau, \omega, q_{\delta_n,\epsilon}(\tau))\rightarrow q_{0,\epsilon}(r; \tau, \omega, q_{0,\epsilon}(\tau))\quad weakly\ in\ \quad H\quad\forall r\ge \tau,
	\end{equation}
	and
	\begin{equation}
		\label{eq:5.43}
		q_{\delta_n,\epsilon}(\cdot; \tau, \omega, q_{\delta_n,\epsilon}(\tau))\rightarrow q_{0,\epsilon}(\cdot; \tau, \omega, q_{0,\epsilon}(\tau))\quad strongly\ in\ \quad L^2(\tau,\tau+T;H)\quad\forall T>0.
	\end{equation}
		\end{lem}
	\begin{proof}
		This proof is similar to \cite[Lemma 3.5]{A.G B}, thus it is omitted here.
	\end{proof}
	\begin{lem}
		\label{lem:5.10}
 Assume the conditions of Theorem \ref{thm:5.7} hold and function $a$ is globally Lipschitz. Let $\{\delta_n\}_{n=1}^{\infty}$ be a sequence so that $\delta_n\rightarrow 0^+$ as $n\rightarrow {+\infty}$, and exist $\epsilon_0\in (0,1]$ such that for all $\epsilon\in (0,\epsilon_0]$. Let $q_{\delta_n,\epsilon}$ and $q_{0,\epsilon}$ be the solutions of equations \eqref{eq:5.2} and \eqref{eq:5.1} with initial data $q_{\delta_n,\epsilon}(\tau)$ and $q_{0,\epsilon}(\tau)$, respectively. If $q_{\delta_n,\epsilon}(\tau)\rightarrow q_{0,\epsilon}(\tau) $ in $H$ as $n\rightarrow+\infty$, then for every $\tau\in R$, $\omega\in\Omega$ and $t\geq\tau$,
	\begin{equation}
		q_{\delta_n,\epsilon}(t; \tau, \omega, q_{\delta_n,\epsilon}(\tau))\rightarrow q_{0,\epsilon}(t;\tau,\omega, q_{0,\epsilon}(\tau))\quad in\ \quad H\quad \forall t\ge\tau,
		\nonumber
	\end{equation}
		\end{lem}
		\begin{proof}
	 The proof is similar to \cite[Lemma 3.8]{A. Gu}, thus the details are omitted here.
	\end{proof}
	\begin{lem}
		\label{lem:5.11}
 Assume the conditions of Lemma \ref{lem:5.10} hold. Letting $\omega\in\Omega$ is fixed, and letting $\epsilon_0\in (0,1]$ such that for all $\epsilon\in (0,\epsilon_0]$. If
	$\delta_n\rightarrow0^+$ as $n\rightarrow+\infty$ and $u_n :=u_{\delta,\epsilon,n}\in\mathcal{A}_{\delta n,\epsilon}(\omega)$, then the sequence $\{u_n\}_{n=1}^{\infty}$ has a convergent
	subsequence in $H$.
		\end{lem}
	\begin{proof}
		Since $u_n\in\mathcal{A}_{\delta_n,\epsilon}(\omega)$, according to the the invariance of $\mathcal{A}_{\delta_n,\epsilon}$ that there exists $u_{n,-1} :=u_{\delta,\epsilon,n,-1}\in\mathcal{A}_{\delta_{n},\epsilon}(\theta_{-1}\omega)$, we have 
		\begin{equation}
			\label{eq:5.44}
			u_n=\Phi_{\delta,\epsilon}(1, \theta_{-1}\omega, u_{n,-1})=u_{\delta_n,\epsilon}(0;-1,\omega,u_{n,-1}).
		\end{equation}
		According to Theorem \ref{thm:5.8} that there exists $N_1=N_1(\omega)\geq1$ that for all n $\geq N_1$,
		\begin{equation}
			\begin{aligned}
				R_{\delta_n,\epsilon}(\theta_{-1}\omega)&\leq1+\frac{1}{m}e^{\int_{-1}^{0}2y_{\delta_n,\epsilon}(\theta_{-1}\omega)ds+2y_{\delta_n,\epsilon}(\theta_{-1}\omega)}\\
				&\times\left(1+C_f|\mathcal{O}|\int_{-\infty}^{-1}e^{-2y_{\delta_n,\epsilon}(\theta_{s-1}\omega)+(m\lambda_1-3C_f)s+\int_{s}^{-1}2y_{\delta_n,\epsilon}(\theta_{\tau-1}\omega)d\tau}ds\right)\\
				&+\left(\frac{1}{m}C_f|\mathcal{O}|+\frac{2}{m}C_f^2|\mathcal{O}|\right)\int_{-1}^{0}e^{-2y_{\delta,\epsilon}(\theta_{s}\omega)+(m\lambda_1-3C_f)s+2y_{\delta,\epsilon}(\omega)+\int_{s}^{0}2y_{\delta,\epsilon}(\theta_{r}\omega)dr}ds.
				\nonumber
			\end{aligned}
		\end{equation}
		Since $u_{n,-1}\in\mathcal{A}_{\delta_n,\epsilon}(\theta_{-1}\omega)\subset B_{\delta_n,\epsilon}(\theta_{-1}\omega)$, by Theorem \ref{thm:5.8} and \eqref{eq:5.44}, for all $n\geq N_1$, we have
		\begin{equation}
			\begin{aligned}
				\label{eq:5.45}
				|u_{n,-1}|^2&\leq\lambda_1^{-1}(1+\frac{1}{m}e^{\int_{-1}^{0}2y_{\delta_n,\epsilon}(\theta_{-1}\omega)ds+2y_{\delta_n,\epsilon}(\theta_{-1}\omega)}\\
				&\times\left(1+C_f|\mathcal{O}|\int_{-\infty}^{-1}e^{-2y_{\delta_n,\epsilon}(\theta_{s-1}\omega)+(m\lambda_1-3C_f)s+\int_{s}^{-1}2y_{\delta_n,\epsilon}(\theta_{\tau-1}\omega)d\tau}ds\right)\\
				&+\left(\frac{1}{m}C_f|\mathcal{O}|+\frac{2}{m}C_f^2|\mathcal{O}|\right)\int_{-1}^{0}e^{-2y_{\delta,\epsilon}(\theta_{s}\omega)+(m\lambda_1-3C_f)s+2y_{\delta,\epsilon}(\omega)+\int_{s}^{0}2y_{\delta,\epsilon}(\theta_{r}\omega)dr}ds).
			\end{aligned}
		\end{equation}
		By \eqref{eq:5.28} and Lemma \ref{lem:5.6} we obtain that
		\begin{equation}
			q_{\delta_n,\epsilon}(s;-1,\omega,q_{n,-1})=u_{\delta_n,\epsilon}(s;-1,\omega,u_{n,-1})e^{-y_{\delta_n,\epsilon}(\theta_{s}\omega)}
			\nonumber
		\end{equation}
		and
		\begin{equation}
			\label{eq:5.46}
			\lim\limits_{\substack{\delta_n\rightarrow0^+}}e^{-y_{\delta_n,\epsilon}(\theta_{-1}\omega)}=e^{-y_{0,\epsilon}(\theta_{-1}\omega)},
		\end{equation}
		where $q_{n,-1}=u_{n,-1}e^{-y_{\delta_n,\epsilon}(\theta_{-1}\omega)}$.
		\par 
		Combining with \eqref{eq:5.45}-\eqref{eq:5.46}, we obtain that the sequence $\{q_{n,-1}\}_{n=1}^{\infty}$ is bounded in $H$. Then there exist a  subsequence $\{ q_{n,-1}\}$ and $q_1$, that $q_{n,-1}\rightarrow q_{-1}$ weakly in $H$. By Lemma \ref{lem:5.10} ensures the existence of $\bar{q}_{0,\epsilon} := \bar{q}_{0,\epsilon}(\cdot; -1, \omega, q_{-1})\in L^2(-1, 0; H)$, such that, up to a subsequence,
		\begin{equation}
			q_{\delta_n,\epsilon}(\cdot;-1,\omega,q_{n,-1})\rightarrow\bar{q}_{0,\epsilon}\quad strongly\ in\ \quad L^2(-1,0;H),
			\nonumber
		\end{equation}
		such that, up to a further subsequence,
		\begin{equation}
			\label{eq:5.47}
			q_{\delta_n,\epsilon}(s;-1,\omega,q_{n,-1})\rightarrow\bar{q}_{0,\epsilon}(s)\quad strongly\ in\ \quad H, \quad a.e.\quad s\in(-1,0).
		\end{equation}
		It follows from Lemma \ref{lem:5.6}, \eqref{eq:5.46}-\eqref{eq:5.47}, we have
		\begin{equation}
			\label{eq:5.48}
			u_{\delta_n,\epsilon}(s;-1,\omega,u_{n,-1})\rightarrow e^{y_{\delta,\epsilon}(\theta_{s}\omega)}\bar{q}_{0,\epsilon}(s)\quad strongly\ in\ \quad H, \quad a.e.\quad s\in(-1,0).
		\end{equation}
		By Lemma \ref{lem:5.10} and \eqref{eq:5.48}, that $\delta_n\rightarrow0^+$ as $n\rightarrow+\infty$, we have
		\begin{equation}
			\label{eq:5.49}
			u_{\delta_n,\epsilon}(0;s,\omega,u_{n,-1})\rightarrow u(0;s,\omega,e^{y_{\delta,\epsilon}(\theta_{s}\omega)}\bar{q}_{0,\epsilon}(s)) \quad strongly\ in\ \quad H,
		\end{equation}
		where $u$ is the solution of \eqref{eq:5.2}, making using of cocycle property,
		\begin{equation}
			u_{\delta_n,\epsilon}(0;s,\omega,u_{\delta_n,\epsilon}(s;-1,\omega,u_{n,-1}))=u_{\delta_n,\epsilon}(0;-1,\omega,u_{n,-1})
			\nonumber
		\end{equation}
		By \eqref{eq:5.49}, we can get that
		\begin{equation}
			\label{eq:5.50}
			u_{\delta_n,\epsilon}(0;-1,\omega,u_{n,-1})\rightarrow u(0;s,\omega,e^{y_{\delta,\epsilon}(\theta_{s}\omega)}\bar{q}_{0,\epsilon}(s)) \quad strongly\ in\ \quad H,
		\end{equation}
		together with \eqref{eq:5.44}, the proof is finished.
	\end{proof}
		\textbf{ 5.4. Upper semi-continuity of random attractor.} In this section, we establish the upper semi-continuity of random attractor when small random perturbations $\delta$ and $\epsilon$ approach zero. Let $(H, \|\cdot\|_H)$ be a Banach space and $\Psi$ be an autonomous dynamical
	system defined on $H$. 
	
	Consider a deterministic equation on H:
	\begin{equation}
		\label{eq:5.51}
		\frac{\partial u}{\partial t}-a(l(u))\Delta u_=f(u),
	\end{equation}
	with the initial condition:
	\begin{equation}
		\label{eq:5.52}
		u(\tau)=u_0\in H.
	\end{equation}
	\par 
	In a similar way as \cite{Caraballo THerrera-Cobos M Marín-Rubio P}, we are able to prove that problem $\eqref{eq:5.51}-\eqref{eq:5.52}$ generates a continuous deterministic dynamical system $\Psi$ in $H$, and $\Psi$ has a unique global attractor $\mathcal{A}$ in $H$. According to Theorem \ref{thm:5.7}, given $\epsilon\in(1,0]$, we have that $\Psi_{0,\epsilon}$ has a unique $\mathcal{D}$-pullback random attractor $\mathcal{A}_{0,\epsilon}(\omega)$, and is a random dynamical system such that for P-almost every $\omega\in\Omega$ and all $t\in R^+$, 
	\begin{equation}
		\Psi_{0,\epsilon}(t,\theta_{-t}\omega)x\rightarrow\Psi(t)x\quad as\ \epsilon\rightarrow0^+.
		\nonumber
	\end{equation}
	uniformly on bounded sets of $H$.
	\par 
	\begin{thm}\label{thm:5.12} Assume function $a$ is globally Lipschitz and satisfies \eqref{eq:2.1}, $f\in C(\Bbb R)$ satisfies \eqref{eq:2.2} and \eqref{eq:2.4} with $p=2$ and $\beta=C_f$, respectively, $\phi\in V\cap H^2(\mathcal{O})$, and $l\in L^2(\mathcal{O})$. Also, let $m\lambda_1>4C_f$. There exists $\epsilon_0\in (0,1]$ such that for all $\epsilon\in(0,\epsilon_0]$, then for almost all $\omega\in\Omega$,
		\begin{equation}
			\label{eq:5.54}
			\lim\limits_{\substack{\epsilon\rightarrow0^+}} dist_H(\mathcal{A}_{0,\epsilon}(\omega),\mathcal{A})=0,
		\end{equation}
		where
		\begin{equation}
			\lim\limits_{\substack{\epsilon\rightarrow0^+}} dist_H(\mathcal{A}_{0,\epsilon}(\omega),\mathcal{A})=\sup_{a\in\mathcal{A}_{0,\epsilon}(\omega)}\inf_{b\in\mathcal{A}}\|a-b\|_{H}.
			\nonumber
		\end{equation}
	\end{thm}
	\begin{proof}
		The proof is similar to \cite[Theorem 2]{T.CaJ.A.} and we omit the detail here.
	\end{proof}
	\begin{thm}
		\label{thm:5.13}
 Assume that function $a\in C(\Bbb R; \Bbb R^+)$ fulfills
	\eqref{eq:2.1}, function $f$ satisfies \eqref{eq:2.2} and \eqref{eq:2.4} with $p=2$ and $\beta=C_f$, respectively. Also, let $m\lambda_1>4C_f$ and $l\in L^2(\mathcal{O})$. Then, there exist $\epsilon_0\in (0,1]$ such that for all $\epsilon\in (0,\epsilon_0]$, for almost all $\omega\in\Omega$,
	\begin{equation}
		\lim\limits_{\substack{\delta\rightarrow0^+}}dist_H(\mathcal{A}_{\delta,\epsilon}(\omega),\mathcal{A}_{0,\epsilon}(\omega))=0.
	\end{equation}
	\begin{proof}
		By Theorem \ref{thm:5.7} and \ref{thm:5.8}, for $\omega\in\Omega$, we have,
		\begin{equation}
			\lim\limits_{\substack{\delta\rightarrow0^+}}|B_{\delta,\epsilon}(0,\omega)|=|B(0,\omega)|.
		\end{equation}
		Together with the above equality, Lemma \ref{lem:5.9} and \ref{lem:5.11}, the proof is finished by \cite[Theorem 3.1]{B.W E}.
	\end{proof}
		\end{thm}
	\par 
	Based on the above analysis, when perturbation parameters $\delta$
	and $\epsilon$ both tend to zero, the convergence relationship between the random attractor of the stochastic equation \eqref{eq:5.2} driven by linear multiplicative noise and the global attractor of the corresponding deterministic equation \eqref{eq:4.2} can be obtained as follows:
	\par
	\begin{thm}
		\label{thm:5.14}
 Assume function $a$ is globally Lipschitz and satisfies \eqref{eq:2.1}, $f\in C(\Bbb R)$ satisfies \eqref{eq:2.2} and \eqref{eq:2.4} with $p=2$ and $\beta=C_f$, respectively. Also, let $m\lambda_1>4C_f$ and $l\in L^2(\mathcal{O})$. Then for almost all $\omega\in\Omega$, we have 
	\begin{equation}
		\label{eq:5.55}
		\lim\limits_{\substack{\delta\rightarrow0^+\\\epsilon\rightarrow0^+}}dist_H(\mathcal{A}_{\delta,\epsilon}(\omega),\mathcal{A})=0,
	\end{equation}
	where
	\begin{equation}
		\lim\limits_{\substack{\delta\rightarrow0^+\\\epsilon\rightarrow0^+}} dist_H(\mathcal{A}_{\delta,\epsilon}(\omega),\mathcal{A})=\sup_{a\in\mathcal{A}_{\delta,\epsilon}(\omega)}\inf_{b\in\mathcal{A}}\|a-b\|_{H}.
		\nonumber
	\end{equation}
		\end{thm}
	\begin{proof}
		By Theorem \ref{thm:5.12} and \ref{thm:5.13}, there exists $\epsilon_0\in (0,1]$, $\tilde{\delta} >0$ such that for all $\epsilon\in(0,\epsilon_0]$, $0<\delta<\tilde{\delta}$, we have
		\begin{equation}
			\begin{aligned}
				dist_H(\mathcal{A}_{\delta,\epsilon}(\omega),\mathcal{A})&=\sup_{a\in\mathcal{A}_{\delta,\epsilon}(\omega)}\inf_{b\in\mathcal{A}}\|a-b\|_{H}\\
				&\leq\sup_{a\in\mathcal{A}_{\delta,\epsilon}(\omega)}\inf_{c\in\mathcal{A}_{0,\epsilon}(\omega)}\|a-c\|_{H}+\sup_{c\in\mathcal{A}_{0,\epsilon}(\omega)}\inf_{b\in\mathcal{A}}\|c-b\|_{H}.
				\nonumber
			\end{aligned}
		\end{equation}
		The proof is completed.
	\end{proof}
	
\end{document}